\documentclass[11pt]{article} 

\usepackage[english]{babel}
\usepackage{graphicx}
\usepackage{float}
\usepackage{amsfonts}
\usepackage{amsmath}
\usepackage{color}

\usepackage{amsthm}

\advance\oddsidemargin by -1in
\advance\evensidemargin by -1in
\def\beq{\begin{equation}}
\def\eeq{\end{equation}}
\def\beqn{\begin{equation*}}
\def\eeqn{\end{equation*}}
\def\bearn{\begin{eqnarray*}}
\def\eearn{\end{eqnarray*}}
\def\bear{\begin{eqnarray}}
\def\eear{\end{eqnarray}}
\def\barr{\begin{array}}
\def\earr{\end{array}}

\newtheorem{lemma}{Lemma}[section]
\newtheorem{proposition}[lemma]{Proposition}

\newtheorem{corollary}{Corollary}[section]

\usepackage{comment}
\usepackage{tcolorbox}
\usepackage{bbm}

\usepackage{booktabs} 
\allowdisplaybreaks

\usepackage[normalem]{ ulem }
\usepackage{soul}
\usepackage{amssymb}
\usepackage{bbm}

\usepackage{ulem}
\newtheorem{thm}{Theorem}

\theoremstyle{remark}
\newtheorem{remark}{Remark}

\renewcommand{\P}{\mathbb{P}}

\newcommand{\N}{\mathbb{N}}
\newcommand{\E}{\mathbb{E}}
\newcommand{\R}{\mathbb{R}}

\newcommand{\Prob}{\mathbb{P}}

\newcommand{\ind}[1]{\mathbbm{1}\{#1\}}

\def\mcO{{\mathcal O}}

\def\hg{{\hat{g}}}

\allowdisplaybreaks
\usepackage{hyperref}
\hypersetup{
    colorlinks=true,
    linkcolor=blue,
    filecolor=gray,      
    urlcolor=blue,
    citecolor=blue,
}

\usepackage{prettyref}
\newrefformat{sec}{Section~\ref{#1}}
\newrefformat{subsec}{Section~\ref{#1}}
\newrefformat{subsubsec}{Section~\ref{#1}}
\newrefformat{thm}{Theorem~\ref{#1}}
\newrefformat{theorem}{Theorem~\ref{#1}}
\newrefformat{THM}{THEOREM~\ref{#1}}
\newrefformat{lemma}{Lemma~\ref{#1}}
\newrefformat{lem}{Lemma~\ref{#1}}
\newrefformat{LEMMA}{LEMMA~\ref{#1}}
\newrefformat{proposition}{Proposition~\ref{#1}}
\newrefformat{prop}{Proposition~\ref{#1}}
\newrefformat{PROP}{PROPOSITION~\ref{#1}}
\newrefformat{corollary}{Corollary~\ref{#1}}
\newrefformat{fig}{Figure~\ref{#1}}
\newrefformat{tab}{Table~\ref{#1}}
\newrefformat{eq}{\eqref{#1}}
\newrefformat{app}{Appendix~\ref{#1}}
\newrefformat{ex}{Example~\ref{#1}}
\newrefformat{result}{Result~\ref{#1}}
\newrefformat{rmk}{Remark~\ref{#1}}
\newrefformat{remark}{Remark~\ref{#1}}

\usepackage{hyperref}

\usepackage{authblk}

\title{Covert Cycle Stealing in a Single FIFO Server (extended version)}

\date{\today}

\author[,1]{Bo Jiang  \thanks{\texttt{bjiang@sjtu.edu.cn}}}
\author[,2]{Philippe Nain \thanks{\texttt{philippe.nain@inria.fr}}}
\author[,3]{Don Towsley \thanks{\texttt{towsley@cs.umass.edu. The research of this author was supported by the US NSF grant CNS-1564067 and the LABEX MILYON (ANR-10-LABX-0070) of Université de Lyon, within the program "Investissements d'Avenir" (ANR-11-IDEX- 0007) operated by the French National Research Agency (ANR). }}}
\affil[1]{John Hopcroft Center for Computer Science, Shanghai Jiao Tong University, Shanghai, China}
\affil[2]{Inria, Ecole Normale Sup\'erieure de Lyon, LIP, 46 all\'ee d'Italie, Lyon, 69364, France}
\affil[3]{College of Information and Computer Sciences, University of Massachusetts, Amherst, MA, 01003, USA. }
 
\begin{document}
	\maketitle

\begin{abstract}

Consider a setting where Willie generates a Poisson stream of jobs and routes them to a single server that follows the first-in first-out discipline.
Suppose there is an adversary Alice, who desires to receive service without being detected.  We ask the question: what is the number of jobs that she can receive covertly, i.e. without being detected by Willie?  In the case where both Willie and Alice jobs have exponential service times with respective rates $\mu_1$ and $\mu_2$, 
we demonstrate a phase-transition when Alice adopts the strategy of inserting a single job probabilistically when the server idles :  
over $n$ busy periods, she can achieve a covert throughput, measured by the expected number of jobs  covertly inserted, of $\mcO(\sqrt{n})$  when $\mu_1 < 2\mu_2$, $\mcO(\sqrt{n/\log n})$  when $\mu_1 = 2\mu_2$, and  $\mcO(n^{\mu_2/\mu_1})$  when $\mu_1 > 2\mu_2$.
When both Willie and Alice jobs have general service times we establish an upper bound for the number of jobs Alice can execute covertly.
This bound is related to the Fisher information. More general insertion policies are also discussed.

\end{abstract}

{\bf keywords: Cycle stealing; Covert communication; Queue.}

\section{Introduction}
\label{sec:intro}

This paper considers the following problem.  Willie has a sequence of jobs that arrive at a first-in first-out (FIFO) queue with a single server, whose processing rate is known to Willie.  There exists another actor, Alice, who wants to sneak jobs into the queue for the purpose of stealing processing cycles from Willie.  This paper asks the following question: can Alice process her jobs without Willie being able to determine this occurrence beyond making a random guess and, if she can, what is her achievable job processing rate?  Answers to this question may apply to several scenarios.  Alice could administer a data center, contract to provide Willie with a server with a guaranteed performance, and then resell some of the processing cycles \cite{EFNBD12}.  Similar considerations apply to network contracts. Willie could own a home computer and Alice could install malware for the purpose of stealing computational resources.  

In order to address this question of {\em covert cycle stealing}, we adopt the following model.  Willie's jobs arrive according to a Poisson process to a FIFO queue served by a single server with a specified processing rate.  Service times of Willie's jobs are assumed to be independent and identically distributed (iid)  according to a general distribution.  Alice can insert jobs as she wishes.  Her service times are also iid coming from a general distribution that may differ from that of Willie's. Once an Alice job starts service, it must remain in service until completion; this can interfere with the processing of Willie's jobs.  Last, both Willie and Alice know their own and the other party's service time distributions and can observe the arrival and departure times of Willie's jobs.  

We formulate the problem as a statistical hypothesis testing problem where Willie's task is to determine whether or not Alice is stealing cycles, based on observed arrivals and departures.  We study the {\em Insert-at-End-of-Busy-Period} (IEBP) policy, where Alice (probabilistically) inserts a single job each time a Willie busy period (to be defined) ends.
We obtain several results, of which the most interesting hold for exponential services for both Willie (rate $\mu_1$) and
Alice (rate $\mu_2$) and establish that over $n$ busy periods, Alice can achieve a covert throughput -- defined as the expected number of covertly inserted jobs --  of 
$\mathcal{O }(\sqrt{n})$ when $\mu_1 < 2\mu_2$, $\mathcal{O} (\sqrt{n/\log n})$ when $\mu_1 = 2\mu_2$, and  $\mathcal{O} (n^{\mu_2/\mu_1})$ when $\mu_1 > 2\mu_2$.  
This is interesting in part because of the phase transition at $\mu_1 =2\mu_2$; earlier studies of covert communications and in steganography focused on establishing $O (\sqrt{n})$ behavior through the control of Alice's parameters, avoiding regions in the parameter space where this behavior  might not hold.

In addition to the above results for the IEBP policy when service times are exponentially distributed, we show that IEBP can also achieve a covert throughput of $\mathcal{O}(\sqrt{n})$ when Willie jobs have general service times and Alice jobs have (hyper-)exponential service times, under some constraints on the service rates.

The rest of the paper is organized as follows.  Section \ref{sec:related} discusses related work. Section \ref{sec:model} introduces the model and needed background on hypothesis testing.  Section \ref{sec:IEBP} introduces the IEBP policy.  Section \ref{sec:main-results} lists the main results, and some preliminary results are established in Section \ref{sec:preliminary-results}. 
Sections \ref{sec:achievability-IEBP-proof}-\ref{sec:proof-phase-transition-exponential} contain the proofs of the main results.
Section \ref{sec:II} focuses on the II policy, a variant of  the  IEBP  policy.   Section  \ref{sec:IIA} introduces  the  II-A  policy  and  derives  an  asymptotic  upper  bound  on Alice's capacity.  The section also derives tighter asymptotic upper bounds for the II-A policy when Alice introduces batches whose sizes are geometrically distributed.  
Concluding remarks are given in Section \ref{sec:conclusion}

A word on the notations. For any $a\in [0,1]$, let $\bar a:=1-a$. We denote the convolution of $f$ and $g$ by $f*g$ and the $n$-fold tensor product of $f$ with itself is denoted by $f^{\otimes n}$; recall that $f^{\otimes n}(x_1, \dots, x_n) = \prod_{i=1}^n f(x_i)$ with $x_i\in \R^d$ ($d\geq 1$). Throughout we use the shorthand notations $t_{i:j}$ for $t_i,\ldots,t_j$ for $i<j$, $a_n \sim_n b_n $ for $\lim_{n\to\infty} a_n / b_n = 1$, and $\lim_n a_n$ (resp. $\liminf_n a_n$, $\limsup_n a_n$) for $\lim_{n\to\infty} a_n$ (resp. $\liminf_{n\to\infty} a_n$, $\limsup_{n\to\infty} a_n$).


\section{Related work}
\label{sec:related}

Cycle stealing has been analyzed in the queueing literature in the context of task assignment in multi-server systems. The goal is to allow servers to borrow cycles from other servers while they are idle so as to reduce backlogs and latencies and prevent servers from being under-utilized \cite{Mor-2003,Mor-Sig-2003,Mor-2005}.
These papers focus on the performance analysis of such systems, in particular, mean response times with or without the presence of switching costs. There is no attempt to hide or cover up the theft of cycles.

This paper focuses on the ability of an unknown user to steal cycles without the owner of the server detecting this. Thus it is an instance of a much broader set of techniques used in digital steganography and covert communications.  Steganography is the discipline of hiding data in objects such as digital images.  A steganographic system modifies fixed-size finite-alphabet covertext objects into stegotext containing
hidden information. A fundamental result of steganography is the {\em square root law (SRL)}, $\mathcal{O}(\sqrt{n})$ symbols of an $n$ symbol covertext may safely be altered to hide an $\mathcal{O}(\sqrt{n}\log n)$-bit message \cite{Fridrich-2009}.
Covert communications is concerned with the transfer of information in a way that cannot be detected, even by an optimal detector. Here, there exists a similar SRL: suppose Alice may want to communicate to Bob in the face of a third party, Willie, without being detected by Willie.  When communication takes place over a channel characterized by additive white Gaussian noise, it has been established that Alice can transmit $O(\sqrt{n})$ bits of information in $n$ channel uses \cite{BGT12}.  This result has been extended to optical (Poisson noise) channels \cite{Boson-2015}, binary channels \cite{CBCJ13}, and many others \cite{Bloch-2016, Wang-2016}.  It has also been extended to include the presence of jammers \cite{Sobers-2017}, and to network settings \cite{Sheik-2018}.  Like our work, both steganography and covert communications rely on the use of statistical hypothesis testing.  One difference from covert communications is that in our setting Alice hides her jobs in exponentially distributed noise (Willie's service times), that is exponentially distributed  Poisson processes (exponentially distributions) and sometimes more general distributions, whereas it is a zero mean Gaussian noise in covert communications.  In addition, in the communications context, Alice has control over the power that she transmits at whereas in our context, Alice does not control the size of her jobs, only the rate at which they are introduced.  

This work also has ties to the detection of service level agreement (SLA) violation problem. Detecting SLA violation in today's complex computing infrastructures, such as clouds infrastructures, presents challenging research issues \cite{EFNBD12}. However no careful analysis of this problem has been conducted.  Our work may provide an avenue to doing such.

During the review process a paper related to ours appeared \cite{YB2020}. In \cite{YB2020}  Alice's jobs arrive continuously according to a Poisson process with rate $\lambda_b$.  It is shown that if Willie knows the number of his jobs successively served in a busy period and $\lambda_b$ lies below a certain threshold, then the expected number of jobs that Alice can covertly insert over $n$ busy periods is $\mathcal{O}(\sqrt{n})$. If instead Willie knows the length of each busy period instead, the expected number of jobs that Alice can  covertly insert is only $\mathcal{O}(1)$. 


\section{Model and Background}
\label{sec:model}

This section gives details about the model we use in the present work and the needed background on hypothesis testing.
As mentioned in the introduction, there is a legitimate user, Willie, who sends a sequence of jobs to a single server with known service rate.  There is also an illegitimate user, Alice, who wants to introduce a sequence of jobs to be serviced.  
The questions that we address are the following: can Alice covertly introduce her stream of jobs, i.e. without Willie being able to tell with confidence whether she has introduced the stream or not, and if so, at what rate can she introduce her jobs?  We answer these questions under the following assumptions:
\begin{enumerate}
\item Willie  jobs arrive at the server according to a Poisson process with rate $\lambda\in (0,\infty)$;
\item the service times of all jobs are independent;
\item the service time distributions are known to both parties;
\item the server serves all jobs in a FIFO manner; 
\item once in service, Alice  jobs cannot be preempted;
\item Willie observes only his arrivals and departures;
\item Alice observes Willie's and her own arrivals and departures.
\end{enumerate}
The first four assumptions are made mainly for tractability. If Alice jobs can be preempted whenever a Willie job arrives, then  Alice can hide her jobs during Willie's idle periods
without affecting his jobs. Consequently, we make the fifth assumption to make the problem interesting. Note that allowing Alice to also observe Willie arrivals and departures gives her the capability to identify idle periods within which to hide her jobs.

Assumption 6 implies that Willie does not know the state of the server. If he does then Alice can only transmit during busy periods. However, if the scheduling policy is FIFO
(Assumption 4)  Alice cannot know if an inserted job of hers will not be the last one of the busy period, in which case it will be detected by Willie; relaxing the FIFO assumption
appears to be very challenging.

Assume that the system is empty at time $0$.
Denote by $A_i$ and $D_i$ the arrival and departure times of Willie's  $i$-th job, respectively, for $i\geq 1$.  We assume that $D_0=0$.
Note that $0<A_i < A_{i+1}$ and $0 < D_i < D_{i+1}$. Let $A_{1:m} = \{A_1, \dots, A_m\}$, $D_{1:m}=\{D_1, \dots, D_m\}$. Let $S_{1:m}=(S_1, \ldots, S_m)$ denote the
{\em reconstructed service times} of the first $m$ jobs, which satisfy the following recurrence relation,
\begin{equation}
\label{reconstructed-services}
S_i = 
\begin{cases}
D_1 - A_1, & i=1, \\
D_i-\max\{A_i,D_{i-1}\}, & i\ge 2.
\end{cases}
\end{equation}
These are the service times perceived by Willie. Note that $(A_{1:m}, S_{1:m})$ and $(A_{1:m},D_{1:m})$ contain the same information, as they uniquely determine each other. It is also all the information available to Willie in our model. 

We define a {\em Willie Busy Period} (W-BP) to be the time interval between the arrival of a Willie job that finds no other Willie job in the system and the first subsequent departure of a Willie job that leaves no other Willie jobs in the system. Let $M_j$ denote the number of Willie's jobs served in the first $j$ W-BPs, which can be defined recursively by $M_0 = 0$ and 
\begin{eqnarray}
M_j=\min\left\{i > M_{j-1} : A_{i+1}>D_i \right\} , \quad j\geq 1.
\label{nb-cust-BP}
\end{eqnarray}
Let $N_j=M_j-M_{j-1}$ denote the number of Willie jobs served in the $j$-th W-BP.

Willie's observation is $W_{1:n}$, where 
\begin{equation}
\label{Wj}
W_j:=\left(M_j=m_j ,A_{(m_{j-1}+1):m_j}=a_{(m_{j-1}+1):m_j}, S_{(m_{j-1}+1):m_j}= s_{(m_{j-1}+1):m_j}\right), \quad j=1,\ldots,n.
\end{equation}
In words, Willie observes $n$ W-BPs and for each W-BP records the number of his jobs that have been served, their arrival times and reconstructed service times.\\

The null hypothesis $H_0$ is  that Alice does not insert jobs and the alternative hypothesis $H_1$is  that Alice inserts jobs.  Willie's test may incorrectly accuse Alice when she does not insert jobs, i.e. he rejects $H_0$ when it is true. This is known as type I error or false alarm, and,  its probability is denoted by $P_{FA}$ \cite{LR05}.  On the other hand, Willie's test may fail to detect insertions of Alice's jobs, i.e. he accepts $H_0$ when it is false. This is known as type II error or missed detection, and  its probability is denoted by $P_{MD}$. Assume that Willie uses classical hypothesis testing with equal prior probabilities of each hypothesis being true.  Then, the lower bound on the sum $P_E=P_{FA}+P_{MD}$ characterizes the necessary trade-off between false alarms and missed detections in the design of a hypothesis test.  If prior probabilities are not equal, $\Prob(H_0) = \pi_0$ and $\Prob(H_1) = \pi_1$, then, $P_E \ge \min (\pi_0,\pi_1)(P_{FA}+P_{MD})$ \cite[Sec. V.B]{BGT12}. Hence scaling results obtained for equal priors apply to the case of non-equal priors and we focus on the former in the remainder of the paper.
\begin{figure}
\begin{tabular}{| c | l | l }
\hline
General&\\
\hline
$\lambda $& arrival rate Willie jobs\\
 $G_1$, $g_1$ &cdf, pdf Willie job service time  \\
$G_2$,  $g_2$  &cdf, pdf Alice job service time  \\
 $1/\mu_1$ &Willie job expected service time\\
 $1/\mu_2$ &  Alice job expected service time\\
 $H_0$ & null hypothesis (Alice does not insert jobs)\\
 $H_1$ & alternate hypothesis (Alice inserts jobs)\\
  $T_V(u_0,u_1)$ & variation distance between pdfs $u_0$ and $u_1$\\
 $H(u_0,u_1)$&  Hellinger distance between pdfs $u_0$ and $u_1$\\
 $X$ & random variable (rv) with pdf $g_1$\\
 \hline\hline
 IEBP policy&\\
 \hline
 W-BP & Willie Busy Period\\
 $Y$ & Willie first job reconstructed service time\\
 $V$ & Willie idle period duration (exp. rate $\lambda$)\\
 $f_i$ & pdf of $(Y,V)$ under $H_i$, $i=0,1$  \\
  & ($f_0(y,v)=g_1(y)\lambda e^{-\lambda v}$)\\
  $q$ & probability Alice inserts a job ($\bar q=1-q$);\\
  &depends on $n$, the number of W-BPs\\
 $Z(q,y,v)$ & $f_1(y,v)/f_0(y,v)$\\
 $C_0$ & Fisher information at origin \\
 & (=$\E[\rho(X,V)^2]$, with $\rho(x,v)$ defined in (\ref{eq:rho}))\\
 $\widetilde f_i$ & pdf of $Y$ under $H_i$, $i=0,1$ ($\widetilde f_0(y)=g_1(y)$)\\
 $\widetilde Z(q,y)$ & $\widetilde f_1(y)/\widetilde f_0(y)$ \\
 $\mu r$ & $\mu_1$\\
 $\mu$ & $\mu_2$\\
 $X_r$ & exponential rv, rate $\mu r$\\
 $N_j$ & nb. Willie jobs served in $j$-th W-BP\\
 $M_j$ & nb. Willie jobs served in first $j$ W-BPs ($M_j=\sum_{l=1}^j N_l$)\\
 $T(n)$ & expected nb. Alice jobs inserted in $n$ W-BPs\\
 & ($T(n)=nq$)\\
 $T_W(n)$ &expected nb. of Willie jobs served in $n$ W-BPs\\
 \hline
 \end{tabular}
 \caption{Glossary of main notations}
 \label{fig:notation}
 \end{figure}


\section{Insert-at-End-of-Busy-Period Policy}
\label{sec:IEBP}
In this section, we consider the strategy that Alice inserts a job probabilistically at the end of each W-BP, which we call the \emph{Insert-at-End-of-Busy-Period (IEBP) Policy}.
Note that there may be an Alice job in the system at the start of a W-BP.  This occurs when Alice has inserted a job at the end of a W-BP and that her job has not completed 
service by the time the next Willie job arrives. 

Throughout the section, we assume that Willie jobs have service time distribution $G_1$ with continuous pdf $g_1$ and finite mean $1/\mu_1>0$, and that Alice jobs have service time distribution $G_2$ with continuous pdf $g_2$ and finite mean $1/\mu_2>0$. Denote by $G_i^*(s)=\int_0^\infty e^{-s x} g_i(x)dx$ the Laplace Stieltjes transform (LST) of $g_i$ for $i=1,2$. We also assume $\lambda/\mu_1<1$ so that the system is stable under $H_0$.

\subsection{Introducing the IEBP Policy}
\label{subsec:IEBP}

To motivate the IEBP policy, we first find the minimum probability that an Alice job interferes with Willie's jobs. Suppose an Alice job is inserted at time $t$, with service time $\sigma_2\sim G_2$. Let $U_t\geq 0$ be the unfinished work (of both Alice and Willie jobs) in the system just before time $t$. The newly inserted Alice job will affect Willie if he sends a job in the interval  $(t,t+U_t+\sigma_2)$, the probability of which is
\begin{align}
 &\quad \; \P(\hbox{at least one Willie job arrives in }(t,t+U_t+\sigma_2 )) \nonumber \\
 &  \ge  \P(\hbox{at least one Willie job arrives in }(t,t+\sigma_2 )) \nonumber \\
 & =  \int_0^\infty  \P(\hbox{at least one Willie job arrives in }(t,t+x))g_2(x)dx\nonumber\\
 & =  \int_0^\infty  (1-e^{-\lambda x})g_2(x) dx =  1-G_2^*(\lambda):=p.
 \label{eq:interference-prob}
\end{align}
Thus if Alice is to insert a single job then she should insert it when the system is idle so as to minimize the probability of interfering with a Willie job.  Motivated by this observation, we introduce the IEBP policy below.\\

\noindent{\em Alice's strategy.}  Alice inserts a job with probability $q$ at the end of each W-BP.  
We refer to this as the {\em Insert-at-End-of-Busy-Period} (IEBP) policy. Given that Alice does insert a job, the probability that it interferes with a Willie job is
given by $p$ in \eqref{eq:interference-prob}.  Thus $pq$ is the probability that an interference occurs in a given W-BP.

\subsection{Willie's Detector}
\label{ssec:detector}

It is not easy to work directly with the observation process $W_{1:n}$  defined in (\ref{Wj}) (Section \ref{sec:model}). Instead, we will work with the statistic
$(Y_{1:n},V_{1:n})$, with $Y_j$ denoting the reconstructed service time of the first Willie job in the $j$-th W-BP and $V_j$ the length of the idle period preceding it.
These quantities are given by
\begin{eqnarray}
Y_j &=& S_{M_{j-1}+1}, \label{def-Yj}\\
V_j &= &A_{M_{j-1}+1} - D_{M_{j-1}}.\label{def-Vj}
\end{eqnarray}

Denote by $f_i(y,v)$ the joint pdf of $(Y,V)$ at $(Y=y, V=v)$ under $H_i$ for $i=0,1$.  Also, let $\widetilde f_i(y):=\int_0^\infty f_i(y,v)dv$ denote
the pdf of $Y$ at $Y=y$ under $H_i$ for $i=0,1$.  Under $H_0$ the system is a standard $M/G/1$ queue; in particular, the random variables (rvs) $Y$ and $V$ are independent 
with pdf $g_1(y)$ and $\lambda e^{-\lambda v}$, respectively, yielding
\begin{eqnarray}
f_0(y,v)&=&g_1(y)\lambda e^{-\lambda v},
\label{value-f0}\\
\widetilde f_0(y)&=&g_1(y).
\label{f0tilde}
\end{eqnarray}
Under $H_1$, $Y_j$ is the sum of the remaining service time of Alice's job, if any, when a Willie job  initiates the $j$-th  W-BP,
and of the service time of this Willie's job. Therefore, under $H_1$ the system is an M/G/1 queue with arrival rate $\lambda$, exceptional first service time in a busy period with
pdf $\widetilde f_1$, and all other  service times (the ordinary customers) in a busy period with  pdf $g_1$.  In  \prettyref{prop:Tw} (Section \ref{sec:main-results})
we derive some performance metrics of interest for this queueing system.

An important feature of the  process $(Y_{1:n}, V_{1:n})$ is that $(Y_1,V_1), \ldots, (Y_n,V_n)$ form iid rvs due to the Poisson nature of Willie job arrivals and
the assumptions that Willie's and Alice's service times are mutually independent processes, further independent of the arrival process. This is the main benefit from
using the statistic $(Y_{1:n}, V_{1:n})$ instead of the statistic $W_n$. From now on  $(Y,V)$ denotes a generic $(Y_j, V_j)$.

The independence of the rvs $(Y_1,V_1), \ldots,(Y_n,V_n)$ under $H_i$ ($i=0,1$) implies that their joint pdf is given by $f_i^{\otimes n}$, the $n$-th fold tensor product of $f_i$ with itself.

The following lemma  shows that $(Y_{1:n},V_{1:n})$ is a {\em sufficient statistic} (e.g. see  \cite[Chapter 1.9]{LR05}), that is, Willie does not lose any information by considering the statistic 
$(Y_{1:n},V_{1:n})$ instead of the statistic $W_{1:n}$  in order to detect the presence of Alice. The proof is given in Appendix \ref{app:sufficient-statistics}.

\begin{lemma} 
\label{lemma:likelihood}
For every $n\geq 1$, $(Y_{1:n},V_{1:n})$ is a sufficient statistic. 
\end{lemma}
Theorem 13.1.1 in \cite{LR05} is established in the case where a  simple hypothesis $P_0$ is tested against a simple alternative $P_1$. In our setting, $q=0$ is the simple hypothesis against the simple alternative $q=q(n)$, and we are asking how we can scale $q(n)$ down to 0 so that we can not differentiate $q=0$ and $q=q(n)$, with $n$ the number of observed W-BPs.

\begin{thm}(\cite[Theorem 13.1.1]{LR05})\hfill
\label{thm:optimal-test}

Using the observed values $(y_{1:n}, v_{1:n})$ of $(Y_{1:n},V_{1:n})$,
any test accepting $H_0$ if $\prod_{i=1}^n f_0(y_j,v_j) > \prod_{i=1}^n f_1(y_j,v_j) $ and rejecting $H_0$ if 
$\prod_{i=1}^n f_0(y_j,v_j) <\prod_{i=1}^n f_1(y_j,v_j)$ minimizes $P_E$. Furthermore, the minimum $P_E$ is given by
\[
P_E^\star = 1 - T_V \left( f_0^{\otimes n},  f_1^{\otimes n}\right),
\]
where
\begin{equation}
\label{eq:total-variation}
	T_V (u_0,u_1) = \frac{1}{2}\int |u_0(x) - u_1(x)|dx
\end{equation}
is the total variation distance between two distributions with densities $u_0$ and $u_1$, respectively.
\end{thm}
We will henceforth assume that for a given $(Y_{1:n},V_{1:n})$  Willie uses the above optimal test.

We say that Alice's insertions are covert provided that, for any $\epsilon>0$, she has an insertion strategy for each $n$ such that
\begin{equation}
\label{covert-criterion}
\liminf_n P_E^\star \ge 1-\epsilon,
\end{equation}
or equivalently from Theorem \ref{thm:optimal-test}, if for any $\epsilon >0$,
\begin{equation}
\label{covert-criterion2}
\limsup_n  T_V \left( f_0^{\otimes n}, f_1^{\otimes n} \right)\leq \epsilon.
\end{equation}
Note that a sufficient condition for Alice's insertions not being covert is that for some $\delta\in (0,1)$ there exists a detector such that 
\begin{equation}
\label{non-covert-criterion}
\limsup_n P_E<\delta.
\end{equation}
Here the limit is taken over the number of busy periods that Willie observes. This covertness criterion was proposed in the context of low probability of detection (LPD) communications in \cite{BGT12}. 
\\

Theorem \ref{thm:optimal-test} suggests using the total variation distance to analyze Willie's detectors. However, the total variation distance is often unwieldy even for 
products of pdfs, like  $f_0^{\otimes n}$ and $f_1^{\otimes n}$.
To overcome this drawback, it is common (e.g. see \cite{BGT12})  to use the following Pinsker's inequality (Lemma 11.6.1 in \cite{CT02})
\begin{equation}
\label{Pinsker-inq}
T_V(u_0,u_1)\leq KL(u_0\| u_1),
\end{equation}
where $KL(u_0 \| u_1) :=\int_{\R^d} u_0(x)\ln  \frac{u_0(x)}{u_1(x)} dx$
is the {\em Kullback-Leibler (KL) divergence} between the probability distributions with pdf $u_0$ and $u_1$, respectively. 

However, we will work with the Hellinger distance, which has the advantage of offering both lower  and  upper bounds on the total variation distance.
The Hellinger distance between two probability distributions with pdf $u_0$ and $u_1$ respectively,
denoted $H(u_0,u_1)$,  is defined by 
\begin{equation} \label{eq:Hellinger-def}
H(u_0,u_1)=\frac{1}{2}\int_{\R^d}\left(\sqrt{u_0(x)}-\sqrt{u_1(x)}\right)^2 dx.
\end{equation}
Note that 
\begin{equation}
\label{affinity-H}
H(u_0,u_1)=1-\int_{\R^d}\sqrt{u_0(x)u_1(x)}dx,
\end{equation}
and $0\leq H(u_0,u_1)\leq 1$. 
It is known \cite[Lemma 4.1]{Kraft-1955} that
\begin{equation}\label{eq:lb-ub}
H(u_0,u_1)\leq T_V(u_0,u_1)\leq \sqrt{2 H(u_0,u_1)}.
\end{equation}
The upper bound (resp. lower bound) in (\ref{eq:lb-ub}) will be used to establish covert (resp. non-covert) results.

We will also use the following well-known property of the Hellinger distance between pdfs $u_0^{\otimes n}$ and $u_1^{\otimes n}$ \cite[Eq. (1.4)]{Oosterhoof-75}:
\begin{eqnarray}
H\left(u_0^{\otimes n},u_1^{\otimes n}\right)
&=&1-\int_{\R^d}\prod_{j=1}^n \sqrt{u_0(x_j) u_1(x_j)}\,\prod_{j=1}^n dx_j \quad \hbox{by (\ref{affinity-H})}\nonumber\\
&=&1-\left(\int_{\R^d}\sqrt{u_0(x) u_1(x)} dx\right)^n.
\label{eq:H-alt}
\end{eqnarray}


\section{Main Results }
\label{sec:main-results}

Let $T(n)$ denote the expected number of jobs that Alice inserts in $n$ W-BPs. Under the IEBP policy,
\begin{equation}
\label{value-Tn}
T(n)=nq. 
\end{equation}
This section presents the main results that characterize $T(n)$ under various conditions as $n$ becomes large. Implicit in all asymptotic results as $n\to\infty$ is 
that $q$ is a function of $n$.\\ 

Recall that $f_i(y,v)$ is the joint pdf of $(Y,V)$ at $(Y=y, V=v)$ under $H_i$ for $i=0,1$, with $f_0$ given in (\ref{value-f0}).
The likelihood ratio 
\begin{equation}
\label{eq:Z-def}
Z(q,y,v):=\frac{f_1(y,v)}{f_0(y,v)},
\end{equation}
plays an important role in determining how many jobs Alice can insert covertly.  It is shown in Lemma \ref{lem:Z-rho} in Appendix \ref{app:Zqxv}
that $Z$ has the following form,
\begin{equation}
Z(q,y,v) = 1 + q \rho(y,v), 
\label{eq:Z-def-with-rho}
\end{equation}
where
\begin{equation}
\label{eq:rho}
\rho(y,v):=\frac{1}{g_1(y)}\int_0^y g_1(u)g_2(v+y-u)du -\overline G_2(v).
\end{equation}
Since  $\rho(y,v)$ does not depend on the insertion probability $q$, this shows that the likelihood ratio $Z(q,y,v)$ depends linearly on $q$.
Define
\begin{equation}
\label{eq:C0}
C_0:=\E[\rho(X,V)^2],
\end{equation}
where $(X,V)$ has pdf $f_0(x,v)$ at $(x,v)$.

It is worth noting that $C_0=J(0)$, with $J(q):=\E\left[\left(\frac{d}{dq} \log f_1(X,V)\right)^2\right]$
the Fisher information of $(Y,V)$ about the parameter $q$. Indeed, since $f_1(x,v)= f_0(x,v)(1+q\rho(x,v))$ by (\ref{eq:Z-def})-(\ref{eq:Z-def-with-rho}), we have
\[
J(q)= \int \frac{1}{f_1(x,v)^2} \left(\frac{d}{dq} f_1(x,v)\right)^2 f_0(x,v) dx dv= \int  \frac{\rho(x,v)^2}{(1+q\rho(x,v))^2} f_0(x,v)dx dv,
\]
and therefore $J(0)=\E[\rho(X,V)^2]$.
The Fisher information evaluates the amount of information that a random variable carries about an unknown parameter  \cite{Kullback-68}.

\prettyref{prop:achievability-IEBP} gives the covert throughput for general  service time distributions. Its proof is given in \prettyref{sec:achievability-IEBP-proof}. 

\begin{proposition}[Covert throughput for general service time distr. and finite $C_0$] \hfill
\label{prop:achievability-IEBP}
Assume $C_0 < \infty$. Under the IEBP policy, the number of jobs Alice can insert covertly is  $T(n)=\mcO(\sqrt{n})$ if \,$\E[\rho(X,V)]=0$, and $T(n) = \mcO(1)$ if \,$\E[\rho(X,V)]\ne 0$.
\end{proposition}

\begin{remark} 
\label{rmk:rho-exp}
$\E[\rho(X,V)]=0$ for any pdf $g_1$ if $g_2$ is the pdf of an exponential or an hyper-exponential rv. Indeed, when $g_2(x)=\mu_2 e^{-\mu_2 x}$, $\rho(x,v)$ in (\ref{eq:rho}) writes
\begin{equation} 
\label{eq:rho-g2-exp}
\rho(x,v)= e^{-\mu_2 v} \left(\frac{(g_1*g_2)(x)}{g_1(x)}-1\right).
\end{equation}
By the independence of $X$ and $V$ and the fact that  $g_1*g_2$ is a pdf,
\begin{equation}
\label{expectation-rho-g2-exp}
\E[\rho(X,V)]=\E[e^{-\mu_2 V}] \cdot \E\left[\left(\frac{(g_1*g_2)(X)}{g_1(X)}-1\right)\right]=0.
\end{equation}
The proof when $g_2$ is the pdf of an hyper-exponential rv is a simple generalization.  
 Note  that  $\E[\rho(X,V)]$ does not always vanish. In particular, $\E[\rho(X,V)]\not =0$ when $G_2$ is an Erlang distribution.  Indeed, when
 $g_2(x)=\frac{(k\mu_2)^k}{(k-1)!} x^{k-1}e^{-k\mu_2 x}$, $k\geq 1$ (Alice service times follow a $k$-Erlang distribution with mean $1/\mu_2$), it is easy to show that for any pdf
 $g_1$, $\E[\rho(X,V)]=(1-G^*_2(\lambda))\left(\frac{(k\mu_2)^k}{\mu_2}-1\right)\not=0$ for  all $k>1$. 
\end{remark}

The next lemma gives conditions for $C_0<\infty$ under various distributional assumptions. Its proof is found in Appendix \ref{app:A}.

\begin{lemma}[Finiteness of $C_0$]\hfill 
\label{lem:C-different-cdf}

\begin{enumerate}
\item Suppose both Alice and Willie have exponential service times, i.e. $g_i(x)=\mu_i e^{-\mu_i x}$ for $i=1,2$.  Then $C_0<\infty$ if and only if $\mu_1<2\mu_2$.

\item Suppose both Alice and Willie have hyper-exponential service times,
 i.e. 
 \[
 g_i(x)=\sum_{l=1}^{K_i} p_{i,l} \mu_{i,l}e^{-\mu_{i,l}x}, 
 \]
 where $\sum_{l=1}^{K_i} p_{i,l} =1$, for $i=1,2$. Then $C_0<\infty$ if and only if
\[
\max_{1\leq l\leq K_i}\mu_{1,l} \leq 2 \min_{1\leq m\leq K_2} \mu_{2,m}.
\]

\item Suppose Willie has Erlang service times and Alice has hyper-exponential service times, i.e.
\[
g_1(x)=\frac{\nu_1^{K_1}}{(K_1-1)!} x^{K_1-1} e^{-\nu_1 x}
\]  
and 
\[
g_2(x)=\sum_{l=1}^{K_2} p_{2,l} \mu_{2,l}e^{-\mu_{2,l}x},
\]
where $\sum_{l=1}^{K_2} p_{2,l} =1$. Then
$C_0<\infty$ if and only if 
\[
\nu_1<2\min_{1\leq l\leq K_2} \mu_{2,l}.
\]
\end{enumerate}
\end{lemma}

\prettyref{prop:achievability-IEBP} gives sufficient conditions for Alice to be covert. This raises the following questions: 
\begin{itemize}
\item[{\bf Q1:}]
When $C_0<\infty$, can Alice insert covertly more than  $\mcO(\sqrt{n})$ jobs on average during $n$ W-BPs? 
\item[ \bf{ Q2:}] 
When $C_0=\infty$, what is the maximum number of jobs that Alice can insert covertly on average during $n$ W-BPs?
\end{itemize}

We do not have full answers to the above questions. \prettyref{prop:converse-IEBP} first gives a necessary condition for Alice to be covert under IEBP. \prettyref{prop:phase-transition-exponential} then provides a partial answer for the IEBP policy when both Alice and Willie have exponential service times. Proofs are found in Sections \ref{sec:converse-IEBP-proof} and \ref{sec:proof-phase-transition-exponential}.

\begin{proposition}[Necessary condition for covertness]\hfill
\label{prop:converse-IEBP} 

Under IEBP Alice cannot be covert if $\limsup_{n\to\infty} q >0$. 
\end{proposition}

Consider now the situation when $\lim_{n\to\infty} q=0$. The result below is the main result of the paper.

\begin{proposition}[Covert throughput and converse for exponential service time distr.]\hfill
\label{prop:phase-transition-exponential}
Assume that $g_i(x)=\mu_i  e^{-\mu_i x}$ for $i=1,2$, and Alice uses the IEBP policy with $\lim_{n\to\infty}q=0$.
She can be covert if
\begin{equation}\label{eq:achievability-IEBP-exp}
T(n)=\begin{cases}
\mcO(\sqrt{n}), &\mbox{if $\mu_1<2\mu_2$},\\
\mcO(\sqrt{n/\log n}), &\mbox{if $\mu_1=2\mu_2$},\\
\mcO(n^{\mu_2/\mu_1}),  &\mbox{if $\mu_1>2\mu_2$.} 
\end{cases}
\end{equation}
She cannot be covert if 
\begin{equation}\label{eq:converse-IEBP-exp}
T(n)=\begin{cases}
\omega(\sqrt{n}), &\mbox{if $\mu_1<2\mu_2$},\\
\omega(\sqrt{n/\log n}), &\mbox{if $\mu_1=2\mu_2$},\\
\omega(n^{\mu_2/\mu_1}),  &\mbox{if $\mu_1>2\mu_2$}.
                  \end{cases}
\end{equation}
\end{proposition}

The above results are in terms of $T(n)$,  the expected number of jobs inserted by Alice over $n$ successive W-BPs.  It is interesting to determine also the  expected number of Willie  jobs served during these $n$ W-BPs under the IEBP policy. Let $T_W(n)$ be this number.
When $q=0$, the system behaves like a standard M/G/1 queue with traffic intensity $\lambda/\mu_1<1$ and it is known that the expected number of jobs served in a busy period is $(1-\lambda/\mu_1)^{-1}$ \cite{Kleinrock75}, yielding $T_W(n)=n(1-\lambda/\mu_1)^{-1}$.  
\prettyref{prop:Tw} below shows that the IEBP policy increases each W-BP by a constant factor.
 The proof is in Appendix \ref{app:Tw-proof}.

\begin{proposition}
\label{prop:Tw} 
Under IEBP,  $T_W(n)=\Theta(n)$. More precisely
\begin{equation}
\label{prop:eq:value-TWn}
T_W(n)= \frac{n}{1-\lambda/\mu_1}+\frac{\lambda q n}{1-\lambda/\mu_1} \int_0^\infty t \hat g_2(t)dt,
\end{equation}
which gives the two-sided inequality 
\begin{equation}\label{eq:Tw-bounds}
n\le T_W(n)\leq n \left(\frac{1+q\lambda/\mu_2}{1-\lambda/\mu_1}\right).
\end{equation}
If $g_2(x)=\mu_2 e^{-\mu_2 x}$, i.e. Alice  job service times are exponentially distributed, then  \begin{equation}
\label{eq:Tw-exponential}
T_W(n)=n\left(\frac{1+pq\lambda/\mu_1}{1-\lambda/\mu_2}\right).
\end{equation} 
\end{proposition}
\begin{remark}
Recall that the Fisher information regarding $q$ is infinity at $q=0$ when $\mu_1 \ge \mu_2$. It is interesting to speculate that this translates into facilitating Willie's detection task, which appears in the form of the phase transition in equations (\ref{eq:achievability-IEBP-exp}) and (\ref{eq:converse-IEBP-exp}). Note that one implication of this transition is that  Alice should select job sizes with mean size $1/\mu_2 <2/\mu_1$ to increase throughput without being detected.
\end{remark}


\section{Preliminary results}
\label{sec:preliminary-results}
We first specialize the two-sided inequality in (\ref{eq:lb-ub}) to the case where $u_0=f_0^{\otimes n}$ and $u_1=f_1^{\otimes n}$ and then develop 
a covert (resp. non-covert) criterion for Alice.
\begin{lemma}\label{lem:Hellinger} 
The Hellinger distance between $f_0^{\otimes n}$ and $f_1^{\otimes n}$ is given by
\begin{equation}\label{eq:Hellinger}
H\left(f_0^{\otimes n},f_1^{\otimes n}\right)= 1-\left(\E\left[\sqrt{Z(q,X,V)}\right]\right)^n.
\end{equation}
\end{lemma}
\begin{proof}
By \eqref{eq:H-alt} and \eqref{eq:Z-def},
\begin{align*}
H\left(f_0^{\otimes n},f_1^{\otimes n}\right) &=
 1-\left(\int_{[0,\infty)^2} \sqrt{f_0(y,v)f_1(y,v)} \,dy dv\right)^n\\
&= 1-\left(\int  f_0(y,v)\sqrt{Z(q,y,v)} \,dy  dv\right)^n\\
&= 1-\left(\E\left[\sqrt{Z(q,X,V)}\right]\right)^n.
\end{align*}
\end{proof}
Specializing (\ref{eq:lb-ub}) to the value of $H\left(f_0^{\otimes n},f_1^{\otimes n}\right)$ found in Lemma \ref{lem:Hellinger} gives,
\begin{lemma}[Lower $\&$ upper bounds on total variation distance for statistic $\{Y_j,V_j\}_j$]  \hfill
\label{lem:lb.up-total-variation}
For every $n\geq 1$
\begin{align}
1-\left(\E\left[\sqrt{Z(q,X,V)}\right]\right)^n \leq T_V\left(f_0^{\otimes n},f_1^{\otimes n}\right)\leq \sqrt{2\left(1-\left(\E\left[\sqrt{Z(q,X,V)}\right]\right)^n\right)}.
\label{2-sided-bounds-Z}
\end{align}
\end{lemma}

Combining Lemma \ref{lem:lb.up-total-variation}, the covert criterion (\ref{covert-criterion2}), and  the non-covert criterion (\ref{non-covert-criterion}) yields the following covert/non-covert criterion for Alice:
\begin{corollary}[Covert/non-covert criteria for the IEPB policy] \hfill \\
 \label{cor:covert-non_covert}
Assume that Willie uses an optimal detector for the sufficient statistic $(Y_{1:n},V_{1:n})$.
Alice's insertions are covert if  for any $\epsilon>0$,
\begin{equation}
\label{cor:covert}
\liminf_n\left(\E\left[\sqrt{Z(q,X,V)}\right]\right)^n\geq 1-\epsilon,
\end{equation}
and Alice's insertions are not covert  if for any  $\delta>0$
\begin{equation}
\label{cor:non-covert}
\limsup_n\left(\E\left[\sqrt{Z(q,X,V)}\right]\right)^n<\delta.
\end{equation}
\end{corollary}
The lower bound (\ref{cor:covert}) is used to derive the covert throughput (\ref{eq:achievability-IEBP-exp})  in Proposition \ref{prop:phase-transition-exponential}.\\

We now state and prove a non-covert criterion for the IEPB policy.  We do so by proposing and analyzing a detector that relies on the  (non-sufficient) statistic $\{Y_j\}_j$.
Recall that rvs $Y_1,\ldots,Y_n$ are iid with common pdf $\widetilde f_i(y)$ under $H_i$, for $i=0,1$.  The non-covert criterion is obtained by applying Theorem 13.1.1 in 
\cite{LR05} to the statistic $Y_{1:n}$, which yields the minimum $P_E$ to be $1-T_V\left(\widetilde f_0^{\otimes n}, \widetilde f_1^{\otimes n}\right)$.

The following lemmas are the analog of Lemmas \ref{lem:Hellinger} and \ref{cor:covert-non_covert} for the statistic $\{Y_j\}$. 
\begin{lemma}\label{lem:Hellinger-Y} 
The Hellinger distance between $\widetilde f_0^{\otimes n}$ and $\widetilde f_1^{\otimes n}$ is given by
\begin{equation}
\label{eq:Hellinger-Y}
H\left(\widetilde f_0^{\otimes n},\widetilde f_1^{\otimes n}\right)= 1-\left(\E\left[\sqrt{\widetilde Z(q,X)}\right]\right)^n,
\end{equation}
with \begin{equation}
\label{eq:Ztilde-def}
\widetilde Z(q,x): = \frac{\widetilde f_1(x)}{g_1(x)}.
\end{equation}
\end{lemma}
\begin{lemma}[Lower bound on total variation distance for statistic $Y_{1:n}$]  
\label{lem:lb.up-total-variation-Y}
For every $n\geq 1$
\begin{align}
1-\left(\E\left[\sqrt{\widetilde Z(q,X)}\right]\right)^n \leq T_V\left(\widetilde f_0^{\otimes n},\widetilde f_1^{\otimes n}\right).
\label{bounds-Z-Y}
\end{align}
\end{lemma}
The proof of Lemma \ref{lem:Hellinger-Y} mimics that of Lemma \ref{lem:Hellinger} and is omitted.  Lemma \ref{lem:lb.up-total-variation-Y} follows from
Lemma \ref{lem:Hellinger-Y} and the lower bound in (\ref{eq:lb-ub}).

The non-covert criterion for the statistic $\{Y_j\}_j$ announced earlier is given below. Its proof follows from (\ref{non-covert-criterion}) and (\ref{bounds-Z-Y}).

\begin{corollary}[Non-covert criterion for the IEPB policy] \hfill \\
 \label{cor:non_covert-Y}
Assume that Willie uses an optimal detector for the statistic $\{Y_j\}_j$. 
Alice's insertions are not covert  if for any  $\delta>0$
\begin{equation}
\label{cor:non-covert-Y}
\limsup_n\left(\E\left[\sqrt{\widetilde Z(q,X)}\right]\right)^n<\delta.
\end{equation}
\end{corollary}
Corollary  \ref{cor:non_covert-Y} is used in the proofs of Proposition \ref{prop:converse-IEBP} and of the converse (\ref{eq:converse-IEBP-exp}) in Proposition \ref{prop:phase-transition-exponential}.\\

\begin{remark}
In direct analogy with the upper bound in  Lemma \ref{lem:lb.up-total-variation} it is worth noting that $ 
T_V\left(\widetilde f_0^{\otimes n},\widetilde f_1^{\otimes n}\right)$ in Lemma \ref{lem:lb.up-total-variation-Y} is upper bounded by
$\sqrt{2\left(1-\left(\E\left[\sqrt{\widetilde Z(q,X)}\right]\right)^n\right)}$. However, this bound is not useful for establishing a covert result since it does not use  
the sufficient statistic $(Y_{1:n},V_{1:n})$. 
\end{remark}


\section{Proof of \prettyref{prop:achievability-IEBP}}
\label{sec:achievability-IEBP-proof}

The proof of Proposition \ref{prop:achievability-IEBP} relies on an upper bound on the total variation distance between
 $f_0^{\otimes n}$ and $f_1^{\otimes n}$, given in Lemma \ref{lem:up-TV} below.  Recall the definition of $\rho(X,V)$ given in
 (\ref{eq:rho}).
 
\begin{lemma}
\label{lem:up-TV}
Let $(X,V)$ be rvs with density $f_0$.
Assume that $\E[\rho(X,V)]=0$. Then, for every $n\geq 1$,
\begin{equation}
\label{ub:TV}
T_V\left(f_0^{\otimes n},f_1^{\otimes n}\right) \leq \frac{1}{2} \sqrt{(1+q^2 C_0)^n -1},
\end{equation}
where $C_0$ is defined in \eqref{eq:C0}.
\end{lemma}
\begin{proof}  Let $\{(X_j, V_j)\}_j$ be iid rvs with pdf $f_0$.
 By \eqref{eq:total-variation},
\begin{eqnarray}
2T_V\left(f_0^{\otimes n}, f_1^{\otimes n}\right) &=&\int_{[0,\infty)^{2n}} \left| \prod_{j=1}^n f_0(y_j,v_j)- \prod_{j=1}^n f_1(y_j,v_j)\right| \prod_{j=1}^n dy_j dv_j\nonumber\\
&=& \int_{[0,\infty)^{2n}}  \prod_{j=1}^n f_0(y_j,v_j)\left| 1- \prod_{j=1}^n Z(q,y_j, v_j) \right| \prod_{j=1}^n dy_j dv_j\nonumber\\
&=&\E\left[\, \left|1-\prod_{j=1}^n Z(q,X_j,V_j)\right|\,\right], \label{TV}
\end{eqnarray}
where the second equality follows from \eqref{eq:Z-def}, and (\ref{TV}) follows from the fact that $f_0$  given in (\ref{value-f0}) is the joint pdf of the
independent rvs $X$ and $V$.

Using the inequality $\E|U|\leq \sqrt{E[U^2]}$ in (\ref{TV}) yields
\begin{eqnarray}
\left(2 T_V\left(f_0^{\otimes n},f_1^{\otimes n}\right)\right)^2
&\leq&  1-2\E\left[ \prod_{j=1}^n Z(q,X_j,V_j)\right]+ \E\left[ \left(\prod_{j=1}^n Z(q,X_j,V_j)\right)^2 \right]\nonumber\\
&=& 1- 2 \left[\E Z(q,X,V)\right]^n +\left[\E[Z(q,X,V)^2]\right]^n \label{inq-lem}\\
&=& -1 + \left(1+2q\E[\rho(X,V)] + q^2 \E[\rho(X,V)^2]\right)^n\nonumber\\
&=& -1 + \left(1+ q^2 \E[\rho(X,V)^2]\right)^n,\nonumber
\end{eqnarray}
where we have used the value of $Z(q,x,v)$ obtained  in \eqref{eq:Z-rho} in Appendix \ref{app:Zqxv} and the assumption that $\E[\rho(X,V)]=0$ to establish the last two identities. 

\end{proof}

We are now in position to prove Proposition \ref{prop:achievability-IEBP}.
\begin{proof}[Proof of  \prettyref{prop:achievability-IEBP}] In order to compute the covert throughput, assume that Willie uses an optimal detector for the sufficient statistic 
$\{(Y_j,V_j),j=1,\ldots,n\}$. Let
\[
q=\frac{\delta}{\phi(n)},
\]
with $\delta\in (0,1]$ and $\phi:\{1,2,\ldots\}\to [1,\infty)$, so that by (\ref{value-Tn})
\begin{equation}
\label{eq:Tn}
T(n)=\frac{\delta n}{\phi(n)}.
\end{equation}

First consider the case $\E[\rho(X,V)] = 0$. 
Note that $T(n)= \mcO(\sqrt{n})$ implies 
\begin{equation} 
\label{liminf-phi}
\limsup_n\frac{\sqrt{n}}{\phi(n)}<\infty.
\end{equation}
By Lemma \ref{lem:up-TV}
\begin{align*}
\sup_{k\geq n}T_V\left(f_0^{\otimes k},f_1^{\otimes k}\right)\leq  \sup_{k\geq n} \frac{1}{2}\sqrt{\left(1+\frac{\delta^2C_0}{\phi(k)^2}\right)^k -1}\sim \frac{\delta \sqrt{C_0}}{2}\,\sup_{k\geq n}  \frac{k }{\phi(k)^2}, \quad\text{as } n\to\infty
 \end{align*}
as  $\lim_n\phi(n)=\infty$ by Lemma \ref{lem:technic} in Appendix \ref{app:A0}.
Therefore,
\begin{equation}
\label{inq-achiev}
\limsup_n T_V\left(f_0^{\otimes n},f_1^{\otimes n}\right)\leq  \frac{\delta \sqrt{C_0}}{2} \left(\limsup_n\frac{\sqrt{n}}{\phi(n)}\right)^2.
\end{equation}
By making $\delta$ small enough,  $\limsup_nT_V(f_0^{\otimes n},f_1^{\otimes n})$ can be made arbitrarily small. 
We then conclude from (\ref{covert-criterion}) that Alice is covert when $T(n)= \mcO(\sqrt{n})$, which completes the proof for the case $\E[\rho(X,V)]=0$.

Now consider the case $\E[\rho(X,V)]\ne 0$. Note that
$T(n)=\frac{\delta n}{\phi(n)} = \mcO(1)$ implies there exist $k>0$ and $n_0$ such that for all $n\geq n_0$, $0\leq \frac{n}{\phi(n)}\leq k$.
Using inequality  (\ref{inq-lem}) and the definition of $Z(q,x,v)$ in (\ref{eq:Z-def-with-rho}) gives
\begin{align}
&\quad\left(2 T_V\left(f_0^{\otimes n},f_1^{\otimes n}\right)\right)^2\nonumber\\
 &\leq  1- 2e^{n\log \left(1+\frac{\delta}{\phi(n)}\E[\rho(X,V)]\right)} +e^{n\log \left(1+\frac{2\delta}{\phi(n)} \E[\rho(X,V)] + \frac{\delta^2}{\phi(n)^2} C_0^2 \right)}\nonumber\\
&\sim 1 -2e^{\E[\rho(X,V)] \frac{\delta n}{\phi(n)}} + e^{2  \E[\rho(X,V)]  \frac{\delta n}{\phi(n)} +\delta^2 C_0^2 \frac{n}{\phi(n)^2 }},
\label{inq2-lem}
\end{align}
as $n\to\infty$. Since $\frac{n}{\phi(n)}$ and $\frac{n}{\phi(n)^2}$ are bounded away from infinity as $n\to\infty$, we see that the r.h.s. of (\ref{inq2-lem}) can be made arbitrarily close to $0$ by letting
$\delta \to 0$. We then conclude from (\ref{covert-criterion}) that Alice is covert when $T(n)= \mcO(1)$, which completes the proof.
\end{proof}


\section{Proof of \prettyref{prop:converse-IEBP}}
\label{sec:converse-IEBP-proof}

The proof uses Corollary \ref{cor:non_covert-Y}.  Take $q=\frac{1}{\phi(n)}$ with $\limsup_n q>0$ or, equivalently, $\liminf_n \phi(n)<\infty$.

\textcolor{black}{Assume that Willie uses an optimal detector for the statistic $\{Y_j\}_j$ so that Corollary \ref{cor:non_covert-Y} applies.
Recall (cf. Section \ref{sec:preliminary-results}) that $\widetilde f_1(y)$ (resp. $g_1(y)$) is the pdf of $Y$ under $H1$ (resp. $H_0$), with
$\widetilde{Z}(q,y)=\frac{\widetilde f_1(y)}{g_1(y)}$ being the associated likelihood ratio.}

\textcolor{black}{ We first derive properties of $\widetilde{Z}(q,y)$, to be used in this proof and in the proof of Proposition \ref{prop:phase-transition-exponential}.}

\textcolor{black}{We claim that 
\begin{equation}
\label{eq:Ztilde}
\widetilde Z(q,y) =\int_0^\infty \lambda e^{-\lambda v}  Z(q,y,v)dv.
\end{equation}
Indeed by  (\ref{eq:Z-def}) and (\ref{value-f0})
\[ 
\int_0^\infty \lambda e^{-\lambda v}  Z(q,y,v)dv =\frac{1}{g_1(y)}   \int_0^\infty f_1(y,v) dv=\frac{\widetilde f_1(y)}{g_1(y)}=\widetilde Z(q,y),
\]
from the definition of $\widetilde Z(q,y)$ in (\ref{eq:Ztilde-def}).} Hence by \eqref{eq:Z-def-with-rho}
\begin{eqnarray}
\label{eq:Ztilde-rho}
\widetilde Z(q,y)= 1+q\widetilde \rho(y),
\end{eqnarray}
with 
\begin{eqnarray}
\label{eq:rho-tilde}
\widetilde \rho(y)&:=&\int_0^\infty \lambda e^{-\lambda v} \rho(y,v) dv\nonumber\\ 
&=&\frac{1}{g_1(y)} \int_0^\infty \lambda e^{-\lambda v} \int_0^y g_1(u) g_2(v+y-u) du dv -  \int_0^\infty \lambda e^{-\lambda v} \overline{G}_2(v) dv \quad \hbox{by }(\ref{eq:rho})\nonumber\\
&=& \frac{(g_1*\hat g_2)(y)}{g_1(y)} - p,
\end{eqnarray}
by using the definition of $p$ in (\ref{eq:interference-prob}), and where 
\begin{equation}
\label{eq:g2-hat}
\hat g_2(t):=\int_0^\infty \lambda e^{-\lambda v} g_2(v+t)dv.
\end{equation}
We are now ready to prove \prettyref{prop:converse-IEBP}. Recall that $X$ is a rv with pdf $g_1$.
We have
\begin{align}
\E\left[\sqrt{\widetilde Z(q,X)}\right] &= \int\sqrt{\prod_{i=0}^1\widetilde f_i(x)} dx
\leq \sqrt{ \prod_{i=0}^1\int \widetilde f_i(x)   dx}=1,\label{cauchy-inq}
\end{align}
by Cauchy-Schwarz inequality.
Equality holds in (\ref{cauchy-inq}) if and only if (see e.g. \cite[p. 14]{CS-inq})
$\widetilde f_1(x) = c \widetilde f_0(x)$
for some constant $c>0$. Since both $\widetilde f_0$ and $\widetilde f_1$ are densities, integrating over $[0,\infty)$ yields $c=1$, which is equivalent to $q=0$
from \eqref{eq:Ztilde-def} and \eqref{eq:Ztilde-rho}. This shows that $\E\left[\sqrt{\widetilde Z(q,X)}\right]<1$ \textcolor{black}{if and only if} $0<q\leq 1$.

Since $\liminf_n \phi(n):=d<\infty$ by assumption,  there exists a subsequence of $\{\phi(n)\}_n$, say $\{\phi(k_n)\}_n$, such that $\phi(k_n)\geq d$
with $\lim_n\phi(k_n)=d$. 

Let $M:=\sup_{1/d\leq q\leq 1} \E\left[\sqrt{\widetilde Z(q,X)}\right]$. Note $\sqrt{\widetilde Z(q,X)} = \sqrt{1+q\widetilde \rho(X)} \le \sqrt{1+|\widetilde \rho(X)|} \le 1+|\widetilde \rho(X)|$. By (\ref{eq:rho-tilde}), $\E[|\rho(X)|] \leq \int (g_1*\hat g_2)(t) dt +p$ and $g_1*\hat g_2$ is integrable as both $g_1$ and $\hat g_2$ are integrable. The Dominated Convergence Theorem then guarantees that the function $q\mapsto \E\left[\sqrt{\widetilde Z(q,X)}\right]$ is continuous. Since $\E\left[\sqrt{\widetilde Z(q,X)}\right]<1$ for all $q\in [1/d,1]$ as shown above, we have $M < 1$. Therefore,
\[
\E\left[\sqrt{\widetilde Z(1/\phi(k_n),X)}\right] \leq M < 1
\]
for all $n$. As  a result
\[
\lim_n \left(\E\left[\sqrt{\widetilde Z(1/\phi(k_n),X)}\right]\right)^n=0,
\]
which implies from Corollary \ref{cor:non_covert-Y} that Alice's insertions are not covert when $\liminf_n \phi(n)<0$, or equivalently when $\limsup_n q=\infty$.


\section{Proof of  \prettyref{prop:phase-transition-exponential}}
\label{sec:proof-phase-transition-exponential}
 
Throughout this section, we assume  that Alice and Willie job service times are exponentially distributed with rate
$\mu_2$ and $\mu_1$, respectively, namely, $g_i(x)=\mu_i e^{-\mu_i x}$ for $i=1,2$. 

The proof of  \prettyref{prop:phase-transition-exponential} relies on Corollaries \ref{cor:covert-non_covert} and  \ref{cor:non_covert-Y}, and Lemma \ref{lem:expansion-Y} below.
Before stating the latter, let us introduce some notation.

Let $\mu_1=r\mu $ and $\mu_2=\mu$. For $r\not=1$, define $\beta=\frac{r}{r-1}$ and note that $r=\frac{\beta}{\beta-1}$ and $1-\beta=\frac{1}{r-1}$. 
Let $X_r$ denote an exponential rv with rate $\mu r$.\\

For  $\theta \in [0,1]$, $x\geq 0$, define
\begin{equation}
\label{def-Xi}
\Xi(\theta,x)=\left\{\begin{array}{ll}
1+\theta(\mu  x-1) &\mbox{if $r=1$}\\
1+\theta \left(\frac{e^{(r-1)\mu x}}{r-1}-\beta\right)&\mbox{if $r\not=1$.}
                       \end{array}
              \right.
\end{equation}
By specializing $Z(q,x,v)$ in (\ref{eq:Z-def-with-rho})  to the case where $g_i(x)=\mu_i e^{-\mu_i x}$ for $i=1,2$,  we obtain from (\ref{eq:rho-g2-exp}) that

\begin{equation}
\label{value-Z-exp}
Z(q,x,v) = \Xi(q e^{-\mu v},x),\quad \forall q\in [0,1],\,x\geq 0, v\geq 0.
\end{equation}
One the other hand, by (\ref{eq:Ztilde}) and the fact that $\Xi(\theta,x)$ is linear in $\theta$,
\begin{equation}
\label{value-tildeZ-exp}
\widetilde Z(q,x) = \int_0^\infty \lambda e^{-\lambda v} \Xi(q e^{-\mu v},x) dv = \Xi(\int_0^\infty q\lambda e^{-(\lambda+\mu) v} dv,x) = \Xi(pq,x),\quad \forall q\in [0,1],\,x\geq 0,
\end{equation}
where we have used that  $p=\lambda/(\mu_2+\lambda)$  (see (\ref{eq:interference-prob})) when Alice job service times are exponentially distributed.

For $\theta \in [0,1]$, define 
\begin{equation}
\label{def:xir}
\xi_r(\theta):= \frac{(\beta-1)\theta}{1-\beta\theta },
\end{equation}
for $r\geq 2$ (i.e. $1<\beta\leq 2$),  and 
\begin{equation}
\label{def:Ibeta}
I_{\beta}:=\beta \int_0^\infty \frac{1+\frac{1}{2}t-\sqrt{t+1}}{t^{\beta+1}}dt,
\end{equation}
for $r>2$. Since $\beta\in (1,2)$ when $r>2$, the generalized integral $I_\beta$ is finite and positive.

\begin{lemma}
\label{lem:expansion-Y} 
For $\theta \in [0,1]$, define 
\begin{equation}
F_r(\theta):=\left\{ \begin{array}{ll}
\frac{1-r}{4(r-2)}\theta+o(\theta)^2) &\mbox{if $0<r<1$}\\
 \frac{1}{4}  \xi_2^2(\theta)\log \xi_2(\theta)+ \Delta_2(\xi_2(\theta)) &\mbox{if $r=2$}\\
-  I_\beta \xi_r^\beta(\theta)+ \Delta_r(\xi_r(\theta)) &\mbox{if $r>2$,}
\end{array}
\right.
\label{def-Fr}
\end{equation}
where, for $t>0$,
\begin{equation}
\Delta_r(t):=\left\{ \begin{array}{ll}
o(t^2 \log t) &\mbox{if $r=2$}\\
o(t^\beta) &\mbox{if $r>2$.}
\end{array}
\right.
\label{def-Gr}
\end{equation}
Then, for $r\in (0,1)\cup [2,\infty)$, 
\begin{equation}
\E\left[\sqrt{ \Xi(\theta,X_r)}\right]= 1+ F_r(\theta),\quad 0\leq \theta\leq 1.
\label{expansion-H}
\end{equation}
\end{lemma}
The proof of Lemma \ref{lem:expansion-Y} is given in Appendix \ref{app:proof-key-lemma}.\\

Since $\xi_r(\theta)$ defined in (\ref{def:xir}) will only be evaluated at $\theta=qe^{-\mu v}$ with $q=\delta/\phi(n)$ and $\delta\in [0,1]$,  thereby yielding
 $\xi_r(\theta)=\frac{(\beta-1) \delta}{e^{\mu v}\phi(n)- \beta \delta}$, we omit the argument $\theta$ in  $\xi_r(\theta)$ to simplify notation.\\

We are now ready to prove Proposition \ref{prop:phase-transition-exponential}. Recall that at the beginning of an idle period, Alice inserts a job  with
probability $q(n)=\frac{\delta}{\phi(n)}$, $\delta\in (0,1]$, with 
\[
\lim_n \phi(n)=\infty.
\label{hyp-phi}
\]
Henceforth we drop the argument $n$ in $q(n)$. Furthermore 
\begin{equation}
\label{Tn-delta}
T(n)=nq = \frac{\delta n}{\phi(n)}
\end{equation}
is the expected number of Alice's insertions in $n$ W-BPs.

\subsection{Proof of \eqref{eq:achievability-IEBP-exp}}
We assume that Willie uses an optimal detector for the sufficient statistic $\{(Y_j,V_j)\}_j$, which allows us to apply Corollary \ref{cor:covert-non_covert}.

\subsubsection{ Case $\mu_1<2\mu_2$}  

The proof follows from \prettyref{prop:achievability-IEBP} since $\E[\rho(X,V)]=0$ when Alice and Willie job service times are exponentially distributed (cf. \prettyref{rmk:rho-exp})
and since 
$C_0<\infty$ when $\mu_1<2\mu_2$, as shown in Lemma \ref{lem:C-different-cdf}-(1).\\

\subsubsection{Case $\mu_1=2\mu_2$}  

Without loss of generality we assume in this section that 
\begin{equation}
\phi(n)\geq 8,\quad  \forall n\geq 1.
\label{hyp-phi}
\end{equation}
This assumption is motivated by the need to have $\log \phi(n)>2$ (for the proof of Lemma \ref{lem:nDn} in Appendix \ref{app:Dn}).

Recall that $X_2$ denotes an exponential rv with rate $2\mu$. By Lemma \ref{lem:expansion-Y}, 
\begin{eqnarray}
\left(\E\left[\sqrt{Z(\delta/\phi(n),X_2,V)}\right]\right)^n
&=&\left(1+\int_0^\infty \lambda e^{-\lambda v} \left(\frac{1}{4}  \xi_2^2\log \xi_2+ \Delta_2(\xi_2)\right) dv\right)^n\nonumber\\
&=&e^{n\log\left(1+ \int_0^\infty \lambda e^{-\lambda v} \,\xi_2^2\log \xi_2\times \left(\frac{1}{4}  + \frac{\Delta_2(\xi_2)}{\xi_2^2\log \xi_2}\right) dv\right)},
\label{Y-part-b0}
\end{eqnarray}
with $\Delta_2(z)=o(z^2\log z)$ and $\xi_2=\frac{\delta e^{-\mu v}}{\phi(n)-2 \delta e^{-\mu v}}>0$ for all $n\geq 1$ and for all $v\geq 0$ thanks to (\ref{hyp-phi}).
For $v\geq 0$ notice that $\xi_2>0$ for all $n$ and 
$\xi_2\to 0$ as $n\to\infty$.
Define
\begin{equation}
\label{def-Dn}
D_n:=\int_0^\infty \lambda e^{-(\lambda+2\mu) v} \frac{\log \xi_2}{(\phi(n)-2\delta e^{-\mu v})^2}  \left(\frac{1}{4}  + \frac{\Delta_2(\xi_2)}{\xi_2^2\log \xi_2}\right) dv,
\end{equation}
so that (\ref{Y-part-b0}) rewrites
\begin{equation}
\label{Y-part-b1}
\left(\E\left[\sqrt{Z(\delta/\phi(n),X_2,V)}\right]\right)^n =e^{n\log(1+\delta^2 D_n)}. 
\end{equation}
The proof of \eqref{eq:achievability-IEBP-exp} for $\mu_1=2\mu_2$ consists in showing that, as $n\to \infty$, the r.h.s. of (\ref{Y-part-b1})
can be made arbitrary close to one by selecting  $\delta$ small enough, and to apply (\ref{cor:covert}) in Corollary \ref{cor:covert-non_covert}.

The first step is to show that $D_n\to 0$ as $n\to \infty$. This result is shown in Lemma \ref{lem:Dn} in Appendix \ref{app:Dn}. Hence,
\begin{eqnarray}
\left(\E\left[\sqrt{Z(\delta/\phi(n), X_2,V)}\right]\right)^n&\sim_n &e^{\delta^2 n D_n}
\label{liminf-H}
\end{eqnarray}
from (\ref{Y-part-b1}). The second step is to show that $nD_n$ is bounded as $n\to\infty$.  This result is shown in Lemma \ref{lem:nDn}  in Appendix \ref{app:Dn}
under the condition that $T(n)=\mathcal{O}(\sqrt{n/\log n})$.\\

 The proof is concluded as follows: when  $T(n)=\mathcal{O}(\sqrt{n/\log n})$, by (\ref{liminf-H}) and Lemma \ref{lem:nDn}, 
 as $n\to\infty$,$ \left(\E\left[\sqrt{Z(\delta/\phi(n), X_2,V)}\right]\right)^n$ can be made arbitrarily close to one by taking 
$\delta$ small enough. The proof of  \eqref{eq:achievability-IEBP-exp} for $\mu_1=2\mu_2$ then follows from (\ref{cor:covert}) in Corollary \ref{cor:covert-non_covert}.

\subsubsection{Case $\mu_1>2\mu_2$} 
Fix $r>2$ so that $1<\beta<2$. 
By Lemma \ref{lem:expansion-Y}
\begin{eqnarray}
\left(\E\left[\sqrt{Z(\delta/\phi(n),X_r,V)}\right]\right)^n
&=&e^{n\log \left(1+\int_0^\infty\lambda e^{-\lambda v} \left(-  I_\beta \xi_r^\beta + \Delta_r(\xi_r)\right) dv\right)}\nonumber\\
&=& e^{n\log(1+\delta^\beta (\beta-1)^\beta E_n)},
\label{Y-part-c}
\end{eqnarray}
with $\Delta_r(z)=o(z^\beta)$, $\xi_r=\frac{ \delta (\beta-1)}{e^{\mu v}\phi(n)-\beta \delta}$, and
\begin{equation}
\label{def-En}
E_n:=-\int_0^\infty\lambda e^{-\lambda v} \frac{I_\beta- \Delta_r(\xi_r)/\xi_r^\beta}{(e^{\mu \beta}\phi(n)-\delta \beta)^\beta}  dv.
\end{equation}

Lemma \ref{lem:En} in Appendix \ref{app:Dn} states that $E_n\to 0$ as $n\to\infty$ and Lemma \ref{lem:nEn} in  Appendix \ref{app:Dn} states that 
$nE_n$ is bounded as $n\to\infty$ when $T(n)=\mathcal{O}(n^{\mu_2/\mu_1})$.
Therefore, cf (\ref{Y-part-c}),
\begin{equation}
\left(\E\left[\sqrt{Z(\delta/\phi(n),X_r,V)}\right]\right)^n \sim_n e^{\delta^\beta (\beta-1)^\beta n E_n},
\label{liminf-600}
\end{equation}
and when $T(n)=\mathcal{O}(n^{\mu_2/\mu_1})$ as $n\to\infty$  the r.h.s. of (\ref{liminf-600}) can be made arbitrarily close to one by selecting $\delta$ small enough.
The proof of \eqref{eq:achievability-IEBP-exp} for $\mu_1>2\mu_2$ then follows from (\ref{cor:covert}) in Corollary \ref{cor:covert-non_covert}.

\subsection{ Proof of \eqref{eq:converse-IEBP-exp}}
We assume that Willie uses an optimal detector for the statistic $\{Y_j\}_j$, which will allows us to use the non-covert criterion in  Corollary \ref{cor:non_covert-Y}.
Since the proofs in Sections \ref{case1}-\ref{case3}  will not depend on $\delta\in (0,1]$, we assume without loss of generality that  $\delta=1$, yielding  $q=\frac{1}{\phi(n)}$ and $T(n)=\frac{n}{\phi(n)}$.

\subsubsection{Case $\mu_1<\mu_2$}
\label{case1}
Assume that $T(n)=\omega(\sqrt{n})$, or equivalently
\begin{equation}
\label{assump-case-converse-a}
\lim_{n\to \infty} \frac{n}{\phi(n)^2}=\infty.
\end{equation}
Assume first that $0<r<1$.  From (\ref{value-tildeZ-exp}) and Lemma \ref{lem:expansion-Y} we obtain
\begin{eqnarray}
\left(\E\left[\sqrt{\widetilde Z(p/\phi(n),X_r)}\right]\right)^n&=& 
e^{n\log\left(1+\frac{p^2}{4}\left( \frac{1-r}{r-2}\right) /\phi(n)^2+o(1/\phi(n)^2)\right)}\nonumber\\
&& \sim_n  e^{\frac{n}{\phi(n)^2}\frac{p^2 }{4}\left(\frac{1-r}{r-2}\right)}\quad \hbox{as } \phi(n)\to\infty \hbox{ as } n\to\infty\nonumber\\
&&\sim_n 0,  \label{limit-tildeZ-a}
\end{eqnarray}
where the latter follows from (\ref{assump-case-converse-a}) together with  $\frac{1-r}{r-2}<0$ when $0<r<1$.
We invoke Corollary \ref{cor:non_covert-Y}  to conclude 
that Alice is not covert when $T(n)=\omega(\sqrt{n})$ and $0<r<1$.
 
 It remains to show that Alice is not covert for $1\leq r<2$ when $T(n)=\omega(\sqrt{n})$ with $\lim_n n/\phi(n)^2=\infty$.
 Without any additional effort, we will prove a stronger result (to be used in the proof of the case $\mu_1 = 2\mu_2$ of \eqref{eq:converse-IEBP-exp}) that Alice is not covert  when  $T(n)=\omega(\sqrt{n})$ and $r\geq 1$.
 By applying Lemma \ref{lem:stochastic-ordering} in Appendix \ref{app:B} to (\ref{value-tildeZ-exp}), we obtain
  \begin{equation}
 \label{inq-proof-a}
\E\left[\sqrt{ \widetilde Z(p/\phi(n), X_{r'})  }\right]\leq \E\left[\sqrt{\widetilde Z(p/\phi(n), X_r)}\right]
\end{equation}
for any $r'\geq r$.  Combining now (\ref{inq-proof-a}) and (\ref{limit-tildeZ-a}) readily yields
\[
\lim_{n\to\infty }\left(\E[\sqrt{\widetilde Z(p/\phi(n), X_{r'})}\right)^n =0
\]
for any $r'\geq 1$. Similarly to the case $0<r<1$ we then conclude from Corollary \ref{cor:non_covert-Y}  that Alice is not covert 
$T(n)=\omega(\sqrt{n})$ and $r\geq 1$. In summary, we have shown that  Alice is not covert  for all $r>0$ when $T(n)=\omega(\sqrt{n})$.\\

\subsubsection{Case $\mu_1=2\mu_2$} 
\label{case2}
Assume that $T(n)=\omega(\sqrt{n/\log n})$ or, equivalently, 
\begin{equation}
\lim_n \frac{\phi(n)}{\sqrt{n\log n}}=0.
\label{limit-phi-proof-b}
\end{equation}
From (\ref{value-tildeZ-exp}) and Lemma \ref{lem:expansion-Y} 
\begin{equation}
\left(\E\left[\sqrt{\widetilde Z(p/\phi(n),X_2)}\right]\right)^n= e^{n\log\left(1 + \frac{1}{4}  \xi_2^2\log \xi_2 + o(\xi_2^2 \log \xi_2)\right)},
\label{expansion-tildeZ-proof-b}
\end{equation}
with $\xi_2=\frac{p}{\phi(n)-2p}$. 
Since $\xi_2\sim_n 0$ when $\lim_{n}\phi(n)=\infty$,  we have 
\[
\xi_2\log \xi_2 \to 0 \quad \hbox{as } n\to \infty. 
\]
Therefore, from (\ref{expansion-tildeZ-proof-b}),
\begin{equation}
\left(\E\left[\sqrt{\widetilde Z(p/\phi(n),X_2)}\right]\right)^n\sim_n  e^{\frac{n}{4}  \xi_2^2\log \xi_2}.
\label{expansion-2-tildeZ-proof-b}
\end{equation}
We have proved in the case $\mu_1 < 2\mu_2$ of \eqref{eq:converse-IEBP-exp} that Alice is not covert  for all $r>0$ when $T(n)=\omega(\sqrt{n})$. As a result, it suffices to focus
on $T(n)$ satisfying (\ref{limit-phi-proof-b}) when $T(n)\not= \omega(\sqrt{n})$. The latter is equivalent to $\phi(n)=\Omega(\sqrt{n})$, that is,
\begin{equation}
\liminf_n \frac{\phi(n)}{\sqrt{n}}>0.
\label{limit2-phi-proof-b}
\end{equation}
We have
\begin{eqnarray}
n \xi_2^2 \log \xi_2 &=&\frac{p^2}{\left(\frac{\phi(n)}{\sqrt{n\log n}}-\frac{2p}{\sqrt{n\log n}}\right)^2}
\left(\frac{\log p}{\log n} - \frac{\log(\phi(n)-2p)}{\log n}\right)\nonumber\\
&\sim_n& \frac{-p^2}{\left(\frac{\phi(n)}{\sqrt{n\log n}}-\frac{2p}{\sqrt{n\log n}}\right)^2}
\times \frac{\log(\phi(n)-2p)}{\log n}.
\label{xi2-log-xi2}
\end{eqnarray}
 By (\ref{limit-phi-proof-b}) the first factor in the r.h.s. of (\ref{xi2-log-xi2}) converges to $-\infty$ as $n\to\infty$. 
Let us focus on the second factor. We have
\begin{eqnarray*}
 \frac{\log(\phi(n)-2p)}{\log n} = \frac{1}{2}+ \frac{\log\left(\frac{\phi(n)-2p}{\sqrt{n}}\right)}{\log n}\sim_n  \frac{1}{2}+ \frac{\log\left(\frac{\phi(n)}{\sqrt{n}}\right)}{\log n}.
\end{eqnarray*}
Assumption (\ref{limit2-phi-proof-b}) ensures that $\frac{\log\left(\frac{\phi(n)}{\sqrt{n}}\right)}{\log n}\to 0$ as $n\to\infty$ and
\[
 \frac{\log(\phi(n)-2p)}{\log n} \to \frac{1}{2}\quad \hbox{as }n \to\infty.
 \]
In summary, we have shown that $n \xi_2^2 \log \xi_2\to-\infty$ as $n\to\infty$ which, in turn, implies from (\ref{expansion-2-tildeZ-proof-b}) that
$\lim_n \left(\E\left[\sqrt{\widetilde Z(p/\phi(n),X_2)}\right]\right)^n=0$. We conclude from Corollary \ref{cor:non_covert-Y} 
that Alice is not covert if $r=2$ and $T(n)=\omega(\sqrt{n/\log n})$.\\

\subsubsection{Case $\mu_1>2\mu_2$} 
\label{case3}
Assume that $T(n)=n/\phi(n)=\omega(n^{\mu_2/\mu_1})$ or, equivalently, 
\begin{equation}
\lim_{n} \frac{\phi(n)}{n^\beta}=0.
\label{limit-n-beta}
\end{equation}
Let $r>2$ so that $\beta \in (1,2)$.
From (\ref{expansion-H}),
\begin{eqnarray}
\left(\E\left[\sqrt{\widetilde Z(p/\phi(n),X_r)}\right]\right)^n &=& e^{n\log \left(1-  I_\beta \xi_r^\beta + o(\xi_r^\beta)\right)}\nonumber\\
&\sim_n& e^{-  I_\beta n \xi_r^\beta },
\label{tildeZ-c}
\end{eqnarray}
since $\xi_r=\frac{(\beta-1)p}{\phi(n)-\beta p} \to 0$ when $\lim_n\phi(n)=\infty$. We have
\begin{eqnarray*}
n\xi_r^\beta  =\frac{(\beta-1)p)^\beta}{\left(\frac{\phi(n)}{n^\beta}-\frac{\beta p}{n^\beta}\right)^\beta}\to +\infty \quad\hbox{as }n\to\infty.
\end{eqnarray*}
Introducing the above limit in (\ref{tildeZ-c}) and using the finiteness and positiveness of $I_\beta$ for $\beta\in (1,2)$, gives
$\lim_n \left(\E\left[\sqrt{\widetilde Z(p/\phi(n),X_r)}\right]\right)^n=0$, which shows by using again 
Corollary \ref{cor:non_covert-Y}  that Alice is not covert if $r>2$ and $T(n)=\omega({n^{\mu_2/\mu_1}})$. \\

This concludes the proof of Proposition \ref{prop:phase-transition-exponential}.


\section{Insert-at-Idle policy} 
\label{sec:II}

In this section, we consider the variant of the IEBP policy where each time the server idles, Alice inserts a job with probability $q$ and stops with probability $\bar q$ (before she tries again at the end of a new W-BP). We call this policy the {\em Insert-at-Idle (II)} policy.
The difference between the IEBP and II policies is that under the former Alice may only insert one job between the end of a W-BP and the start of the next W-BP, whereas
under the II policy she may insert more than one job during this time period.

We will show that when Alice job service times are exponentially distributed all covert/non-covert results obtained under the IEBP policy hold under the II policy. The intuition behind this
is that when Alice job service times are exponentially distributed, Willie sees "the same system behavior" under either policy; indeed,  under either policy a job of his can interfere with at most one Alice job in a W-BP, whose remaining service time is exponentially distributed.

Throughout this section quantities with the subscript ``$+$'' refer to the II policy.
Let $Y_{+,j}$ be the reconstructed service time of the first Willie job in the $j$-th W-BP, and  $V_{+,j}$ the duration of the idle period between the $(j-1)$-th and the $j$-th W-BPs. 
The rvs $(Y_{+,j},V_{+,j})$, $j=1,\ldots,n$, are iid, and we denote by $(Y_+,V_+)$ a generic element with the same distribution.

The argument in \prettyref{sec:IEBP} to prove that $\{Y_j,V_j\}_j$ is a sufficient statistics under the IEBP policy can be reproduced 
to argue that $\{Y_{+,j},V_{+,j}\}_j$ is a sufficient statistics; this is the case, as, similar to the IEBP policy, only the first Wille job in a W-BP may interfere with 
an Alice job under the II policy.

Introduce $f_{+,i}$  the pdf of $(Y_+,V_+)$ under $H_i$ for $i=0,1$, so that the joint pdf of $\{(Y_{+,j},V_{+,j})\}_{j=1}^n$ under $H_i$ is
$f_{+,i}^{\otimes n}$.
Similarly, let $\widetilde f_{+,i}$ be the pdf of $Y_+$ under $H_i$, $i=0,1$, with $\widetilde f^{\otimes n}_{+,i}$  the pdf under $H_i$ of the iid rvs $\{Y_{+,j}\}_{j=1}^n$.

Clearly, $f_{+,0}(x,v)=f_0(x,v)=\lambda e^{-\lambda v}g_1(x) $ (cf.  \eqref{value-f0}) and $\widetilde f_{+,0}(x)=\widetilde f_0 (x)= g_1(x)$.

The following lemma is proved in Appendix \ref{app:D0}.
\begin{lemma}[pdfs $f_{+,1}$ and $\widetilde f_{+,1}$ under the II$_+$ policy]\hfill
\label{lem:II}

For any pdf $g_1$ and $g_2(x)=\mu_2 e^{-\mu_2 x}$,
\begin{eqnarray}
f_{+,1}(x,v)&=& \lambda e^{-\lambda v} g_1(x) \left[1+ qe^{-\mu_2\bar q v}\left( \frac{(g_1*g_2)(x)}{g_1(x)}-1\right)\right]\label{lem:f+1xv}\\
\widetilde f_{+,1}(x)&=&  g_1(x) \left[1+ \frac{pq}{1-\bar p q}\left(\frac{(g_1*g_2)(x)}{g_1(x)}-1\right)\right],\label{lem:f+1x}
\end{eqnarray}
for all $x\geq 0$, $v\geq 0$.
\end{lemma}
From now on $g_2(x)=\mu_2 e^{-\mu_2 x}$. Recall that $X$ is a rv with pdf $g_1$ and $V$ is an exponential rv with rate $\lambda$, independent of $X$.

By replacing $f_i$ by $f_{+,i}$, $i=0,1$, in the derivation of (\ref{inq-lem}) in the proof of Lemma \ref{lem:up-TV},  we obtain 
\begin{eqnarray}
\lefteqn{(2T_V)^2\left(f^{\otimes n}_{+,0},f^{\otimes n}_{+,1}\right)\leq
1-2\left(\E\left[\frac{f_{+,1}(X,V)}{f_{+,0}(X,V)}\right]\right)^n +\left(\E\left[\frac{f_{+,1}(X,V)^2}{f_{+,0}(X,V)^2}\right]\right)^n}\nonumber\\
&=& 1-2\left(1+\E\left[e^{-\mu_2 \bar q V}\right]\E\left[\frac{g_1*g_2(X)}{g_1(X)}-1\right]\right)^n
+\left(\E\left[\left(1+ qe^{-\mu_2\bar q V}\left( \frac{(g_1*g_2)(X)}{g_1(X)}-1\right)\right)^2\right]\right)^n\nonumber\\
&=&-1+\left(1+q^2 \E\left[ e^{-2\mu_2 \bar q V}\left(\frac{(g_1*g_2)(X)}{g_1(X)}-1\right)^2\right]\right)^n,
\label{ub:TV-variant0}
\end{eqnarray}
where the latter equality follows from the identity $\E\left[\frac{g_1*g_2(X)}{g_1(X)}\right]=\int_0^\infty g_1*g_2(x)dx=1$ and the independence of the rvs $X$ and $V$.
Define $C_1=\E\left[ e^{-2\mu_2 \bar q V}\left(\frac{(g_1*g_2)(X)}{g_1(X)}-1\right)^2\right]$. Inequality (\ref{ub:TV-variant0}) then rewrites
\begin{equation}
T_V\left(f^{\otimes n}_{+,0},f^{\otimes n}_{+,1}\right)\leq \frac{1}{2}\sqrt{(1+q^2C_1)^n-1}, \quad n\geq 1.
\label{ub:TV-variant}
\end{equation}

Let $T_{+,n}$ be the expected number of jobs inserted by Alice over $n$ W-BPs. Observe that
\begin{equation}
\label{def-T+n}
T_+(n)=n\times q(p+\bar p\bar q)\sum_{i\geq 0} (i+1) (\bar p q)^i = \frac{n q}{1-\bar p q}.
\end{equation}
Mimicking now the proof of \prettyref{prop:achievability-IEBP} with \eqref{ub:TV-variant} replacing \eqref{ub:TV}, and $T_+(n)$ in (\ref{def-T+n}) replacing $T(n)=nq$,  we obtain the following covert result:
\begin{proposition} \label{prop:achievability-II}
Assume that $g_2(x)=\mu_2 e^{-\mu_2 x}$ and  $C_1<\infty$.  Under the II policy, Alice can achieve a covert throughput of $T_+(n)=\mcO(\sqrt{n})$.
\end{proposition}
The lemma below gives the Hellinger distances between $f^{\otimes n}_{+,0}$ and $f^{\otimes n}_{+,1}$,  and between $\widetilde f^{\otimes n}_{+,0}$ and $\widetilde f^{\otimes n}_{+,1}$ when 
 Alice and Willie job service times are exponentially distributed.
 \begin{lemma}
 \label{lem-II-2}
 Assume that $g_i(x)=\mu_i e^{-\mu_i x}$, $i=0,1$. Then, for every $n\geq 1$,
 \begin{eqnarray}
 H\left(f^{\otimes n}_{+,0},f^{\otimes n}_{+,1}\right)&=& 1- \left(\E\left[\sqrt{\Xi\left(qe^{-\mu_ 2\bar q V},X\right)}\right]\right)^n
 \label{H*-gen}\\
 H\left(\widetilde f^{\otimes n}_{+,0},\widetilde f^{\otimes n}_{+,1}\right)&=& 1- \left(\E\left[\sqrt{\Xi(p q/(1-\bar p q),X)}\right]\right)^n,
 \label{H*-only}
 \end{eqnarray}
with $X$ an exponential rv with rate $\mu_1$, where the mapping $\Xi$ is defined in (\ref{def-Xi}).
\end{lemma}

For the sake of comparison, recall that under the IEBP policy the  Hellinger distances  corresponding to (\ref{H*-gen}) and (\ref{H*-only}) are given by 
(Hint: introduce (\ref{value-Z-exp}) and (\ref{value-tildeZ-exp}) in (\ref{eq:Hellinger}), respectively)
  \begin{eqnarray}
 H\left(f^{\otimes n}_{0},f^{\otimes n}_{1}\right)&=& 1- \left(\E\left[\sqrt{\Xi\left(qe^{-\mu_ 2 V},X\right)}\right]\right)^n
 \label{sec5:H-gen}\\
 H\left(\widetilde f^{\otimes n}_{0},\widetilde f^{\otimes n}_{1}\right)&=& 1- \left(\E\left[\sqrt{\Xi(p q,X)}\right]\right)^n,
 \label{sec5:H-only}
 \end{eqnarray}
 respectively,  when  $g_i(x)=\mu_i e^{-\mu_i x}$ for $i=0,1$.
 
 We then see that, for $q$ small, $H\left(f^{\otimes n}_{+,0},f^{\otimes n}_{+,1}\right)\sim H\left(f^{\otimes n}_{0},f^{\otimes n}_{1}\right)$ and
 $H\left(\widetilde f^{\otimes n}_{+,0},\widetilde f^{\otimes n}_{+,1}\right) \sim H\left(\widetilde f^{\otimes n}_{0},\widetilde f^{\otimes n}_{1}\right)$, thereby explaining why Proposition \ref{prop:phase-transition-exponential} holds under
 the II policy, as announced earlier (a rigorous proof mimicks the (very lengthy) proof of Proposition \ref{prop:phase-transition-exponential}).\\
 
 In conclusion, as $n\to\infty$, policies IEBP and II behave the same as far as covert/non-covert results are concerned when Alice job service times are exponentially distributed.  This means that Alice should rather use the II policy since the expected number of jobs that she inserts over a {\em finite} number $n$ of W-BPs, given by $nq/(1-\bar p q)$, is larger  under the II policy that it is under the IEBP policy (given by $nq$).


\section{Insert-at-Idle-and-at-Arrivals Policy }
\label{sec:IIA}
Throughout this section we assume that the service times of Willie and Alice are exponentially distributed with rate $\mu_1$ and $\mu_2$, respectively.

We have observed at the beginning of \prettyref{subsec:IEBP} that Alice should preferably inserts jobs at idle times;  this was the motivation for introducing and 
investigating the IEBP policy in \prettyref{sec:IEBP} and its variant, the II policy investigated in   \prettyref{sec:II}.

But can Alice submit more jobs covertly  if she also inserts jobs at other times than at idle times, typically, just after an arrival /departure of a Willie job? 
This is the question we try to answer in this section.
Note that, because of the FIFO assumption, Alice cannot benefit from inserting a job at a time $t+$ if time $t$ is neither an arrival time nor a departure time of a Willie job. 

In this section, we assume that Alice inserts jobs  at idle times {\em and} at arrival times (see Remark \ref{rem:II-D}). More precisely, 
\begin{itemize}
\item
each time the server idles  Alice inserts one job with probability $q$ and does not insert a job  with probability $\bar q$;
\item
after the arrival of each Willie job, Alice inserts a  batch of $s\geq 0$ jobs  with probability $qQ(s)$ and with probability $1-q$ she does not insert any job.
\end{itemize}
These policies are called  {\em Insert-at-Idle-and-at-Arrivals} (II-A) policies. Let ${\mathcal A}$  be the set of all such policies.  A policy in ${\mathcal A}$ is fully characterized
by the pair $(q,Q)$, with $Q$ a pdf with support in $\{0,1,\ldots\}$.
Notice that the II-A policy reduces to the II policy when $Q_B(0)=1$ (no job inserted at arrival times).

For the time being we do not make any assumption on $Q$ (later on we will assume that it has a finite  support). We assume that successive job batch sizes inserted by Alice just after arrival epochs are mutually independent rvs. 

Let 
\[
\mathcal{G}_Q(z)=\sum_{s\geq 0}z^sQ(s)
\]
 be the generating function of $Q$ and denote by $B$ the expected batch size. 
For the time being, the pair $(q,Q)$ is fixed, that is, we focus on a particular policy in ${\mathcal A}$.
Under this policy,  all of Willie jobs are susceptible to interference from Alice.  This is in contrast with the IEBP policy, where only the first of Willie  jobs in a W-BP can be affected by Alice. 
 
Unlike for the IEBP policy  (cf. Proposition \ref{prop:phase-transition-exponential}),  we have not been able to find a sufficient statistics composed of iid rvs.  As a result, we will
only focus on obtaining an upper bound for $T(n)$. To obtain such a result, recall that Willie does not need to work with a sufficient statistics as it is enough for Willie to use a 
detector that prevents Alice from being covert (this argument was used to prove the non-covert part of Proposition \ref{prop:phase-transition-exponential}).

The non-covert result  is stated in  Proposition \ref{prop:non-covert-pi}. We will see that it  gives a loose upper bound since, in particular, it does not reduce to the non-covert
result  obtained under the II policy (see the remark after the proof of Proposition \ref{prop:geometric-exponential}). This is due to the fact that Willie does not use the full information he has about Alice jobs or, equivalently, he does not use a sufficient statistics.
 This said, we conjecture that the results in Proposition \ref{prop:phase-transition-exponential} should hold for all policies in ${\mathcal A}$ and also for a much broader class of policies
 (e.g. stationary policies) provided service times are exponentially distributed. \\

Recall the definition of a {\em Willie Busy Period} (W-BP) and \emph{Willie Idle Period} (W-IP) introduced at the beginning of \prettyref{sec:IEBP}. We call a {\em cycle} the period consisting of a W-IP followed by a W-BP.  
Denote by $N_A$ and $N_W$ the expected number of Alice  jobs and Willie  jobs served during a cycle, respectively.
 \begin{lemma}
\label{lem:NA-NW}
 Under the II-A policy the queue is stable iff $\rho_1+q\rho_2 B<1$. In this case,
  \begin{eqnarray}
 \lefteqn{\E[N_W]=\Biggl[1 +
q\frac{\rho_2  p}{1-q\bar p}\left(\frac{1}{1-q\bar p} - \frac{(1+p)Q(1)} {\bar p}\right)}\label{lem:NW}\\ 
&& +    q^2 \frac{\rho_2}{1-q\bar p}\left(p B-\frac{p}{1-qp}(1-\mathcal{G}_Q(\bar p))\right)\Biggr]\times \frac{1}{1- \rho_1 - q\rho_2B}\nonumber\\
&=&\frac{1}{1-\rho_1} +\frac{\rho_2}{1-\rho_1}\left(p -(2- p) Q(1)-\frac{B}{1-\rho_1}\right) q + o(q),
 \end{eqnarray}
and
\begin{eqnarray}
\E[N_A]&=& qB \E[N_W] + q\,\frac{\bar p \bar q +p}{(1-q\bar p)^2} \left(\bar q \overline{Q}(0)+\mathcal{G}_Q(\bar p) \right)\label{lem:NA}\\
&=&q\left(\frac{B}{1-\rho_1}+\overline{Q}(0)+\mathcal{G}_Q(\bar p)\right)+ o(q).\nonumber
\end{eqnarray}
 \end{lemma}
 The proof is given in Appendix \ref{app:C}. From now on we assume that $q\in [0,q_0)$ with $q_0:=(1-\rho_1)/(\rho_2 B)$
 so that the stability condition $\rho_1+q\rho_2 B<1$ holds (recall that $\rho_1<1$ -- see  \prettyref{sec:IEBP}). In particular, W-BPs have finite expected lengths when $\rho_1+q\rho_2 B<1$.\\
 
To apply the results of \prettyref{sec:IEBP}, Willie needs to come up with a detector built in such a way that the reconstructed service times form an iid sequence.
 To this end, he will use the following detector, hereafter refers to as $D_W$: from each of the first $n$ W-BP, he picks a job uniformly at random and reconstructs its service time.
 Under the enforced assumptions (Poisson arrivals and exponential service times), this detector produces an iid sequence of reconstructed service times. 
 
 We consider a generic cycle and denote by $J\in\{1,\ldots,N_W\}$ the identity of the Willie  job picked at random in the W-BP.  Let $\pi_J:=\P(J=1)$ be the probability that the first job is picked.
 Let $Y$ be the reconstructed service time of the randomly picked job $J$ (recall that in \prettyref{sec:IEBP} $Y$ denotes the reconstructed service time of the first job, corresponding to $\pi_J = 1$. We use the same notation here for the sake
of simplicity, as no confusion should arise).
 
 We have
 \begin{equation}
 \label{y-pi}
 Y=\sigma_1+ \sum_{r=1}^{S_J} \tau_r,
 \end{equation}
 where $\sigma_1$ is a generic service time for a Willie  job, $\tau_1, \ldots,\tau_r$ are the service times of $r$ different Alice' s jobs, and $S_J$ is the number of Alice
 jobs that interferes with $J$. Define
 \begin{equation}
w_i(x):=\frac{d}{dx}\P_{H_i}(Y<x), \quad i=0,1,
  \label{def:Yqx}
 \end{equation}
 the pdf of the Willie job reconstructed service time in a W-BP under $H_i$. Clearly, $w_0(x)=\frac{d}{dx}\P_{H_0}(Y<x)=g_1(x)$.  
 The Hellinger distance between the pdfs $w_0$ and $w_1$ is (cf. \eqref{eq:Hellinger-def})
 \begin{equation}
 \label{Hellinger-II-A}
 H(w_0,w_1)=1- \E\left[\sqrt{W(q,X)}\right],
 \end{equation}
 where $X$ is an exponential rv with parameter $\mu_1$, and
 \begin{equation}
 \label{def:W}
 W(q,x)=\frac{w_1(x)}{g_1(x)}.
  \end{equation}
 The mapping $W$ corresponds to the mapping $Z$ for the IEBP policy (see (\ref{eq:Z-def})).
  The lemma below determines $W(q,x)$. The proof is provided in Appendix \ref{app:D}.

 \begin{lemma} 
 \label{lem-Yqx}
 For $q\in [0,q_0)$, 
 \begin{equation} 
W(q,x)=\Delta_1(q)+\Delta_2 (q)\Phi_1(x)+q\Phi_2(x), \quad x\geq 0,
\label{lem:eq:Y}
\end{equation}
where
\begin{align}
\Delta_1(q)&:= \frac{\bar q}{1-q\bar p}\left(\bar q + q \mathcal{G}_Q(\bar p)\right)\pi_J +(1-q\overline{Q}(0))\bar \pi_J\geq 0
\label{def:Delta1}\\
\Delta_2(q)&:=q\,\frac{\bar p(1-q\overline{Q}(0)) +\mathcal{G}_Q(\bar p)-Q(0)}{\bar p(1-q\bar p)}\geq 0
\label{def:Delta2}\\
\Phi_1(x) & :=p  \pi_J \frac{(g_1*h_1)(x)}{g_1(x)}
\label{def:Phi1}\\
\Phi_2(x) &:= p \pi_J\sum_{s\geq 2} \frac{(g_1*h_s) (x)}{g_1(x)}\sum_{l\geq s} Q(l)\bar p^{l-s} \nonumber\\
&\quad \quad +  \bar \pi_J \sum_{s\geq 1}\frac{(g_1*h_s)(x)}{g_1(x)} Q(s),
\label{def:Phi2}
\end{align}
where $h_s(x)=\mu_2^{s} x^{s-1}e^{-\mu_2 x}/(s-1)!$ is the pdf of a $s$-stage Erlang rv with mean $k/\mu_2$.
\end{lemma}
  \begin{lemma}
\label{lem-power-series}
Assume that  the support of $Q$ is finite. If $\mu_1<2\mu_2$ there exists a finite constant $c_0$ such that 
\begin{equation}
\label{power-series-sqrt-Y}
\E\left[\sqrt{W(q,X)}\right]=1+c_0 q^2 +o(q^2),
\end{equation}
where $X$ is an exponential rv with rate $\mu_1$.
\end{lemma}
The proof is provided in Appendix \ref{app:E}. Below is the main result of this section.
\begin{proposition}
\label{prop:non-covert-pi}
Assume that the support of $Q$ is finite. 
For all $\mu_1$ and $\mu_2$, Alice is not covert if $T(n)=\omega(\sqrt{n})$.
\end{proposition}
\begin{proof} Without loss of generality assume that the support of $Q$ is contained in $\{0,1,\ldots,S\}$ with $S<\infty$.

Let $Y_1,\ldots,Y_n$ be the Willie reconstructed job service times over $n$ W-BPs. 
Under $H_1$ (resp. $H_0$), the joint pdf of
$Y_1,\ldots,Y_n$ is $w^{\otimes n}_1$ (resp. $w^{\otimes n}_0=g^{\otimes n}_1$) 
since these rvs are iid.  Hence, from \eqref{eq:H-alt}, 
\[
H\left(w_0^{\otimes n},w_1^{\otimes n}\right)=1-\left(\E\left[\sqrt{W(q,X)}\right]\right)^n,\quad \forall n\geq 1,
\]
which in turn gives, by using (\ref{eq:lb-ub}), 
\begin{equation}
\label{inq-W}
1-\left(\E\left[\sqrt{W(q,X)}\right]\right)^n \leq T_V\left(w_0^{\otimes n},w_1^{\otimes n}\right),\quad \forall n\geq 1.
\end{equation}
Let $T_n=\frac{n}{\phi(n)}=\omega(\sqrt{n})$ or, equivalently, $\lim_{n\to\infty} \frac{\phi(n)^2}{n}=0$. Upon replacing $\widetilde f_1(x)$ by $w_1(x)$, the same argument in the proof of  
\prettyref{prop:converse-IEBP} shows that Alice is not covert if $\liminf_n \phi(n)<\infty$. Therefore, we assume from now on that $\lim_n \phi(n)=\infty$.

Assume first that $\mu_1<2 \mu_2$.  By Lemma \ref{lem-power-series}, we conclude that 
\begin{equation}
\label{converse}
\lim_{n\to\infty}\left( \E\left[\sqrt{W(1/\phi(n),X))}\right]\right)^n =0
\end{equation}
 which proves, thanks to (\ref{inq-W}) and  (\ref{covert-criterion2}), that Alice is not covert if $T(n)=\omega(\sqrt{n})$ and $\mu_1<2 \mu_2$.

Let us show that (\ref{converse}) holds when $\mu_1\geq 2\mu_2$ which will complete the proof.
From (\ref{g1*hs}), we see that for each $s=1,\ldots,S$,  the mapping $x\to (g_1*h_s)(x)/g_1(x)$  is non-decreasing in $[0,\infty)$. On the other hand,
notice that both sums in (\ref{def:Phi2}) are finite under the assumption that $Q$ has a finite support, and  observe that each term $(g_1*h_s)(x)/g_1(x)$, $s=1,\ldots,S$, is multiplied
by a non-negative constant. Therefore,  the mappings $x\to \Phi_i(x)$, $i=1,2$, are non-decreasing in $[0,\infty)$, which in turn
shows that the mapping $x\to W(q,x)$ (given in (\ref{lem:eq:Y}))  is non-decreasing in $[0,\infty)$ for all $q\in [0,1]$, since $\Delta_i(q)\geq 0$ for $i=1,2$ (Hint: in (\ref{def:Delta2})
$\mathcal{G}_Q(\bar p)\geq Q(0)$ by definition of the generating function $\mathcal{G}_Q$).\\

Let $\widetilde X_\nu$ be an exponential rv with rate $\nu$. Take  $\nu_1$ such that $\mu_1 \geq  2\mu_2> \nu_1$. The stochastic inequality
$X=\widetilde X_{\mu_1}\leq_{st} \widetilde X_{\nu_1}$
together with increasingness (shown above) of the  mapping $x\to W(q,x)$ for $q\in [0,1]$ , yield
\begin{equation}
\label{inq:X}
 \E\left[\sqrt{W(q,X)}\right] \leq \E\left[\sqrt{W(q,\widetilde X_{\nu_1})}\right], \quad \forall q\in [0,1].
\end{equation}
Then,  (\ref{converse})  and (\ref{inq:X})  imply
\[
\lim_{n\to\infty}\left(\E\left[\sqrt{W(1/\phi(n), X}\,\right]\right)^n=0
\]
when $T(n)=\omega(\sqrt{n})$. This completes the proof.
\end{proof}
In general, the asymptotic upper bound in Proposition \ref{prop:non-covert-pi} is loose as Willie's detector lacks of information (cf. discussion at the beginning of this section).
This is the case  when $Q_B(0)=1$ (i.e. the II-A policy reduces to the II policy)  as the bound is larger than the bound for $\mu_1=2\mu_2$ ($\omega(\sqrt{n/\log n})$)
and for  $\mu_1>2\mu_2$ ($\omega(n^{\mu_2/\mu_1})$) under the II policy (see \prettyref{sec:II}). 

Last, we consider a variant of the II-A policies where Alice inserts a batch of jobs that is geometrically distributed with mean $1/a$ at times the server becomes idle and immediately after the arrival of Willie  job, both with probability $q$.  We have the following result:

\begin{proposition}
\label{prop:geometric-exponential}
Assume that $g_i(x)=\mu_i  e^{-\mu_i x}$ for $i=1,2$. When Willies uses detector $D_W$,
Alice is not covert if she inserts \begin{itemize}
\item[(a)] $\omega(\sqrt{n})$ jobs when $\mu_1<2a \mu_2$
\item[(b)] $\omega(\sqrt{n/\log n})$  jobs when $\mu_1=2a\mu_2$
\item[(c)] $\omega(n^{a\mu_2/\mu_1})$ jobs when $\mu_1>2a\mu_2$
\end{itemize}
on average over $n$ W-BPs.
\end{proposition}

\begin{proof}  Note that each  batch of Alice  jobs incurs a total amount of service time that is exponentially distributed with rate $a\mu_2$. This coupled with Willie  detector produces a pdf for the hypothesis $H_1$ of the form (\ref{eq:Z-def}) for some $p > \lambda/(\lambda+a\mu)$ in (\ref{eq:rho}). The arguments leading to the converse in Proposition \ref{prop:phase-transition-exponential} apply to this case to yield the desired result. 
\end{proof}

\begin{remark}[Geometric batch size]
Note that geometric batching provably reduces covert throughput in the range $2a\mu_2 \le \mu_1 <2\mu_2$ under a variant of the II-A policy using batches with finite support.  This appears to be due to the exponential tail.  We conjecture that batches of size greater than one can only reduce covert throughput.
A similar result holds for a variant of the IEBP policy where Alice introduces a batch of jobs with probability $q$ each time the server becomes idle where the batch is geometrically distributed with mean $1/a$ leading to a considerably smaller covert throughput than is possible when Alice introduces only one job at a time.  This is evidence that batching again may be harmful and that Alice should introduce only one job at a time. 
\end{remark}

\begin{remark}[Insert-at-Idle-and-at-Departure]\label{rem:II-D}
The analysis of the policy, called II-D, where Alice may insert a job each time the server idles and may also insert a batch of jobs after each Willie job departure (provided the system is not empty)  is more involved than that of the II-A policy. This is so because the reconstructed service time of a Willie job in a W-BP depends on what happened in this busy period  prior to the arrival  of this job. To illustrate this, assume first that the $j$th Willie job ($j>1$) in a W-BP arrives during the service time of the 1st Willie job in this W-BP. Then, job $j$ will not be affected by any Alice's insertions in this W-BP. But if job $j$ arrives during the service time
of the $(j-1)$st Willie job then it may be affected by $0,1$ or up to $j-2$ Alice batches, depending on how many batches Alice insert at departures of Wille jobs $1,2,\ldots, j-2$.
This is in contrast with the II-A policy, where job $j>1$ in a W-BP will be affected by at most one Alice's batch (the batch 
inserted after the arrival of customer $j-1$, if any).
\end{remark}


\section{Concluding remarks}
\label{sec:conclusion}
In this paper we have studied covert cycle stealing in an M/G/1 queue. We have obtained  a phase transition result on the expected number of
jobs that Alice can covertly insert in $n$ busy periods when both Alice and Willie's  jobs have exponential service times and established partial covert results for arbitrary service times.   
Several research directions present themselves.  We conjecture that \prettyref{prop:phase-transition-exponential} holds for a more general class of distributions; it would be interesting to verify this. It would be useful to weaken the assumption that Willie's detectors rely on observations being independent and identically distributed random variables; this would lead to consideration of a larger class of policies on Alice's behalf.  Another direction would be to allow Alice to control her job sizes and study what benefit this would provide her. Yet another is to consider other hypothesis testing techniques including generalized likelihood ratio test (GLRT), sequential detection, etc.  GLRT could lead to relaxing the need for Willie to know Alice's parameters whereas sequential detection could lead to more timely detection of Alice.

\bibliographystyle{plain}
\bibliography{arxiv3-bib}

\appendix

\section{Appendix: Proof of Lemma \ref{lemma:likelihood}}
\label{app:sufficient-statistics}

Let $p_{i,n}(w_{1:n})$ be the pdf of $W_{1:n}=(W_1,\ldots,W_n)$ at $w_{1:n}=(w_1,\ldots,w_n)$ under $H_i$ for $i=0,1$.  
Also let $\widetilde p_{i,j}(w_j)$ be the pdf\,\footnote{Clearly $\widetilde p_{i,j}(w_j)=\int_{\R^{n-1}} p_{i,n}(w_{1:n})dw_1\cdots dw_{j-1} dw_{j+1} \cdots dw_n$.}
 of $W_j$ at $w_j$ under $H_i$ for $i=0,1$. 
Note that $p_{i,n}$ (resp. $\widetilde p_{i,j}$) is a  generalized pdf since $W_{1:n}$ (resp. $W_j$) contains integer and continuous components.

By the general multiplicative formula,
\begin{equation}
\label{piWn}
\widetilde p_{i,n}(w_{1:n})=\widetilde p_{i,1}(w_1)\times \prod_{j=2}^n \widetilde p_{i,j}(w_j\,|\, W_{1:j-1}=w_{1:j-1}).
\end{equation}
Let $w_j=(m_j, a_{(m_{j-1}+1):m_j}, s_{(m_{j-1}+1):m_j})$, $y_j=s_{m_{j-1}+1}$, and $v_j=a_{m_{j-1}+1}-d_{m_{j-1}}$.
We have
\[
\widetilde p_{i,j}(w_j\,|\, W_{1:j-1}=w_{1:j-1})=\widetilde p_{i,j}(w_j\,|\, A_{m_{j-1}}=a_{m_{j-1}}, D_{m_{j-1}}=d_{m_{j-1}},A_{m_{j-1}+1}>d_{m_{j-1}}),
\]
since the probability distribution of the number of customers served in a busy period in an M/G/1 queue is entirely determined 
once we know the duration of the first service time in this busy period \cite[Chapter 5.9]{Kleinrock75}. Hence, 
\begin{eqnarray}
\lefteqn{\widetilde p_{i,j}(w_j\,|\, W_{1:j-1}=w_{1:j-1}) ={\bf 1}\left(a_{m_{j-1}+1}>d_{m_{j-1}}\right)f_i(y_j, v_j)}\nonumber\\
&&\times p\left(M_j=m_j, A_{(m_{j-1}+2):m_j}=a_{(m_{j-1}+2):m_j}, S_{(m_{j-1}+2):m_j}=s_{(m_{j-1}+2):m_j}\,|\, Y_j=y_j, V_j=v_j\right),
\label{piWj}
\end{eqnarray}
where the latter density is independent of $H_0$ and $H_1$. The pdf $\widetilde p_{i,1}(w_1)$ is given by the r.h.s. of (\ref{piWj}) by letting $j=1$. Putting (\ref{piWn}) and
 (\ref{piWj})  together yields the factorization result
 \begin{equation}
 \label{piWn-value}
 p_{i,n}(w_{1:n})=\prod_{j=1}^n f_j(y_j,v_j)\times \hbox{other factors independent of $H_0$ and $H_1$},
 \end{equation}
 which proves that $(Y_{1:n},V_{1:n})$ is a sufficient statistic \cite[Chapter 1.9]{LR05}.


\section{Appendix}
\label{app:Zqxv}
Recall that $Z(q,x,v)=\frac{f_1(x,v)}{f_0(x,v)}$, with $f_i(x,v)$ the pdf of $(Y,V)$ at $(x,v)$ under $H_i$ for $i=0,1$.

\begin{lemma}
\label{lem:Z-rho}
\begin{equation} \label{eq:Z-rho}
Z(q,x,v)=1+q\rho(x,v),  
\end{equation}
where 
\begin{equation}
\label{eq:rho2}
\rho(x,v):=\frac{1}{g_1(x)}\int_0^x g_1(u)g_2(v+x-u)du -\overline G_2(v).
\end{equation}
 \end{lemma}

\begin{proof}
Consider a generic W-BP. Let $\sigma_1$ (resp. $\sigma_2$) denote a generic service time of a Willie (resp. Alice) job. Let $A$ be the event that Alice inserts a job at the end of the W-BP. 
Then
\[
Y = \sigma_1 + \ind{A}\cdot (\sigma_2 - V)^+,
\]
where $(z)^+ = \max\{z,0\}$. We first compute the conditional density $f_1(x\mid v)$ of $Y$ given $V$. Given $A^C$, $Y=\sigma_1$, so 
\[
f_1(x\mid v, A^c) = g_1(x).
\]
Given $A$ and $V=v$, we have $Y = \sigma_1 + (\sigma_2-v)^+$, so that
\[
f_1(x\mid v,A) = g_1(x) G_2(v) + \int_0^x g_1(u) g_2(x+v-u) du.
\]
Recall the probability of $A$ under $H_1$ is $q$,  so
\begin{align}
f_1(x\mid v) & = q f_1(x\mid v,A) +  \bar q f_1(x\mid v, A^c)\nonumber\\
&=  g_1(x) +  q \left[\int_0^x g_1(u) g_2(x+v-u) du - g_1(x) \bar G_2(v)\right]\nonumber\\
&= g_1(x) [1 + q\rho(x,v)], \label{lem:f1xv}
\end{align}
by using the definition of $\rho(x,v)$ in (\ref{eq:rho2}). Therefore,
\[
Z(q,x,v)=\frac{f_1(x\mid v)\lambda e^{-\lambda v}}{f_0(x,v)}=1+q\rho(x,v),
\]
by using (\ref{value-f0}), which concludes the proof.
\end{proof}

\section{Appendix: Proof of Lemma \ref{lem:C-different-cdf}}
\label{app:A}
 
 Let $g_2(x)=\sum_{l=1}^{K_2} p_{2,l} g_{2,l}(x)$ with $g_{2,l}(x):=\mu_{2,l} e^{-\mu_{2,l} x}$, $p_{2,l}\geq 0$ for all $l$ and $\sum_{l=1}^{K_2} p_{2,l}=1$, namely, Alice job service times follow an hyper-exponential distribution with mean $1/\mu_2=\sum_{l=1}^{K_2} 1/\mu_{2,l}$.  Denote by $G^*_1(s)=\int_0^\infty e^{-s x} g_1(x)dx$ the Laplace transform of Willie job service times.\\
 
 By using \eqref{eq:rho}, we find
 \[
 \rho(x,v)=\frac{1}{g_1(x)}\sum_{l=1}^{K_2} p_{2,l} e^{-\mu_{2,l} v} (g_1*g_{2,l})(x) - \sum_{l=1}^{K_2} p_{2,l} e^{-\mu_{2,l}v},
 \]
 so that 
 \[ 
\E\left[ \rho(X,V)^2\right] = \alpha_1 -2\alpha_2 + \alpha_3
 \]
with 
\begin{eqnarray*}
\alpha_1&:=& \int_{[0,\infty)^2} \frac{\lambda e^{-\lambda v}}{g_1(x)} \left[\sum_{l=1}^{K_2} p_{2,l}e^{-\mu_{2,l}v} (g_1*g_{2,l})(x)\right]^2 dv dx\\
&=& \sum_{l=1\atop m=1}^{K_2}\frac{\lambda p_{2,l}p_{2,m}  }{\mu_{2,l}+\mu_{2,m}}\int_0^\infty  \frac{(g_1*g_{2,l})(x)\times (g_1*g_{2,m})(x)}{g_1(x)} dx;\label{def:alpha1}\\
\alpha_2&:=&\sum_{l=1\atop m=1}^{K_2} p_{2,l} p_{2,m}\int_{[0,\infty)^2} \lambda e^{-(\lambda+\mu_{2,l}) v}(g_1*g_{2,l})(x) e^{\mu_{2,m}x}dv dx\\
&=&\sum_{l=1\atop m=1}^{K_2} p_{2,l}p_{2,m}  \frac{\lambda \mu_{2,l}G^*_1(\mu_{2,m})}{(\lambda +\mu_{2,l})(\mu_{2,l}+\mu_{2,m})}\leq 1;\\
\alpha_3&:=&\sum_{l=1\atop m=1}^{K_2}p_{2,l} p_{2,m}  \int_{[0,\infty)^2} \lambda e^{-\lambda v} g_1(x) e^{-(\mu_{2,l}+\mu_{2,m})x} dv dx\\
&=&\sum_{l=1\atop m=1}^{K_2}p_{2,l} p_{2,m}  G^*_1(\mu_{2,l}+\mu_{2,m})\leq 1.
\end{eqnarray*}
We conclude from the above that $\E[ \rho(X,V)^2]<\infty$ if and only if 
\begin{equation}
\label{def:beta-l-m}
\beta_{l,m}:= \int_0^\infty  \frac{(g_1*g_{2,l})(x)\times (g_1*g_{2,m})(x)}{g_1(x)} dx<\infty
\end{equation}
for all $l,m=1,\ldots, K_2$.\\

{\bf Case 1:} $g_1(x)=\sum_{i=1}^{K_1} p_{1,i} \mu_{1,i} e^{-\mu_{1,i} x}$, $p_{1,i}\geq 0$ for all $i$ and $\sum_{i=1}^{K_1} p_{1,i}=1$, namely, Willie job service times follow an
 hyper-exponential distribution with mean $1/\mu_1=\sum_{i=1}^{K_1} 1/\mu_{1,i}$.

We have
\[
(g_1*g_{2,l})(x)= \sum_{i=1}^{K_1} p_{1,i} \mu_{1,i}\mu_{2,l} \Biggl[x e^{-\mu_{2,l}x} \,{\bf 1}(\mu_{1,i}=\mu_{2,l})+\frac{e^{-\mu_{2,l}x}-e^{-\mu_{1,i}x}}{\mu_{1,i}-\mu_{2,l}}\,{\bf 1}(\mu_{1,i}\not=\mu_{2,l})\Biggr]
\]
for $l=1,\ldots, K_2$, so that
\begin{eqnarray*}
\lefteqn{(g_1*g_{2,l})(x)\times (g_1*g_{2,m})(x)=\sum_{i=1\atop j=1}^{K_1} p_{1,i}p_{1,j}\mu_{1,i}\mu_{1,j} \Biggl[ P_{i,l}(x) P_{j,m}(x)e^{-(\mu_{2,l}+\mu_{2,m})x}} \\
&& - a_{i,l}b_{j,m}x e^{-(\mu_{1,i}+\mu_{2,m})x}-a_{j,m}b_{i,l} x e^{-(\mu_{1,j}+\mu_{2,l})x}+ b_{i,l}b_{j,m} e^{-(\mu_{1,i}+\mu_{1,j})x}\Biggr],
\end{eqnarray*}
with
\begin{eqnarray*}
a_{i,l}&:=&\mu_{2,l} {\bf 1}(\mu_{1,i}=\mu_{2,l})\\
b_{i,l}&:=&\frac{\mu_{2,l}}{\mu_{1,i}-\mu_{2,l}}{\bf 1}(\mu_{1,i}\not=\mu_{2,l}) \\
P_{i,l}(x)&:=&a_{i,l}x+b_{i,l},
\end{eqnarray*}
for $i=1,\ldots,K_1$, $l=1,\ldots,K_2$. 
Define $\mu^*_1=\max_{1\leq i\leq K_1} \mu_{1,i}$. Then,
\begin{eqnarray}
\beta_{l,m}&=&\sum_{i=1\atop j=1}^{K_1} p_{1,i}p_{1,j}\mu_{1,i}\mu_{1,j}
\int_0^\infty \frac{P_{i,l}(x)P_{j,m}(x) e^{-(\mu_{2,l}+\mu_{2,m}-\mu^*_1)x}}{\sum_{i=1}^{K_1} p_{1,i} \mu_{1,i} e^{(\mu^*_1-\mu_{1,i})x}} dx\nonumber\\
&-&\sum_{j=1}^{K_1}p_{1,j}\mu_{1,j}b_{j,m} \int_0^\infty \frac{\sum_{i=1}^{K_1}p_{1,i}\mu_{1,i}a_{i,l}e^{-\mu_{1,i}x}}{\sum_{i=1}^{K_1} p_{1,i} \mu_{1,i}  e^{-\mu_{1,i}x}}
x e^{-\mu_{2,m}x} dx\nonumber\\
&-&\sum_{j=1}^{K_1}p_{1,j}\mu_{1,j} b_{j,l} \int_0^\infty \frac{\sum_{i=1}^{K_1}p_{1,i}\mu_{1,i}a_{i,m}e^{-\mu_{1,i}x}}{\sum_{i=1}^{K_1} p_{1,i}\mu_{1,i}  e^{-\mu_{1,i} x}} 
x e^{-\mu_{2,l}x} dx\nonumber\\
&+& \sum_{j=1}^{K_1}p_{1,j}\mu_{1,j} b_{j,m}\int_0^\infty \frac{\sum_{i=1}^{K_1} p_{1,i}\mu_{1,i}b_{i,l}e^{-\mu_{1,i}x}}{\sum_{i=1}^{K_1} p_{1,i} \mu_{1,i} e^{-\mu_{1,i} x}} e^{-\mu_{1,j}x}dx.
\label{C0-cond-hyp}
\end{eqnarray}
The second, third, and fourth integrals in  the r.h.s. of (\ref{C0-cond-hyp}) are finite since $\lim_{x\to\infty} \frac{\sum_{i=1}^{K_1}p_{1,i}\mu_{1,i}a_{i,l}e^{-\mu_{1,i}x}}{\sum_{i=1}^{K_1} p_{1,i}\mu_{1,i} e^{-\mu_{1,i} x}}$ and $\lim_{x\to\infty} \frac{\sum_{i=1}^{K_1}p_{1,i}\mu_{1,i}b_{i,l}e^{-\mu_{1,i}x}}{\sum_{i=1}^{K_1} p_{1,i}\mu_{1,i} e^{-\mu_{1,i} x}}$
are finite for any $l=1,\ldots, K_2$. The first integral is finite if and only if 
\begin{equation}
\label{finiteness-C0-hyp}
\mu^*_1=\max_{1\leq i\leq K_1}\mu_{1,i}\leq 2 \min_{1\leq l\leq K_2} \mu_{2,l}.
\end{equation}
This shows that $C_0<\infty$ when (\ref{finiteness-C0-hyp}) holds.

In particular, when $K_1=K_2=1$ (exponential service times for both Alice and Willie jobs) then $C_0<\infty$ if and only if $\mu_1<2\mu_2$.  For further reference, note that
\begin{equation}
\label{rho-expo}
\rho(x,v)=\left\{\begin{array}{ll}
e^{-\mu_1 v} (\mu_1 x-1 ) & \mbox{if $\mu_1=\mu_2$}\\
e^{-\mu_2 v}\left(\frac{\mu_2 e^{-(\mu_2-\mu_1)x}-\mu_1}{\mu_1-\mu_2}\right) & \mbox{if $\mu_1\not=\mu_2$}
                        \end{array}
                \right.
\end{equation}
when $g_i(x)=\mu_i e^{-\mu_i x}$, $i=1,2$.

{\bf Case 2:} $g_1(x)=\nu_1^{K_1} x^{K_1-1} e^{-\nu_1 x}/(K_1-1)!$ with $1/\mu_1=K_1/\nu_1$ (Willie job service times follow a $K_1$-stage Erlang pdf with mean $1/\mu_1$).

We have, with $\eta_l:=\frac{\nu_1^{K_1} \mu_{2,l} }{(K_1-1)!}$,
\[
(g_1*g_{2,l})(x)=\left\{\begin{array}{ll}
\eta_l \, \frac{x^{K_1}}{K_1} e^{-\mu_{2,l}x}&\mbox{if $\nu_1=\mu_{2,l}$}\\
\eta_l   e^{-\mu_{2,l}x}\int_0^x u^{K_1-1} e^{-(\nu_1-\mu_{2,l})u} du &\\
&\mbox{if $\nu_1\not=\mu_{2,l}$,}
                                  \end{array}
                            \right.
\]
for $l=1,\ldots,K_2$.  

Define $\xi_l(k)=\int_0^x u^{k-1}e^{-(\nu_1-\mu_{2,l}) u}du$ for $k\geq 1$. Integrating by part gives
\[
\xi_l(k)= \frac{-x^{k-1}e^{-(\nu_1-\mu_{2,l}) x}}{\nu_1-\mu_{2,l}}+ \frac{k-1}{\nu_1-\mu_{2,l}} \xi_l(k-1),\,\,k\geq 2,
\]
which yields (use that $\xi_l(1)=(1-e^{-(\nu_1-\mu_{2,l}) x})/(\nu_1-\mu_{2,l})$)
\[
\xi_l(k)=-e^{-(\nu_1-\mu_{2,l}) x} \sum_{i=1}^{k} \frac{(k-1)!}{(k-i)!} \frac{x^{k-i}}{(\nu_1-\mu_{2,l})^i}+\frac{(k-1)!}{(\nu_1-\mu_{2,l})^k}.
\]
Therefore,
\begin{equation}
\label{g1*g2-erl}
(g_1*g_{2,l})(x)=Q_{1,l}(x)e^{-\mu_{2,l}x} -Q_{2,l}(x)e^{-\nu_1 x},
\end{equation}
with
\begin{eqnarray}
Q_{1,l}(x)&:=&\eta_l\frac{x^{K_1}}{K_1}\,{\bf 1}(\nu_1=\mu_{2,l}) +\eta_l\frac{(K_1-1)!}{(\nu_1-\mu_{2,l})^{K_1}} \,{\bf 1}(\nu_1\not=\mu_{2,l})\nonumber\\
&&\label{def:Q1l}\\
Q_{2,l}(x)&:=& \eta_l \sum_{i=1}^{K_1} \frac{(K_1-1)!}{(K_1-i)!} \frac{x^{K_1-i}}{(\nu_1-\mu_{2,l})^i}\,{\bf 1}(\nu_1\not=\mu_{2,l}).\nonumber\\
&&\label{def:Q2l}
\end{eqnarray}
When $\nu_1=\mu_{2,l}$ or $\nu_1=\mu_{2,m}$ it is easily seen from (\ref{g1*g2-erl})-(\ref{def:Q2l}) that $\beta(l,m)<\infty$ if and only if $\nu_1<\mu_{2,l}+\mu_{2,m}$.

Let us investigate the (less trivial) remaining case when $\nu_1\not= \mu_{2,l}$ and $\nu_1\not= \mu_{2,m}$. In this case we have, from (\ref{g1*g2-erl})-(\ref{def:Q2l}),
\begin{eqnarray*}
\lefteqn{\frac{(g_1*g_{2,l})(x) (g_1*g_{2,m})(x)}{g_1(x)} }\\
&= &\frac{\eta_l\eta_m}{\nu_1^{K_1} x^{K_1} e^{-\nu_1 x}}
\Biggl[ \frac{(K_1-1)!}{(\nu_1-\mu_{2,l})^{K_1}} e^{-\mu_{2,l}x} - \sum_{i=1}^{K_1} \frac{(K_1-1)!}{(K_1-i)!} \frac{x^{K_1-i}}{(\nu_1-\mu_{2,l})^i} e^{-\nu_1 x}\Biggr]\\
&&\times\Biggl[ \frac{(K_1-1)!}{(\nu_1-\mu_{2,lm})^{K_1}} e^{-\mu_{2,m}x} - \sum_{i=1}^{K_1} \frac{(K_1-1)!}{(K_1-i)!} \frac{x^{K_1-i}}{(\nu_1-\mu_{2,m})^i} e^{-\nu_1 x}\Biggr]\\
&=&\frac{\eta_l\eta_m ((K-1)!)^2}{\nu_1^{K_1}   (\nu_1-\mu_{2,l})^{K_1} (\nu_1-\mu_{2,m})^{K_1}     }\\
&&\times\Biggl[ e^{(\nu_1-\mu_{2,l})x} - \sum_{j=0}^{K_1-1} \frac{(x(\nu_1-\mu_{2,l}))^j }{j!}\Biggr] \times \frac{1}{x^{K_1}}
\Biggl[e^{-\mu_{2,m}x} - e^{-\nu_1 x} \sum_{j=0}^{K_1-1} \frac{ (x(\nu_1-\mu_{2,m})^j}{j!}\Biggr],
\end{eqnarray*}
which shows that $(g_1*g_{2,l})(x) (g_1*g_{2,m})(x)/g_1(x)$ is well-defined when $x\to 0$ and is $[0,\infty)$-integrable if and only if $\nu_1<\mu_{2,l}+\mu_{2,m}$. 

In summary, $C_0<\infty$ if and only if $\nu_1<2\min_{1\leq l\leq K_2} \mu_{2,l}$ or, equivalently, if and only if 
$\mu_1<\frac{2}{K_1}\min_{1\leq l\leq K_2} \mu_{2,l}$.


\section{Appendix}
\label{app:A0}
\begin{lemma}
\label{lem:technic}
Let $f,g :\N \to [0,\infty)$. 
If $\limsup_n\frac{g(n)}{f(n)}<\infty$ with $\lim_n g(n)=\infty$ then $\lim_n f(n)=\infty$.
\end{lemma}
\begin{proof}
If $\lim_n \frac{g(n)}{f(n)}=0$ then clearly $\lim_n f(n)=\infty$. Assume now that  there exist $0< L<\infty$ and $n_0$ such that for all $n>n_0$
\[
 \sup_{k\geq n} \frac{g(k)}{f(k)}<L.
 \]
Since $\sup_{k\geq n} \frac{g(k)}{f(k)} \geq \frac{g(n)}{f(n)}$, $f(n)>L^{-1}g(n)$ for $n>n_0$, 
which proves the lemma since $\lim_n g(n)=\infty$. 
\end{proof}
%


\section{Appendix: Proof of Lemma \ref{lem:expansion-Y}} 
\label{app:proof-key-lemma}

\begin{proof}
Assume that $0<r<1$ (i.e. $\beta<0$).  Recalling that  $X_r$ is an exponential rv with rate $\mu r$,  definition (\ref{def:xir}) yields  
\begin{eqnarray*}
\lefteqn{\P\left(\sqrt{ \Xi(\theta,X_r)}>z\right)}\\
&=&
\left\{\begin{array}{ll}
0 &\mbox{if $z>\sqrt{1-\theta\beta}$}\\
\left(\frac{1-\theta\beta -z^2}{\theta(1-\beta)} \right)^{-\beta} &\mbox{if $\sqrt{1-\theta}\leq z\leq \sqrt{1-\theta\beta}$}\\
1 &\mbox{if $z< \sqrt{1-\theta}$,}
        \end{array}
         \right.
 \end{eqnarray*}
which gives
\begin{eqnarray}
\lefteqn{\E\left[\sqrt{\Xi(\theta,X_r)}\right] =\int_0^\infty \P\left(\sqrt{\Xi(\theta,X_r)}>z\right) dz}\nonumber \\
 &=& \sqrt{1-\theta} +(\theta(1-\beta))^\beta 
  \int_{\sqrt{1-\theta}}^{\sqrt{1-\theta\beta}}(-y^2+1-\theta\beta)^{-\beta} dy\nonumber\\
&=&  \sqrt{1-\theta} +\left(\frac{\theta(1-\beta)}{1-\theta\beta}\right)^\beta (1-\theta\beta)^{1/2}
 \int_{\sqrt{\frac{1-\theta}{1-\theta\beta}}}^1 (1-y^2)^{-\beta} dy.\nonumber\\
&& \label{sqrtY-r-less-1}
 \end{eqnarray}
Recall that $\xi_r(\theta)=(1-\beta)\theta/(1-\theta\beta)$, so that $\theta=\xi_r(\theta)/(1-\beta+\beta \xi_r(\theta))$. Substitution into (\ref{sqrtY-r-less-1}) yields
(with $\xi_r\equiv \xi_r(\theta)$ with a slight abuse of notation)
\begin{eqnarray*}
\E\left[\sqrt{ \Xi(\theta,X_r)}\right]&=&\sqrt{1-\frac{\xi_r}{1-\beta + \beta\xi_r}} +\xi_r^\beta \left(1+ \frac{\beta\xi_r}{1-\beta}\right)^{-1/2} \int_{\sqrt{1-\xi_r}}^1 (1-y^2)^{-\beta} dy\\
&=&\sqrt{1-\frac{\xi_r}{1-\beta + \beta\xi_r}} +\frac{1}{2}\xi_r^\beta \left(1+ \frac{\beta\xi_r}{1-\beta}\right)^{-1/2}
\int_0^{\xi_r} \frac{x^{-\beta}}{\sqrt{1-x}}dx.
 \end{eqnarray*}
When $x$ is small, $x^{-\beta}/\sqrt{1-x}\sim x^{-\beta}+x^{1-\beta}/2$, so that  
$\int_0^{\xi_r} \frac{x^{-\beta}}{\sqrt{1-x}}dx \sim \xi_r^{1-\beta}/(1-\beta)+\xi_r^{2-\beta}/(2(2-\beta))$ as $\xi_r$ is small. With this, we obtain
\begin{eqnarray}
\E\left[\sqrt{ \Xi(\theta,X_r)}\right] &=&1 +\frac{1-4\beta+2\beta^2}{4(2-\beta)(1-\beta)^2} \,\xi_r^2+o(\xi_r^2)\nonumber\\
&=&1+  \frac{1-4\beta+2\beta^2}{4(2-\beta)}  \theta^2+o(\theta^2).
\label{series-EY}
 \end{eqnarray}
 Since $ \frac{1-4\beta+2\beta^2}{4(2-\beta)} =\frac{1-r}{4(r-2)}$,  this proves the lemma when $r<1$.\\
 
 Consider now the case where $r\geq 2$.  Notice that $1<\beta\leq 2$ when $r\geq 2$. It is easily seen from  (\ref{def-Xi}) that, for $r>1$, 
\[
\frac{d}{dz}\P( \Xi(\theta,X_r)<z)=
\left\{\begin{array}{ll}
\frac{\beta (\theta (\beta-1))^\beta}{ (z-1+\theta\beta)^{\beta+1}}&\mbox{ if $z\geq 1-\theta$}\\
0& \mbox{ if $z<1-\theta$,}
        \end{array}
\right.
 \]
which yields
\begin{eqnarray}
\E\left[\sqrt{ \Xi(\theta,X_r)}\right]&=&\beta \theta^\beta  (\beta-1))^\beta \int_{1-\theta}^\infty \frac{\sqrt{z}}{(z-1+\theta\beta)^{\beta+1}}dz\nonumber\\
&=&\frac{2\beta \theta^\beta (\beta-1)^\beta}{(1-\theta\beta)^{\beta-1/2}}\int_{\sqrt{\frac{1-\theta}{1- \theta\beta}}}^\infty \frac{t^2}{(t^2-1)^{\beta+1}} dt.
\label{sqrtY-r-more-1}
\end{eqnarray}
We are now ready to address the case when $r\geq 2$.\\

Assume first that $r=2$, so that $\beta=2$, $\xi_2\equiv \xi_2(\theta) = \frac{ \theta}{1-2\theta}$, and $\theta = \frac{\xi}{2 \xi + 1}$.
 By (\ref{sqrtY-r-more-1}) we have 
\begin{align*}
\E\left[\sqrt{ \Xi(\theta,X_2)}\right]&=4\left(1+2\xi_2\right)^{-1/2} \xi_2^2
\int_{\sqrt{\xi_2+1}}^\infty \frac{t^2}{(t^2-1)^{3}} dt\\
&=2\left(1+2\xi_2\right)^{-1/2} \xi_2^2
\int_{\xi_2}^\infty \frac{\sqrt{y+1}}{y^{3}} dy.
\end{align*}
Let 
$h(y) = \frac{1+\frac{1}{2} y - \sqrt{y+1}}{y^3}.$
We have (Hint: use L'H\^opital's rule to get the 2nd equality)
\[
\lim_{\xi_2\to 0}\frac{2\xi_2^2  \int_{\xi_2}^\infty \frac{\sqrt{y+1}}{y^{3}} dy - 1 -  \xi_2}{2 \xi_2^2 \log \xi_2}
= \lim_{\xi_2\to 0}  \frac{-\int_{\xi_2}^\infty h(y) dy}{\log\xi_2}= \lim_{\xi_2\to 0} \frac{h(\xi_2)}{\xi_2^{-1}}
= \lim_{\xi_2\to 0} \frac{1+\frac{1}{2}\xi_2 - \sqrt{\xi_2+1}}{\xi_2^{2} }= \frac{1}{8},
\]
and
\[
2\xi_2^2  \int_{\xi_2}^\infty \frac{\sqrt{y+1}}{y^{3}} dy = 1+ \xi_2 + \frac{1}{4}  \xi_2^2\log \xi_2 + o(\xi_2^2 \log \xi_2).
\]
It follows that
\begin{eqnarray}
\E\left[\sqrt{ \Xi(\theta,X_2)}\right] &=& \left(1+2\xi_2\right)^{-1/2} 2\xi_2^2
\int_{\xi_2}^\infty \frac{\sqrt{y+1}}{y^{3}} dy\nonumber\\
& =& \left(1- \xi_2  + O(\xi_2^2)\right) \left(1+ \xi_2 +\frac{1}{4}  \xi_2^2\log \xi_2 + o(\xi_2^2 \log \xi_2)\right)\nonumber\\
&=& 1 + \frac{1}{4}  \xi_2^2\log \xi_2 + o(\xi_2^2 \log \xi_2).
\label{expansion-Y-proof-b}
\end{eqnarray}
This proves the lemma when $r=2$.\\

Finally, assume that $r>2$.  Recall that $\xi_r \equiv \xi_r(\theta)= \frac{(\beta-1)\theta}{1-\beta\theta}$,  so that $\theta = \frac{\xi_r}{\beta \xi_r + \beta - 1}$ and, by
 (\ref{sqrtY-r-more-1}),
\begin{align*}
\E\left[\sqrt{ \Xi(\theta,X_r)}\right]&=2\beta\left(1+\frac{\beta}{\beta - 1}\xi_r\right)^{-1/2} \xi_r^\beta
\int_{\sqrt{\xi_r+1}}^\infty \frac{x^2}{(x^2-1)^{\beta+1}} dx\\
&=\beta\left(1+\frac{\beta}{\beta - 1}\xi_r\right)^{-1/2} \xi_r^\beta
\int_{\xi_r}^\infty \frac{\sqrt{y+1}}{y^{\beta+1}} dy.
\end{align*}
Let 
\[
h(y) = \frac{1+\frac{1}{2} y - \sqrt{y+1}}{y^{\beta+1}}.
\]
Note that $h(y)>0$ for $y>0$. As $y\to\infty$, $h(y) \sim \frac{1}{2} y^{-\beta}$. As $y\to 0$, $h(y) \sim \frac{1}{8} y^{1-\beta}$. Since $\beta\in (1,2)$ for $r>2$, the generalized integral $I_\beta:=\beta \int_0^\infty h(y) dy$ is finite and positive. 

Therefore,
\begin{eqnarray*}
\lim_{\xi_r\to 0}\frac{\beta\xi_r^\beta  \int_{\xi_r}^\infty \frac{\sqrt{y+1}}{y^{\beta+1}} dy - 1 - \frac{\beta}{2(\beta-1)} \xi_r}{\beta \xi_r^\beta}
&=& \lim_{\xi_r\to 0}  \int_{\xi_r}^\infty \frac{\sqrt{y+1}}{y^{\beta+1}} dy -\frac{1}{\beta} \xi_r^{-\beta} - \frac{1}{2(\beta-1)} \xi_r^{1-\beta}\\
&=&\lim_{\xi_r\to 0} -\int_{\xi_r}^\infty h(y)dy = -\frac{I_\beta}{\beta},
\end{eqnarray*}
and
\[
\beta\xi_r^\beta  \int_{\xi_r}^\infty \frac{\sqrt{y+1}}{y^{\beta+1}} dy = 1+  \frac{\beta}{2(\beta-1)} \xi_r -  I_\beta  \xi_r^\beta + o(\xi_r^\beta).
\]
It follows that
\begin{eqnarray}
\E\left[\sqrt{ \Xi(\theta, X_r}\right] &=& \left(1+\frac{\beta}{\beta - 1}\xi_r\right)^{-1/2} \beta\xi_r^\beta
\int_{\xi_r}^\infty \frac{\sqrt{y+1}}{y^{\beta+1}} dy\nonumber\\
& =& \left(1-\frac{\beta}{2(\beta-1)} \xi_r  + O(\xi_r^2)\right) \left(1+  \frac{\beta}{2(\beta-1)} \xi_r -  I_\beta \xi_r^\beta + o(\xi_r^\beta)\right)\nonumber\\
&=& 1-  I_\beta \xi_r^\beta + o(\xi_r^\beta).
\label{limit-c}
\end{eqnarray}
This  proves the lemma  when $r>2$.
\end{proof}


\section{Appendix}
\label{app:Dn}

\begin{lemma}
\label{lem:Dn}
\begin{equation}
\lim_n D_n=0, 
\label{limit-Dn}
\end{equation}
where $D_n$ is defined in (\ref{def-Dn}).
\end{lemma}
\begin{proof}
 Fix $\epsilon>0$. Since $\Delta_2(z)=o(z^2\log z)$ there exists $z_\epsilon>0$ such that for all $0<z<z_\epsilon$, 
$\left|\frac{\Delta_2(z)}{z^2\log z}\right|<\epsilon$. Since for all $n$ such that $\frac{\delta}{\phi(n)-2\delta}<z_\epsilon$ we have 
$\xi_2=\frac{\delta e^{-\mu v}} {\phi(n)-2\delta e^{-\mu v}}<z_\epsilon$ for all $v\geq 0$ (Hint: the mapping $v\to \xi_2$ is nonincreasing
in $[0,\infty)$ and $\xi_2=\frac{\delta}{\phi(n)-2\delta}$ when $v=0$), we conclude that for $n$ large enough,
\begin{equation}
\label{Delta_2-bound}
\sup_{v\geq 0}\left|\frac{\Delta_2(\xi_2)}{\xi_2^2\log \xi_2}\right|<\epsilon.
\end{equation}
Hence, for $n$ large enough,
\begin{eqnarray}
|D_n| &\leq&  \left(\frac{1}{4} +\epsilon\right) \int_0^\infty \lambda e^{-(\lambda+2\mu) v} \left| \frac{\log \xi_2}{(\phi(n)-2\delta e^{-\mu v})^2}\right| dv\nonumber\\
&=&\left(\frac{1}{4} +\epsilon\right) \int_0^\infty \lambda e^{-(\lambda+2\mu) v}\left|\frac{\log(\phi(n)-2\delta e^{-\mu v})+\mu v-\log \delta}{(\phi(n)-2\delta e^{-\mu v})^2}\right|dv.
\label{inq-Y-2}
\end{eqnarray}
For $n$ large enough 
\begin{eqnarray}
a_n(v)&:=&\left|\frac{\log(\phi(n)-2\delta e^{-\mu v})+\mu v-\log \delta}{(\phi(n)-2\delta e^{-\mu v})^2}\right|
 \label{def-an}\\
&\leq&
\mu v -\log \delta+\frac{|\log(\phi(n)-2\delta e^{-\mu v})|}{(\phi(n)-2\delta e^{-\mu v})^2}\nonumber
\end{eqnarray}
for all $v\geq 0$.
It is easy to check that  for all $n$ such that $\phi(n)>2\delta +\sqrt{e}$,  the mapping $v\to\frac{\log(\phi(n)-2-\delta e^{-\lambda v})}{(\phi(n)-2\delta e^{-\lambda v})^2}$ is non-decreasing in $[0,\infty)$. Therefore, for all $n$ such that $\phi(n)>2\delta  +\sqrt{e}$,   
\[
\frac{\log(\phi(n)-2\delta e^{-\mu v})}{(\phi(n)-2\delta e^{-\mu v})^2} \leq    \frac{\log(\phi(n)-2\delta)}{(\phi(n)-2\delta)^2}\quad \hbox{for all } v\geq 0.
\]
This shows that for all $n$ such that $\phi(n)>2\delta  +\sqrt{e}$ [Hint: $\phi(n)-2\delta e^{-\lambda v}>1$ for all $v\geq 0$ when $\phi(n)>2\delta  +\sqrt{e}$]
\begin{eqnarray}
\frac{|\log(\phi(n)-2\delta e^{-\mu v})|}{(\phi(n)-2\delta e^{-\mu v})^2} &=& \frac{\log(\phi(n)-2\delta e^{-\mu v})}{(\phi(n)-2\delta e^{-\mu v})^2}\nonumber\\
&\leq& \frac{\log(\phi(n)-2\delta)}{(\phi(n)-2\delta)^2}  \quad \forall v\geq 0.
\label{inq200} 
\end{eqnarray}
We conclude from (\ref{def-an}) and (\ref{inq200}) that for $n$ large enough
[Hint: for $n$ large enough,  $\log(\phi(n)-2\delta)/(\phi(n)-2\delta)^2)<1$ since  $\log t/t^2\to 0$ as $t\to \infty$ and $\phi(n)\to\infty$ as $n\to \infty$]
\[
0\leq a_n(v)\leq \mu v +1-\log \delta\quad \hbox{for all } v\geq 0.
\]
Since for every $v\geq 0$, $a_n(v)\to  0$ as $n\to \infty$ (cf. (\ref{def-an})), and 
\[
\int_0^\infty \lambda e^{-(\lambda+2\mu) v} ( \mu v +1-\log \delta) dv =\frac{\lambda \mu}{(\lambda+2\mu)^2} +\frac{\lambda (1-\log \delta)}{\lambda+2\mu}<\infty,
\]
we may apply the Bounded Convergence Theorem to the sequence $\{a_n(v)\}_n$, to get from (\ref{inq-Y-2}) that
\[
\lim_n |D_n|\leq \left(\frac{1}{4}+\epsilon\right)\lim_n\int_0^\infty \lambda e^{-(\lambda+2\mu) v} a_n(v)dv=  
 \left(\frac{1}{4}+\epsilon\right)\int_0^\infty \lambda e^{-(\lambda+2\mu) v} \lim_n a_n(v)dv=0.
\]
This concludes the proof of the lemma. 
\end{proof}
\begin{lemma}
\label{lem:nDn}

Assume that $T(n)=\mathcal{O}(\sqrt{n/\log n})$. Then $nD_n$ is bounded as $n\to\infty$.
\end{lemma}
\begin{proof}
Define
\[
f_n(v):=- \frac{n \log \xi_2 }{(\phi(n)-2\delta e^{-\mu v})^2} \left(\frac{1}{4}+\frac{\Delta_2(\xi_2)}{\xi_2^2\log \xi_2}\right),
\]
so that  (cf. (\ref{def-Dn}))
\begin{equation}
\label{nDn}
nD_n= -\int_0^\infty \lambda e^{-(\lambda +2\mu)v} f_n(v)dv.
\end{equation}
Since $\xi_2=\frac{\delta e^{-\mu v}}{\phi(n)-2\delta e^{-\mu v}}$, $f_n(v)$ rewrites
\begin{equation}
f_n(v)=\frac{n\left(\log\left(\phi(n)-2\delta e^{-\mu v}\right)-\log \delta+\mu v \right)}{(\phi(n)-2\delta e^{-\mu v})^2}  \left(\frac{1}{4}+\frac{\Delta_2(\xi_2)}{\xi_2^2\log \xi_2}\right).
\label{def:fnv}
\end{equation}
For any $v\geq 0$ and $\delta\in (0,1)$, we see from (\ref{def:fnv}) that
\begin{equation}
\label{inq-fn}
f_n(v)\geq \frac{n(\log \phi(n)-2)}{(\phi(n)-2\delta e^{-\mu v})^2} \left(\frac{1}{4}+\frac{\Delta_2(\xi_2)}{\xi_2^2\log \xi_2}\right).
\end{equation}
Thanks to assumption (\ref{hyp-phi}), $\log \phi(n)>2$ for $n$ large enough. Also,  by (\ref{Delta_2-bound}) 
$\sup_{v\geq 0}\frac{\Delta_2(\xi_2)}{\xi_2^2\log \xi_2}$ can be made arbitrarily small by letting $n\to \infty$. These two properties combined show from (\ref{inq-fn}) that
$f_n(v)\geq 0$ for $n$ large enough.

On the other hand, from (\ref{Delta_2-bound})  and (\ref{def:fnv}) we see that, for $n$ large enough,
\[
f_n(v)\leq \left(\frac{1}{4}+\epsilon\right) \frac{n\log\phi(n)-\log \delta +\mu v}{(\phi(n)-2)^2}, \quad \forall v\geq 0.
\]
Therefore, for $n$ large enough ,
\begin{eqnarray}
\lefteqn{0\leq \int_0^\infty \lambda e^{-(\lambda+2\mu) v} f_n(v)dv }\nonumber\\
&\leq &\left(\frac{1}{4}+\epsilon\right) \int_0^\infty \lambda e^{-(\lambda+2\mu) v}  \frac{n\log\phi(n)-\log \delta +\mu v}{(\phi(n)-2)^2}dv\nonumber\\
&=&\left(\frac{1}{4}+\epsilon\right)\left( \frac{\lambda}{\lambda +2\mu}\right)  \frac{n\log\phi(n)-\log \delta +\mu/(\lambda+2\mu)}{(\phi(n)-2)^2}\nonumber\\
&\sim_n&  \left(\frac{1}{4}+\epsilon\right)\left( \frac{\lambda}{\lambda +2\mu}\right)  \frac{n\log\phi(n)}{\phi(n)^2},
\label{inq-int-fn}
\end{eqnarray}
by using that $\lim_n \phi(n)=\infty$. We are left with finding $\phi$ such that $\frac{n\log\phi(n)}{\phi(n)^2}=\mathcal{O}(1)$.

To this end,  let $T(n)=\mathcal{O}(\sqrt{n/\log n})$ or, equivalently by (\ref{Tn-delta}), 
 $\liminf_n \frac{\phi(n)}{\sqrt{n\log n}}=a$ for some $a>0$. 
Let us write $n\log\phi(n)/\phi(n)^2$ as
\begin{equation}
\label{identity-temp}
\frac{n\log\phi(n)}{\phi(n)^2}=\frac{1}{\left(\phi(n)/\sqrt{n\log n}\right)^2}\left(\frac{\log(\phi(n)/\sqrt{n\log n})}{\log n}+\frac{\log(\log n)}{\log n}+\frac{1}{2}\right).
\end{equation}
Assume that $\limsup_n \frac{\phi(n)}{\sqrt{n\log n}}=b<\infty$. Then,
\[
\limsup_{n\to \infty} \frac{n\log\phi(n)}{\phi(n)^2}\leq \frac{1}{a^2} \limsup_{n\to \infty} \left(\frac{\log(\phi(n)/\sqrt{n\log n})}{\log n}+\frac{\log(\log n)}{\log n}+\frac{1}{2}\right)=\frac{1}{2a^2}<\infty
\]
by using $\lim_{x\to\infty}\frac{\log x}{x}=0$.  Assume now that $\limsup_n\frac{\phi(n)}{\sqrt{n\log n}}=\infty$. By rewriting (\ref{identity-temp}) as 
\[
\frac{n\log\phi(n)}{\phi(n)^2}=\frac{1}{\log n} \cdot \frac{\log(\phi(n)/\sqrt{n\log n}}{\left(\phi(n)/\sqrt{n\log n}\right)^2} + 
\frac{1}{\left(\phi(n)/\sqrt{n\log n}\right)^2}\left(\frac{\log(\log n)}{\log n}+\frac{1}{2}\right),
\]
we immediately conclude that $\lim_n \frac{n\log\phi(n)}{\phi(n)^2}=0$ thanks again to $\lim_{x\to\infty}\frac{\log x}{x}=0$.  
This shows that $\frac{n\log\phi(n)}{\phi(n)^2}=\mathcal{O}(1)$ when $T(n)=\mathcal{O}(\sqrt{n/\log n})$, and therefore by (\ref{inq-int-fn}) and (\ref{nDn}),  
that 
\begin{equation}
\label{nDn-bounded}
nD_n=\mathcal{O}(1),
\end{equation}
when $T(n)=\mathcal{O}(\sqrt{n/\log n})$, which completes the proof.
\end{proof}
\begin{lemma}
\label{lem:En}
\begin{equation}
\lim_n E_n=0, 
\label{limit-Dn}
\end{equation}
where $E_n$ is defined in (\ref{def-En}).
\end{lemma}
\begin{proof}
Fix $\epsilon>0$. Since $\Delta_r(z)=o(z^{\beta})$ when $r>2$, there exists $z_\epsilon>0$ such that for all $0<z<z_\epsilon$, 
$\left|\frac{\Delta_r(z)} {z^{\beta}  }\right|<\epsilon$. Since for all $n$ such that $\frac{\delta(\beta-1)}{\phi(n)-\delta \beta}<z_\epsilon$ we have 
$\xi_r=\frac{\delta (\beta-1)} {e^{\mu v}\phi(n)-\delta \beta}<z_\epsilon$ for all $v\geq 0$ (Hint: the mapping $v\to \xi_r$ is nonincreasing
in $[0,\infty)$ and $\xi_r=\frac{\delta(\beta-1)}{\phi(n)-\delta\beta}$ when $v=0$), we conclude that for $n$ large enough,
\begin{equation}
\label{Delta_r-bound}
\sup_{v\geq 0}\left|\frac{\Delta_r(\xi_r)}{\xi_r^\beta}\right|<\epsilon.
\end{equation}
Hence, for $n$ large enough,
\begin{eqnarray*}
|E_n| &\leq& (I_\beta +\epsilon) \int_0^\infty \frac{\lambda e^{-\lambda v}}{(\phi(n)e^{\mu v}-\delta\beta )^\beta} dv \nonumber\\
&\leq& \frac{I_\beta +\epsilon}{(\phi(n)-\delta\beta)^\beta} \int_0^\infty \lambda e^{-\lambda v} dv \nonumber\\
&= &\frac{I_\beta +\epsilon}{(\phi(n)-\delta\beta )^\beta} \to 0\quad \hbox{as } n\to\infty.
\label{inq-En}
\end{eqnarray*}
\end{proof}
\begin{lemma}
\label{lem:nEn}

Assume that $T(n)=\mathcal{O}(n^{\mu_2/\mu_1})$ with $\mu_1>2\mu_2$. Then $E_n$ defined in (\ref{def-En}) is bounded as $n\to\infty$.
\end{lemma}
\begin{proof}
Define
\begin{equation}
k_n(v)=n\frac{I_\beta- \Delta_r(\xi_r)/\xi_r^\beta}{(e^{\mu\beta v}\phi(n)-\delta \beta)^\beta},\label{def-knv}
\end{equation}
where $I_\beta$ is defined in (\ref{def:Ibeta}).
With this new function we can rewrite $nE_n$ (cf. (\ref{def-En})) as 
\begin{equation}
 nE_n=  \int_0^\infty\lambda e^{-\lambda  v}k_n(v)  dv.
\label{nEn}
\end{equation}
Notice that
\begin{eqnarray}
\label{def:gn}
k_n(v) &= & \frac{I_\beta- \Delta_r(\xi_r)/\xi_r^\beta}{ (e^{\mu\beta v} \phi(n)/n^{1/\beta}-\delta \beta/n^{1/\beta})^\beta} \nonumber \\
&\leq &  \frac{I_\beta- \Delta_r(\xi_r)/\xi_r^\beta}{(\phi(n)/n^{1/\beta}-\delta \beta/n^{1/\beta})^\beta},
\label{bound-knv}
\end{eqnarray}
for all $n\geq 1$ and $v\geq 0$. Recall that $I_\beta>0$.  Let $\epsilon<I_\beta$ in  (\ref{Delta_r-bound}).  From (\ref{bound-knv}) we see that  for  $n$ large enough
\begin{equation}
\label{bounds-knv}
0\leq k_n(v)\leq \frac{I_\beta +\epsilon}{(\phi(n)/n^{1/\beta}-\delta \beta/n^{1/\beta})^\beta}\quad \hbox{for all } v\geq 0.
\end{equation}
Recall that $\beta =\frac{r}{r-1}$ yielding $r=\frac{\beta}{\beta-1}$. Assume that  $T(n)=\frac{\delta}{n\phi(n)}=\mcO(n^{1/r})$ for $\delta\in (0,1)$, or equivalently
\[
\liminf_n \frac{\phi(n)}{n^{1/\beta}}=b
\]
for some $b>0$.
From (\ref{bounds-knv}) we obtain
\begin{eqnarray}
0&\leq& \limsup_n \int_0^\infty \lambda e^{-\lambda v} k_n(v)dv\nonumber\\
&\leq& \limsup_n \frac{I_\beta +\epsilon}{  (\phi(n)/n^{1/\beta}-\delta \beta/n^{1/\beta})^\beta} \nonumber\\
&= &\frac{I_\beta +\epsilon}{  (\liminf_n\phi(n)/n^{1/\beta})^\beta}=\frac{I_\beta +\epsilon}{b^\beta}.
\label{bounds-integral}
\end{eqnarray}
This shows  that $nE_n\in \mathcal{O}(1)$. 
\end{proof}


\section{Appendix}
\label{app:B}
\begin{lemma}
\label{lem:stochastic-ordering}
For any $\theta \in [0,1]$, $r'\geq r$,
\[
\E\left[\sqrt{ \Xi(\theta,X_{r'}) }\right]\leq \E\left[\sqrt{ \Xi(\theta,X_r) }\right].
\]
\end{lemma}
\begin{proof}
Fix $\theta\in[0,1]$.
When $r'\geq r$ then $X_{r'}\leq_{st} X_r$, which in turn implies that  $\Xi(\theta,X_{r'})\leq_{st} \Xi(\theta,X_r)$ as  the mapping
 $x\to \Xi(\theta,x)$ in (\ref{def-Xi}) is nondecreasing in $[0,\infty)$,

Therefore,
\[
\E\left[\sqrt{\Xi(\theta,X_{r'})}\right]\leq \E\left[\sqrt{\Xi(\theta,X_r)}\right],
\]
as the mapping $x\to \sqrt{x}$ is nondecreasing in $[0,\infty)$.
\end{proof}


\section{Appendix: Proof of  Proposition \ref{prop:Tw}}
\label{app:Tw-proof}
Recall that under the IEBP policy the system behaves as an M/G/1 queue with an exceptional first job in each busy period.
The service times of first jobs in busy periods have pdf  $\widetilde f_1$ and the  service times of the other jobs have pdf $g_1$. 
The numbers of jobs served in different busy periods are iid rvs, characterized by the random variable $M$,  so that  the expected number of 
Willie  jobs served during  $n$ W-BPs  is $T_W(n)=n\E[M]$. 

Let us calculate $\E[M]$.  To this end, introduce $\mathcal{G}_M(z)=\E[z^M]$, $|z|\leq 1$, the generating function of the number of jobs served in a busy period.  
Recall (Section \ref{sec:IEBP}) that the reconstructed service times of the first job served in different busy periods are iid rvs, and let $Y$ be a generic reconstructed service time. Let $\tau^*(s)=\E[e^{-Ys}]=\int_0^\infty e^{-s x} \widetilde f_1(x)dx $ be the LST of the reconstructed service time.
Since  the LTS of the service times all the other Willie jobs in a W-BP is $G^*_1(s)$, we get from  \cite{Bhat-64}
\begin{equation}
\label{Mz-II}
\mathcal{G}_M(z)=z\tau^*(\lambda(1-d(z)),\quad |z|\leq 1,
\end{equation}
where $d(z)$ is the root with the smallest modulus of the equation $t=zG_1^*(\lambda(1-t))$.

Noting that $d(1)=1$ and $\frac{d}{d(z)}|_{z=1} =\frac{1}{1-\lambda/\mu_1}$, we obtain from \eqref{Mz-II} 
\begin{equation}\label{eq:EM}
\E[M]=\frac{1-\lambda/\mu_1+\lambda \E[Y]}{1-\lambda/\mu_1},
\end{equation}
provided that the stability condition $\lambda/\mu_1 <1$ holds.  It remains to find $\E[Y]$. For that, we will use the identity $\E[Y]=-\frac{d\tau^*(s)}{ds}|_{s=0}$. 
But before that we need to calculate $\tau^*(s)$. 

When IEBP is enforced (or, equivalently, under $H_1$) we know that  $Y$ has pdf $\widetilde f_1$ (see Section \ref{ssec:detector}) . 
Multiplying both sides of (\ref{eq:Ztilde-rho}) by $g_1(x)$  and using the definition of $\widetilde Z(q,x)$ in (\ref{eq:Ztilde-def}) 
along with (\ref{eq:rho-tilde}) gives
\[
\widetilde f_1(x)=(1-qp) g_1(x)+q (g_1* \hat g_2)(x),
\]
where $\hat g_2(x)$ is defined in (\ref{eq:g2-hat}). Therefore,
 \begin{align} \label{eq:LT-Y}
\tau^*(s) &=\int_0^\infty e^{-s x} \widetilde f_1(x) dx =\int_0^\infty  e^{-s x} [(1-pq)g_1(x) + q(g_1*\hat g_2)(x)] dx \nonumber\\
 &= G^*_1(s)\left(1-pq +q \int_0^\infty e^{-st} \hat g_2(t)dt\right).
 \end{align}
Differentiating \eqref{eq:LT-Y} with respect to $s$ at $s=0$ and using the identity\footnote{$\int_0^\infty \hat g_2(t)dt =\int_0^\infty \lambda e^{-\lambda v} \int_v^\infty g_2(t)dt dv= \int_0^\infty \lambda e^{-\lambda v} (1-G_2(v))dv 
=1-G^*_2(\lambda)=p$.} $\int_0^\infty  \hg_2(t)dt=p$ yields
\begin{eqnarray*}
\E[Y]&=&\frac{1-pq}{\mu_1}+\frac{q}{\mu_1}\int_0^\infty \hg_2(t)dt+ q\int_0^\infty t \hat g_2(t)dt\\
&=& \frac{1}{\mu_1}+ q\int_0^\infty t \hat g_2(t)dt,
\end{eqnarray*}
By (\ref{eq:EM}), 
\begin{equation}\label{eq:EM-2}
\E[M]=\frac{1}{1-\lambda/\mu_1}+\frac{\lambda q}{1-\lambda/\mu_1} \int_0^\infty t \hat g_2(t)dt,
\end{equation}
and 
\begin{equation} \label{eq:value-TWn}
T_W(n)= n\E[M] = \frac{n}{1-\lambda/\mu_1}+\frac{\lambda q n}{1-\lambda/\mu_1} \int_0^\infty t \hat g_2(t)dt.
\end{equation}
Now we upper bound the integral in \eqref{eq:EM-2} and \eqref{eq:value-TWn}. We have
\begin{eqnarray*}
\int_0^\infty t \hat g_2(t)dt&=& \int_0^\infty \lambda e^{-\lambda v}  \int_0^\infty t g_2(v+t)dvdt\\
&\leq &  \int_0^\infty \lambda e^{-\lambda v}  \int_0^\infty (t+v)g_2(v+t)dvdt\\
&=&  \int_0^\infty \lambda e^{-\lambda v}  \int_{v}^\infty u g_2(u)dudv\\
&\leq&   \int_0^\infty \lambda e^{-\lambda v}  \int_{0}^\infty u g_2(u)dudv=\frac{1}{\mu_2}.
\end{eqnarray*}
Hence, $\E[M]\leq \frac{1+q\lambda/\mu_2}{1-\lambda/\mu_1}$ and $T_W(n)\leq n \left(\frac{1+q\lambda/\mu_2}{1-\lambda/\mu_1}\right)$. This shows the upper bound in \eqref{eq:Tw-bounds}. The lower bound is trivial.

If $g_2(x)=\mu_2 e^{-\mu_2 x}$ then   from \eqref{eq:g2-hat} and \eqref{eq:interference-prob} we find
$p=\frac{\lambda}{\mu_2+\lambda}$ and $\hg_2(x)=\frac{\lambda\mu_2}{(\lambda+\mu_2)}e^{-\mu_2 x}=p \mu_2 e^{-\mu_2 x}$, which yields  \eqref{eq:Tw-exponential}.


\section{Proof of Lemma \ref{lem:II}}
\label{app:D0}

\begin{proof}
We denote by $f^{(k)}$ the $k$th convolution of $f$ with itself. For the time being we do not make any assumption on $g_2(x)$.

Let $A$ be the event that Alice inserts a job at the end of a W-BP, with $\P(A)=q$.
Given that Alice inserts a job at the end of a W-BP, let $B_i$ ($i\geq 1$) be the event that Alice $i$th job inserted after the end of a  W-BP affects Willie first job. Notice that
\[
\P(B_i | A, V_*=t)= q^{i-1}\P\left(\sum_{l=1}^{i-1} \sigma_{2,l}<t <\sum_{l=1}^{i} \sigma_{2,l}\right),\quad i\geq 1.
\]
Let us calculate $\P_{H_1}(Y_*<x,V_*<v)$. For the sake of simplicity we will drop the subscript $H_1$.
We have
\begin{eqnarray}
\lefteqn{\P(Y_*<x,V_*<v)=}\nonumber\\
&& q\P(Y_*<x,V_*<v|A)+\bar q \P(Y_*<x,V_*<v| A^c)\nonumber\\
&=&q\int_0^v P(Y_*<x |A, V_*=t) \lambda e^{-\lambda t} dt\nonumber \\
&&+ \bar q G_1(x)(1-e^{-\lambda v}).
\label{int1000}
\end{eqnarray}
Let us focus on $P(Y_*<x|A, V_*=t)$. We have
\begin{eqnarray*}
\lefteqn{\P(Y_*<x|A, V_*=t) =\sum_{i\geq 1} P( \{Y_*<x\} \cap B_i |A, V_*=t)}\\
&&+  \P\left( \{Y_*<x \} \cap \left(\cup_{l\geq 1}B_l\right)^c \right)\\
&=&P(t <\sigma_{2,1}<x+t -\sigma_1)\\
&&+ \sum_{i\geq 2}  \int_{u=0}^t  q^{i-1}P( t-u<\sigma_{2,i}<x+t-u+\sigma_1)\\
&&\times g_2^{(i-1)}(u)du+  G_1(x)\P\left( \left(\cup_{l\geq 1}B_l\right)^c \right) \\
\end{eqnarray*}
Let us find $\P\left( \left(\cup_{l\geq 1}B_l\right)^c|A, V=t\right)$. This is the probability that no Alice job intersects with a Willie job given that Alice
inserts a job at the end of a W-BP and that $V_*=t$. Given $A$ and $V_*=t$, there is no interference if Alice inserts $i\geq 1$ jobs successfully and that she does not insert an $(i+1)$-st job.
The probability of this event is $\bar q q^{i-1}\P(\sigma_{2,1}+\cdots+\sigma_{2,i}<t)$. Therefore,
\begin{eqnarray*}
\lefteqn{\P\left ( \left(\cup_{l\geq 1}B_l\right)^c|A, V=t\right)=}\\
&&\bar q G_1(x)\sum_{l\geq 1} q^{l-1}\P(\sigma_{2,1}+\cdots+ \sigma_{2,l}<t)\\
&=&\bar q G_1(x)\sum_{i\geq 1} q^{l-1} G_2^{(l)}(t).
\end{eqnarray*}
Therefore,
\begin{eqnarray*}
\lefteqn{\P(Y<x|A, V=t)=P(t<\sigma_{2,1}<x+t-\sigma_1)}\\
&&+ \sum_{i\geq 2}  \int_{u=0}^t  q^{i-1}P(t-u< \sigma_{2,i}<x+t-u+\sigma_1) g_2^{*(i-1)}(u)du\\
&&+ \bar qG_1(x) \sum_{l\geq 1} q^{l-1} G_2^{*(l)}(t).
\end{eqnarray*}
Hence,
\begin{eqnarray*}
\lefteqn{\P(Y<x,V<v)=q \int_0^v \lambda e^{-\lambda t} P(t<\sigma_{2,1}<x+t-\sigma_1) dt}\\
&&+\int_{t=0}^v \lambda e^{-\lambda t} \Biggl[ \int_{u=0}^t  \P(t-u <\sigma_{2,i}< x+t-u-\sigma_1)\\
&&\times \sum_{i\geq 2}q^{i}  g_2^{(i-1)}(u)du+\bar q G_1(x)\sum_{l\geq 1} q^{l} G_2^{*(l)}(t)\Biggr]dt \\
&&+ \bar q G_1(x)(1-e^{-\lambda v}),
\end{eqnarray*}
which gives after conditioning on $\sigma_1$

\begin{eqnarray*}
\lefteqn{\P(Y<x,V<v)}\\
&&=q\int_{t=0}^v \lambda e^{-\lambda t} \int_{y=0}^x (G_2(x-y+t)-G_2(t))g_1(y) dydt\\
&&+q \int_{t=0}^v \lambda e^{-\lambda t}\Biggl[\int_{u=0}^t \sum_{i\geq 1}q^{i}  g_2^{*(i)}(u) \\
&&\times  \int_{y=0}^x (G_2(x-y+t-u)-G_2(t-u)) g_1(y)dy du\Biggr] dt \\
&&+\bar q G_1(x) \int_{0}^v \lambda e^{-\lambda t}  \sum_{l\geq 1} q^{l} G_2^{*(l)}(t)dt+ \bar q G_1(x)(1-e^{-\lambda v}).
\end{eqnarray*}

From now on we will assume that $G_2(t)=1-e^{-\mu_2 t}$ (Alice service times are exponential), which implies that
\[
G_2^{*(l)}(t)=1 -e^{-\mu_2 t}\sum_{m=0}^{l-1}\frac{(\mu_2 t)^m}{m!}, \quad l\geq 1.
\]
Easy algebra gives  (Hint:  $g^{*(i)}(t)=\frac{d}{dt} G_2^{*(i)}(t)$)
\begin{eqnarray}
\sum_{l\geq 1} q^{l} G_2^{*(l)}(t)
&=&\frac{q}{\bar q}\left(1- e^{-\mu_2\bar q t}\right)\label{G2-convol}\\
\sum_{i\geq 1}q^i  g_2^{(i)}(u)du&=&\mu_2 q e^{-\mu_2\bar q u}. \label{g2-convol}
\end{eqnarray}

Lengthy but easy algebra using (\ref{G2-convol})-(\ref{g2-convol}) gives

\begin{eqnarray*}
\lefteqn{\P(Y<x,V<v)=\left(1-e^{-\lambda v}\right)G_1(x)}\\
&&-\frac{pq}{1-\bar p q} \left(1-e^{-(\lambda +\mu_2 \bar q)v}\right)\frac{(g_1*g_2)(x)}{\mu_2}.
\end{eqnarray*}

Again after easy algebra, we  finally find 
\begin{eqnarray*}
\lefteqn{f_{+,1}(x,v)=\frac{\partial^2 }{\partial x\partial v}\P(Y<x,V<v)}\\
&=& \lambda e^{-\lambda v} g_1(x) \left[1+ qe^{-\mu_2\bar q v} \left(\frac{(g_1*g_2)(x)}{g_1(x)}-1\right)\right].
\end{eqnarray*}
and
\begin{eqnarray*}
\widetilde f_{+,1}(x)&=&\int_0^\infty f_{+,1}(x,v) dv\\
&=&  g_1(x) \left[1+ \frac{pq}{1-\bar p q}\left(\frac{(g_1*g_2)(x)}{g_1(x)}-1\right)\right].
\end{eqnarray*}

\end{proof}


\section{Proof of Lemma \ref{lem:NA-NW}}
\label{app:C}

Under the II policy the queue behaves as an M/M/1 queue with an exceptional first customer. Let $\hat \sigma$ be the expected service time of this first customer
and let $\hat \tau$ be the expected service time of the other customers.  Then (\cite{Bhat-64}  -  see also \prettyref{sec:IEBP}),
\begin{equation}
\label{appC:NW}
\E[N_W]=\frac{1-\lambda\hat \sigma+\lambda \hat \tau}{1-\lambda \hat \sigma}.
\end{equation}
$\hat \tau$ is the sum of the Willie's job expected service time (given by $1/\mu_1$) and of the expected time needed to serve all Alice's jobs inserted just after a Willie's job arrival.
The latter quantity is given by $qB/\mu_2$. Hence,  $\hat \tau=\frac{1}{\mu_1}+ \frac{qB}{\mu_2}$. $\hat \sigma$ is the sum of Willie's job expected service time (given by $1/\mu_1$)
and of the expected time needed to serve all Alice's jobs present in the queue at the beginning of a W-BP. The probability that there are $s$ such jobs 
in given by $\P(\mathcal{E}(s))$ in (\ref{eq:E1})-(\ref{eq:Es}) in Appendix \ref{app:D}. Therefore,
\[
\hat \sigma= \frac{1}{\mu_1}+\frac{1}{\mu_2}\sum_{s\geq 1} \P(\mathcal{E}(s)) s.
\]
Elementary algebra then gives
\[
\hat \sigma=\frac{1}{\mu_1} + \frac{q p}{\mu_2(1-q\bar p)}\left(\frac{1-  q(1- \mathcal{G}_Q(\bar p) )}{1-q\bar p} + \frac{B - (1+p)Q(1)}{p}\right).
\]
Introducing $\hat \tau$ and $\hat \sigma$ into (\ref{appC:NW}) gives (\ref{lem:NW}).  

During a W-BP, $qB\E[N_W]$ Alice's jobs are inserted on average. Therefore, $\E[N_A]$ is the sum of these jobs and of the expected number of jobs that Alice inserts during a 
W-IP. Let call $\E[N_{A,IP}]$ this number.  Let $\kappa$ be the number of
Alice's jobs in the system at the beginning of a W-IP. Note that these jobs were inserted just after the arrival of the last Willie's job served in the previous W-BP,
so that $\P(\kappa=k)=qQ(k)$ if $k\geq 1$ and $\P(\kappa=0)=\bar q+qQ(0)$. If $\kappa=0$ Alice's inserts $i\geq 1$ jobs in a W-IP  if either she inserts successfully $i$ jobs and stops there
(prob. $(q\bar p)^i \bar q$) or if she inserts $i$ jobs but the last one is not successful (prob. $(q\bar p)^{i-1} q p$) giving the overall prob. $q(q \bar p)^{i-1}( \bar p \bar q +p)$. Hence,  the expected number of Alice's jobs inserted in a W-IP given that $\kappa=0$ is $q( \bar p \bar q +p) \sum_{i\geq 1} (qp)^{i-1}\, i=q( \bar p \bar q +p)/(1-q\bar p)^2$.
If $\kappa>0$
Alice will not insert any job in a W-IP if a Willie's job arrives within the time to serve these $\kappa$ jobs, the probability of this event being $\bar p^\kappa$ and otherwise she will insert 
$i\geq 1$ jobs with the prob. $\bar p^\kappa q( \bar p \bar q +p)(qp)^{i-1}$. Hence, the expected number of Alice's jobs inserted in a W-IP given that $\kappa\geq 1$ is
$q( \bar p \bar q +p) \bar p^\kappa\sum_{i\geq 1}(qp)^{i-1}\,i=q( \bar p \bar q +p) \bar p^\kappa/(1-qp)^2$. Finally,
\begin{eqnarray*}
\E[N_{A,IP}]&=&\frac{q( \bar p \bar q +p)}{(1-q\bar p)^2} \left(\bar q+qQ(0)) +\sum_{k\geq 1}Q(k) \bar p^k\right)\\
&=&\frac{q( \bar p \bar q +p)}{(1-q\bar p)^2} \left(\bar q+qQ(0)) +\mathcal{G}_Q(\bar p) - Q(0)\right)
\end{eqnarray*}
and 
\[
\E[N_{A}]=qB \E[N_W] + q\,\frac{\bar p \bar q +p}{(1-q\bar p)^2} \left(\bar q \overline{Q}(0)+\mathcal{G}_Q(\bar p) \right),
\]
which concludes the proof.


\section{Proof of Lemma \ref{lem-Yqx}}
\label{app:D}
Throughout the proof we will skip the subscript $H_1$ in $\P_{H_1}(Y<x)$ for the sake of conciseness.
In this appendix $U_{\lambda}$ denotes an exponential rv with rate $\lambda$. 

Define the events
\begin{eqnarray*}
{\mathcal E}_s&=&\{s\ \hbox{ Alice's jobs interfere with } J \hbox{ given } J=1\}\\
{\mathcal F}_s&=&\{s\ \hbox{ Alice's jobs interfere with } J \hbox{ given } J\not=1\}\\
{\mathcal G}_l&=&\{\hbox{Alice inserts } l \hbox { jobs after the  arrival of a  Willie's  job}\}
\end{eqnarray*}
for $s\geq 0$, $l\geq 0$. We have 
\begin{eqnarray}
\P({\mathcal F}_0)&=& \bar q +qQ0) = 1 - q\overline{Q}(0)  \label{eq:F0}\\
\P({\mathcal F}_s)&=& q Q(s),\,\, s\geq 1   \label{eq:Fs}\\
\P({\mathcal G}_0)&=&\bar q+qQ(0)= 1 - q\overline{Q}(0)  \label{eq:G0}\\
\P({\mathcal G}_l)&=&qQ(l),\,\, l\geq 1\label{eq:GI}.
\end{eqnarray}
Let $T-$ be the time at which a W-BP ends. Time $T$ is the time at which Alice inserts one job with probability $q$  and
$0$ job with probability $\bar q$ is the system if empty at $T-$. Let us determine $\P({\mathcal E}_s)$ for $s\geq 0$. We have
\begin{equation}
\label{eq:E0-0}
\P({\mathcal E}_0)=\P({\mathcal E}_0\,|\, {\mathcal G}_0)(1 - q\overline{Q}(0))+   q \sum_{l\geq 1} \P({\mathcal E}_0\,|\, {\mathcal G}_l) Q(l).
\end{equation}
Recall that $\bar p=\frac{\mu_2}{\mu_2+\lambda}$ is the probability that no Willie's job arrives during the service time of a Alice's job.

Throughout, we will use that
\begin{equation}
\label{def:p}
 P\left(U_\lambda>\sum_{r=1}^k\tau_r\right)= \bar p^k
 \end{equation}
 when $\tau_1,\tau_2,\ldots$ are iid exponential rvs with rate $\mu_2$.
Given ${\mathcal G}_0$, there is no interference if Alice does not submit a job when an idle period starts (prob. $\bar q$)  or if Alice submits one job (prob. $q$)
and that during the service time of this Alice's job there is no arrival of a Willie job (prob.  $p$) and Alice 
does not submit another job when the system becomes idle again (prob. $\bar q$),  etc.
This gives (same argument/result as in (\ref{def-T+n}))
\begin{equation}
\label{eq:E0-1}
\P({\mathcal E}_0\,|\, {\mathcal G}_0) =  \bar q +  \bar q q \bar p +  \bar q  (q\bar p)^2 + \bar q(q\bar p)^3+\cdots
= \bar q\sum_{i=0}^\infty  (q\bar p)^i=\frac{\bar q}{1-q\bar p}.
\end{equation}
For $l\geq 1$, 
\begin{eqnarray*}
\P({\mathcal E}_0\,|\, {\mathcal G}_l)&=& \bar q \P\left(U_\lambda>\sum_{r=1}^l\tau_r\right) + \bar q q \P\left(U_\lambda>\sum_{r=1}^{l+1}\tau_r\right)\\
&&+
\bar q  q^2 \P\left(U_\lambda>\sum_{r=1}^{l+2}\tau_r\right)+\cdots\\
&=& \bar q \sum_{i=0}^\infty P\left(U_\lambda>\sum_{r=1}^{l+i}\tau_r\right) q^i= \bar q p^l  \sum_{i=0}^\infty (q\bar p)^i\\
&= &\frac{\bar q p^l}{1-q\bar p}.
\end{eqnarray*}
In summary, 
\begin{equation}
\label{eq:E0-cond}
\P({\mathcal E}_0\,|\, {\mathcal G}_l)=  \frac{\bar q \bar p^l}{1-q \bar p}, \quad \forall l\geq 0.
\end{equation}
Therefore, from (\ref{eq:E0-0})-(\ref{eq:E0-cond}),

\begin{eqnarray}
\P({\mathcal E}_0)&=&\frac{\bar q}{1-q\bar p}\left(1-q\overline{Q}(0) +q\sum_{l\geq 1}\bar p^l Q(l)\right)\nonumber\\
&=&\frac{\bar q}{1-q\bar p}(1-q\overline{Q}(0) + q\mathcal{G}_Q(\bar p) - qQ(0))\nonumber\\
&=&\frac{\bar q}{1-q\bar p} \left(\bar q+q\mathcal{G}_Q(\bar p)\right).
\label{eq:E0}
\end{eqnarray}

Consider now $\P({\mathcal E}_s\,|\, {\mathcal G}_l)$ for $s\geq 1$. We will investigate separately the case $s=1$ and $s\geq 2$. 
For $s=1$, $l\geq 1$, we have
\begin{eqnarray*}
\P({\mathcal E}_1\,|\, {\mathcal G}_l)&=& \P\left( \sum_{r=1}^{l-1} \tau_r< U_\lambda< \sum_{r=1}^l \tau_r\right)+ q \P\left( \sum_{r=1}^l \tau_r< U_\lambda< \sum_{r=1}^{l+1}\tau_r\right)\\
&&+ q^2 \P\left( \sum_{r=1}^{l+1} \tau_r< U_\lambda< \sum_{r=1}^{l+2}\tau_r\right)+ \cdots \\
&=& \sum_{i\geq 0} q^i \P\left( \sum_{r=1}^{l+i-1} \tau_r< U_\lambda< \sum_{r=1}^{l+i}\tau_r\right)\\
&=& \sum_{i\geq 0} q^i  \P\left( \sum_{r=1}^{l+i}\tau_r>U_\lambda\right) -  \sum_{i\geq 0} q^i  \P\left( \sum_{r=1}^{l+i-1}\tau_r>U_\lambda\right)\\
&=& \sum_{i\geq 0} q^i (1-\bar p^{l+i}) -  \sum_{i\geq 0} q^i (1-\bar p^{l+i-1})\\
&=&\bar p^{l-1} \sum_{i\geq 0} (q\bar p)^i- \bar p^l \sum_{i\geq 0} (q\bar p)^i\\
&=&\frac{\bar p p^{l-1}}{1-qp}.
\end{eqnarray*}
For $s=1$ and $l=0$, then 
\[
\P({\mathcal E}_1\,|\, {\mathcal G}_0) = 1- \P({\mathcal E}_0\,|\, {\mathcal G}_0) = \frac{q p}{1-q\bar p}.
\]
We then obtain

\begin{eqnarray}
\P({\mathcal E}_1)&=& \P({\mathcal E}_1\,|\, {\mathcal G}_0)\P({\mathcal G}_0) + \sum_{l \geq 1} \P({\mathcal E}_1\,|\, {\mathcal G}_l)\P({\mathcal G}_l)\nonumber\\
&=&\frac{q p}{1-q\bar p}(1-q\overline{Q}(0)) +\sum_{l\geq 1}   \frac{p \bar p^{l-1}}{1-q\bar p}  qQ(l)\nonumber\\
&=&\frac{qp}{1-q\bar p}\left(1-q\overline{Q}(0) +\frac{\mathcal{G}(Q)(\bar p)}{\bar p}-\frac{Q(0)}{\bar p}    \right).
\label{eq:E1}
\end{eqnarray}

Assume now that $s\geq 2$. There can be two interferences or more only if Willie interferes with jobs in the system at time $T-$. Therefore,
\[
\P({\mathcal E}_s\,|\, {\mathcal G}_l)=0 \quad \hbox{if  } 0\leq l<s
\] 
and, for $l\geq s$,
\begin{eqnarray*}
\P({\mathcal E}_s\,|\, {\mathcal G}_l)&=& \P\left(  \sum_{r=1}^{l-s} \tau_r<U_\lambda< \sum_{r=1}^{l-s+1}\tau_r \right)\\
&=& \P\left( U_\lambda< \sum_{r=1}^{l-s+1}\tau_r \right)- \P\left( U_\lambda< \sum_{r=1}^{l-s}\tau_r \right)\\
&=& 1-\bar p^{l-s+1} - (1-\bar p^{l-s})= \bar p^{l-s} p.
\end{eqnarray*}
We then obtain
\begin{equation}
\label{eq:Es}
\P({\mathcal E}_s)= \sum_{l\geq s}\P({\mathcal E}_s\,|\, {\mathcal G}_l)\P({\mathcal G}_l)= \frac{q p}{\bar p^s}\sum_{l=s}^\infty  Q(l)\bar p^l,  \quad \forall s\geq 2.
\end{equation}
Under $H_1$, the cdf of $Y$, Willie job reconstructed service time, is given by 
\begin{eqnarray*}
\lefteqn{\P(Y<x)}\\
&=&\P(Y<x\,|\, J=1)\pi_J + \P(Y<x\,|\, J\not=1)\bar \pi_J\\
&=&\P(Y<x\,|\, J=1,{\mathcal E}_0) \P({\mathcal E}_0)\pi_J +\sum_{s\geq 1}\P(Y<x\,|\, J=1,{\mathcal E}_s) \P({\mathcal E}_s)\bar \pi_J \\
&&+\P(Y<x\,|\, J\not=1,{\mathcal F}_0) \P({\mathcal F}_0)\bar \pi_J +\sum_{s\geq 1}\P(Y<x\,|\, J\not=1,{\mathcal F}_s) \P({\mathcal F}_s)\bar \pi_J)\\
&=& \P(\sigma<x)\left(\P({\mathcal E}_0)\pi_J  + \P({\mathcal F}_0)\bar \pi_J\right)
+\sum_{s\geq 1}\P\left(\sigma+\sum_{r=1}^s \tau_r<x\right)\left( \P({\mathcal E}_s)\pi_J+\P({\mathcal F}_s)\bar \pi_J\right)\\
&=& G_1(x) \left(\P({\mathcal E}_0)\pi_J + \P({\mathcal F}_0)\bar \pi_J\right) +\sum_{s\geq 1}(G_1*h_s)(x) \left( \P({\mathcal E}_s)\pi_J+\P({\mathcal F}_s)\bar \pi_J\right).
\end{eqnarray*}
From  (\ref{eq:F0}), (\ref{eq:Fs}), (\ref{eq:E0}),  (\ref{eq:E1}), and (\ref{eq:Es}) we find 
\begin{eqnarray*}
\P({\mathcal E}_0)\pi_J + \P({\mathcal F}_0)\bar \pi_J &=&\frac{\bar q}{1-q\bar p}(\bar q+ q \mathcal{G}_Q(\bar p))\pi_J +(1-q\overline{Q}(0))\bar \pi_J\\
\P({\mathcal E}_1)\pi_J+\P({\mathcal F}_1)\bar \pi_J &=& \frac{q p}{1-q\bar p}\left(1-q\overline{Q}(0) +\frac{\mathcal{G}_Q(\bar p)-Q(0)}{\bar p}\right)\pi_J \\
&&+ qQ(1)\bar \pi_J\\
\P({\mathcal E}_s)\pi_J+\P({\mathcal F}_s)\bar \pi_J&=&  q p \pi_J \sum_{l\geq s} Q(l)\bar p^{l-s} + q Q(s)\bar \pi_J, \quad \forall s\geq 2.
\end{eqnarray*}
Hence, for $x\geq 0$,
\begin{eqnarray}
\P(y<x)&=&G_1(x)\left(\frac{\bar q}{1-q\bar p}(\bar q+ q\mathcal{G}_Q(\bar p))\pi_J+(1-q\overline{Q}(0))\bar \pi_J\right)
+ q(G_1*h_1)(x)\nonumber\\
&&\times \left(\frac{p}{1-q\bar p}\left(1-q\overline{Q}(0)+\frac{\mathcal{G}_Q(\bar p)-Q(0)}{\bar p}\right)\pi_J + Q(1)\bar \pi_J\right)\nonumber\\
&&
+q \sum_{s\geq 2} (G_1*h_s )(x)\left( p  \pi_J \sum_{l\geq s} Q(l)p^{l -s}+ Q(s)\bar \pi_J\right).
\label{eq:y}
\end{eqnarray}
Dividing both sides of (\ref{eq:y}) by $g_1(x)$  gives (\ref{lem:eq:Y}). 


\section{Proof of Lemma \ref{lem-power-series}}
\label{app:E}
For any mapping $h(q)$, we denote by $h^\prime(q)$ its 1st derivative and by $h^{\prime\prime}(q)$ its 2nd derivative at $q$ when they do exist. 

Define 
\begin{equation}
\label{def:Fq}
F(q)=\E\left[\sqrt{W(q,X)}\right]=\int_0^\infty \sqrt{g_1(x)} \sqrt{g_1(x)W(q,x)} \,\,dx. 
\end{equation}
Notice that, by Cauchy-Schwarz inequality,
\[
0\leq F(q)\leq \left(\int_0^\infty g_1(x) dx \right)^{1/2} \left( \int_0^\infty g_1(x) W(q,x)dx\right)^{1/2}=1,
\]
since $g_1(x) W(q,x)$ is the density of a nonnegative rv, which implies that $\int_0^\infty g_1(x)W(q,x)dx=1$.

Define $f(q,x)= \sqrt{ \Delta_1(q)+\Delta_2 (q)\Phi_1(x)+q\Phi_2(x)}$, so that (cf. (\ref{def:Fq}))
\[
F(q)=\int_0^\infty \mu_1 e^{-\mu_1 x} f(q,x) dx.
\]
For later use, note that
\begin{eqnarray}
\label{d-f}
\frac{df(q,x)}{dq}&=&\frac{ \Delta^\prime_1(q) + \Delta^\prime_2(q)\Phi_1(x)+\Phi_2(x)}{2f(q,x)}\\
\frac{d^2f(q,x)}{dq^2}&=&\frac{\Delta_1^{\prime\prime}(q)+ \Delta_2^{\prime\prime}(q)\Phi_1(x)}{2f(q,x)}
-\frac{ \left( \Delta^\prime_1(q) + \Delta^\prime_2(q)\Phi_1(x)+\Phi_2(x)\right)^2}{4 f(q,x)^3}.
\label{d2-f}
\end{eqnarray}
Assume that  there exist $0<q<\min\left\{1,\frac{1-\rho_1}{\rho_2 B}\right\}$ (recall that the queue is stable when $\rho_1+q\rho_2B<1$ -- see Lemma \ref{lem:NA-NW}) 
and three non-negative mappings $h_i$, $i=0,1,2$, satisfying $\int_0^\infty e^{-\mu_1 x}h_i(x)dx<\infty$, $i=0,1,2$, such that
\begin{enumerate}
\item
for all $x\geq 0$, $q \to f(q,x)$ is continuous  in  $[0,q_1)$;
\item
for all $q\in [0,q_1)$, $x\to  f(q,x)$ is continuous in $[0,\infty)$;
\item
for all $(q,x)\in [0,q_1)\times [0,\infty)$, $|f(q,x)|\leq h_0(x)$;
\item
for all $x\geq 0$, $q \to \frac{d f(q,x)}{dq}$ is continuous  $q[0,q_1)$;
\item
for all $q\in [0,q_1)$, $q \to \frac{df(q,x)}{dq}$  is continuous in $[0,\infty)$;
\item
for all $(q,x)\in [0,q_1)\times [0,\infty)$, $\left|\frac{df(q,x)}{dq} \right |\leq h_1(x)$;
\item
for all $x\geq 0$, $q \to \frac{d^2 f(q,x)}{dq^2}$ is continuous  in $[0,q_1)$;
\item
for all $q\in [0,q_1)$, $q \to \frac{d^2f(q,x)}{dq^2}$  is continuous in $[0,\infty)$;
\item
for all $(q,x)\in [0,q_1)\times [0,\infty)$, $\left|\frac{d^2f(q,x)}{dq^2} \right |\leq h_2(x)$.
\end{enumerate}
Then, by the Dominated Convergence Theorem,  (1)-(9) will ensure that $F(q)$ is  twice differentiable in $[0,q_1)$.\\

When $g_1(x)=\mu_1 e^{-\mu_1 x}$,
\begin{eqnarray}
\frac{g_1* h_s(x)}{g_1(x)}&=&\frac{\mu_2^s}{(s-1)!}\int_0^x \mu_2^s (x-t)^{s-1} e^{-(\mu_1-\mu_2)(x-t)} dt\nonumber\\
&=&\frac{\mu_2^s}{(s-1)!}\int_0^x \mu_2^s u^{s-1} e^{-(\mu_1-\mu_2)u} du,
\label{g1*hs}
\end{eqnarray}
which, for each $s\geq 1$,  is continuous in  $[0,\infty)$. We may then apply the Beppo Levi Monotone Convergence Theorem to $\Phi_2(x)$ (as  both sums in $\Phi_2(x)$ have non-negative terms) 
to get that the mapping $x\to \Phi_2(x)$ is continuous on $[0,\infty)$, and so is the mapping $x\to \Phi_1(x)$.  

On the other hand, it is easily from the definitions 
of $\Delta_1(q)$ and $\Delta_2(q)$ that
\[
x\to \{\Delta_1(q), \Delta^\prime_1(q), \Delta_1^{\prime\prime}(q), \Delta_2(q), \Delta^\prime_2(q), 
\Delta_2^{\prime\prime}(q)\}
\]
 are all continuous mappings in $[0,1]$.
This shows (1) and (2); this will also show (4), (5), (7) and (8) if we can show that  there exists $q_2\in (0,q_1)$ such that $f(x,q)\not =0$ for all $q\in [0,q_2)$, $x\geq 0$.
From the definition of $f(q,x)$, we see that 
\begin{equation}
f(q,x) \geq \Delta_1(q) \geq \bar q \left(\bar q + q {\cal G}_Q(\bar p)\right)\pi_J,
\label{lb:fqx}
\end{equation}
for all $q\in [0,1]$, $x\geq 0$. It is easily seen that the mapping $q\to \zeta(q):=\bar q \left(\bar q + q {\cal G}_Q(\bar p)\right)\pi_J$ is strictly decreasing in $[0,1]$ with
$\zeta(0)=1$ and $\zeta(1)=0$. Pick $\delta\in (0,1)$. The above implies that there exists $q_2\in (0,q_1)$ ) such that $(\bar q + q {\cal G}_Q(\bar p))\pi_J\geq \delta$ for  all
$q\in [0,q_2]$, which in turn implies that
\begin{equation}
\label{lw-fqx}
f(q,x)\geq \delta>0,
\end{equation}
for all $q\in [0,q_2]$, $x\geq 0$. This establishes the validity of  (4), (5), (7) and (8).

We are left with proving (3), (6), and (9). For $q\in [0, q_1)$, $x\geq 0$, we have
\begin{eqnarray}
f(q,x)&\leq &1+\eta_1+\eta_2\Phi_1(x)+\Phi_2(x) :=k_1(x)\label{lower-bound-fqx}\\
\left|\frac{df(q,x)}{dq}\right|&\leq& \frac{1}{2\delta}(\eta_3+\eta_4 \Phi_1(x)+\Phi_2(x)):=k_2(x)\label{lower-bound-dfqx}\\
\left|\frac{d^2f(q,x)}{dq^2}\right|&\leq& \frac{1}{2\delta}(\eta_5+\eta_6 \Phi_1(x))+\frac{1}{4\delta^3}\left(\eta_3+\eta_4 \Phi_1(x)+\Phi_2(x)\right)^2:=k_3(x), \label{lower-bound-d2fqx} 
\end{eqnarray} 
where
\begin{eqnarray*}
&&\eta_1:=\sup_{q\in [0,q_1)}|\Delta_1(q)|, \quad \eta_2:=\sup_{q\in [0,q_1)} |\Delta_1^{\prime}(q)|\nonumber\\
&&\eta_3:=\sup_{q\in [0,q_1)} |\Delta_1^{\prime\prime}(q)|,\quad  \eta_4:=\sup_{q\in [0,q_1)}|\Delta_2(q)|\nonumber\\
&& \eta_5:=\sup_{q\in [0,q_1)} |\Delta_2^{\prime}(q)|, \quad  \eta_6:=\sup_{q\in [0,q_1)} |\Delta_2^{\prime\prime}(q)|.
\end{eqnarray*}
The constants $\eta_i$, $i=1,\ldots,6$ are all finite as the mappings
\[
q\to \{\Delta_1(q),  \Delta_1^{\prime}(q),\Delta_1^{\prime\prime}(q), \Delta_2(q),\Delta_2^{\prime}(q), \Delta_2^{\prime\prime}(q)\}
\]
are all continuous in $[0,q_1)$ from the very definition of $\Delta_1(q)$ and $\Delta_2(q)$ in (\ref{def:Delta1})-(\ref{def:Delta2}).

By Beppo Levi Monotone Convergence Theorem (which applies here as all terms in the sums in (\ref{def:Phi2}) are non-negative as already noticed) we get 
\begin{eqnarray*}
\int_0^\infty g_1(x) \Phi_2(x)dx&=& \sum_{s\geq 2} \sum_{l\geq s} Q(l)p^{l-s} \int_0^\infty g_1*h_s(x)dx\\
&=&  \sum_{s\geq 2} \sum_{l\geq s} Q(l)p^{l-s}<\infty,
\end{eqnarray*}
where the finiteness of $\sum_{s\geq 2} \sum_{l\geq s} Q(l)p^{l-s}$ is shown in  
Lemma \ref{lem:app-F} in Appendix \ref{app:F}. This shows that $\int_0^\infty h_i(x)dx <\infty$, $i=1,2$, where $h_1(x)$ and $h_2(x)$ are defined
in (\ref{lower-bound-fqx})-(\ref{lower-bound-dfqx}), and proves the validity of (3) and (6).

It is shown in Lemma \ref{lem:app-F3} in Appendix \ref{app:F} that when  $\mu_1<2\mu_2$
\[
\int_0^\infty \mu_1^{-\mu_1 x} \left(\eta_4 \Phi_1(x)+\Phi_2(x)\right)^2 dx<\infty,
\]
which implies  together with  the finiteness
of  $\int_0^\infty g_1(x) \Phi_i(x)dx$, $i=1,2$,  that $\int_0^\infty \mu_1 e^{-\mu_1 x} h_3(x)dx$,
where $h_3(x)$ is defined in (\ref{lower-bound-d2fqx}). This proves (9) when $\mu_1<2\mu_2$.

We have therefore shown that there exists $q_1\in (0,q_0)$ such $F(q)$ is twice differentiable  in $[0,q_1)$ when $\mu_1<2\mu_2$. Application of Leibniz's differentiation rule gives
\begin{eqnarray}
F(0)&=&1\nonumber\\
F^{\prime}(0)&=&\frac{1}{2}\int_0^\infty \mu_1 e^{-\mu_1 x} \left(\alpha+\beta\Phi_1(x)+\Phi_2(x)\right)dx
 \label{dF0}\\
F^{\prime\prime}(0)&=&-\frac{1}{4}\int_0^\infty \mu_1 e^{-\mu_1 x}\left(\alpha+\beta\Phi_1(x)+\Phi_2(x)\right)^2 dx\nonumber\\
&&+(p -{\cal G}_Q(\bar p)) \left(p \pi_J - \int_0^\infty \mu_1 e^{-\mu_1 x}\Phi_1(x)dx \right) \nonumber\\
&=& -\frac{1}{4}\int_0^\infty \mu_1 e^{-\mu_1 x}\left(\alpha+\beta\Phi_1(x)+\Phi_2(x)\right)^2 dx
\label{d2F0}
\end{eqnarray}
where 
\begin{eqnarray}
\alpha&:=&-(p+\overline{{\cal G}_Q} (\bar p))\pi_J -\overline{Q}(0)\bar \pi_J\label{def:alpha}\\
\beta&:=&1+\frac{{\cal G}_Q(\bar p)-Q(0)}{\bar p}\label{def:beta}.
\end{eqnarray}
Note that (\ref{d2F0}) holds since $\int_0^\infty \mu_1 e^{-\mu_1 x} \Phi_1(x)dx = p\pi_J \int_0^\infty g_1*h_s (x)dx = p \pi_J$ from
the definition of $\Phi_1(x)$ in (\ref{def:Phi1}), and since $g_1*h_s$ is a pdf on $[0,\infty)$.\\

It is shown in Lemma \ref{lem:app-F-2} in Appendix \ref{app:F} that $F^\prime(0)=0$. Hence, by Taylor's Theorem, 
\begin{equation}
\label{eq:ps}
F(q)= 1 +c_0 q^2 + o(q^2),
\end{equation}
with $c_0:=\frac{1}{2}F^{\prime\prime}(0)<\infty$ when $\mu_1<2\mu_2$, which completes the proof.

\section{Appendix}
\label{app:F}

\begin{lemma}
\label{lem:app-F}
\[
\sum_{s\geq 2}\sum_{l\geq s} Q(l)\bar p^{l-s}=\frac{1-\bar p \mathcal{G}_Q(\bar p)}{p}-  \frac{\mathcal{G}_Q(\bar p)(1+\bar p)-Q(0)}{\bar p}.
\]
\end{lemma}
\begin{proof}
\begin{eqnarray*}
\sum_{s\geq 2}\sum_{l\geq s} Q(l)\bar p^{l-s}&=& \sum_{s\geq 0}\sum_{l\geq s} Q(l)\bar p^{l-s}-  \sum_{l\geq 1} Q(l)\bar p^{l-1} - \sum_{l\geq 0} Q(l)\bar p^l\\
&=&\sum_{s\geq 0}\sum_{l\geq s} Q(l)\bar p^{l-s} - \frac{\mathcal{G}_Q(\bar p)-Q(0)}{\bar p}-\mathcal{G}_Q(\bar p)\\
&=& \sum_{l\geq 0} Q(l)\bar p^l \sum_{s=0}^l \left(\frac{1}{\bar p}\right)^s  - \frac{\mathcal{G}_Q(\bar p)(1+\bar p)-Q(0)}{\bar p}\\
&=&\sum_{l\geq 0} Q(l)\bar p^l  \left(\frac{1-{\bar p}^{-(l+1)}}{1-\bar p^{-1}}\right) -  \frac{\mathcal{G}_Q(\bar p)(1+\bar p)-Q(0)}{\bar p}    \\
&=& -\frac{1}{p} \sum_{l\geq 0} Q(l)\bar p^l \left(1-\bar p^{-1}\right) -  \frac{\mathcal{G}_Q(\bar p)(1+p)-Q(0)}{\bar p}    \\
&=&\frac{1-\bar p \mathcal{G}_Q(\bar p)}{p}-  \frac{\mathcal{G}_Q(\bar p)(1+\bar p)-Q(0)}{\bar p}.
\end{eqnarray*}
\end{proof}

\begin{lemma}
\label{lem:app-F-2}
\[
F^\prime(0)=0,
\]
where $F^\prime(0)$ is given in (\ref{dF0}).
\end{lemma}

\begin{proof}
The definitions of $\Phi_1(x)$ and $\Phi_2(x)$ in (\ref{def:Phi1})-(\ref{def:Phi2}) yield
\begin{equation}
\int_0^\infty g_1(x) \Phi_1(x)dx =p \pi_j \int_0^\infty g_1*h_1 (x) dx =p \pi_j,
\label{eq:Phi1-2}
\end{equation}
with $g_1(x)=\mu_1 e^{-\mu_1 x}$, and 
\begin{eqnarray}
&& \int_0^\infty \mu_1 e^{-\mu_1 x} \Phi_2(x)dx\nonumber\\
&&=p \pi_J \int_0^\infty  
 \sum_{s\geq 2}g_1*h_s (x)\sum_{l\geq s} Q(l)\bar p^{l-s} dx +\bar \pi_J \int_0^\infty \sum_{s\geq 1} g_1*h_s(x) Q(x) dx \nonumber\\
 &&= p \pi_J  \sum_{s\geq 2} \left(\int_0^\infty g_1*h_s (x) dx\right) \sum_{l\geq s} Q(l)\bar p ^{l-s}+\bar \pi_J\sum_{s\geq 1} \int_0^\infty g_1*h_s (x)Q(s) dx 
 \label{eq:sum-integral}\\
 &&= p \pi_J  \sum_{s\geq 2} \sum_{l\geq s} Q(l)\bar p^{l-s} +\bar \pi_J(1-Q(0)),
 \label{eq:Phi2-2}
 \end{eqnarray}
since $\int_0^\infty g_1*h_s(x) dx=1$ and   $\sum_{s\geq 0}Q(s)=1$, 
where the interchange of the integrals and sums  in (\ref{eq:sum-integral}) is justified by the Beppo Levi's Monotone Convergence Theorem.
Thus, by using  (\ref{dF0}), the definition of $\alpha$ and $\beta$ in (\ref{def:alpha})-(\ref{def:beta}), (\ref{eq:Phi1-2}), (\ref{eq:Phi2-2}), and
Lemma \ref{lem:app-F},  we obtain
\begin{eqnarray*}
\lefteqn{2F^{\prime}(0)= -(p +\overline{\mathcal{G}_Q}(\bar p))\pi_J-\overline{Q}(0)\bar \pi_J +\left(1+\frac{\mathcal{G}_Q(\bar p)-Q(0)}{\bar p}\right)p \pi_J} \\
&&+ \pi_J (1-\bar p \mathcal{G}_Q(\bar p))- \frac{\mathcal{G}_Q(\bar p)(1+\bar p)-Q(0)}{\bar p} p \pi_J +\bar \pi_J(1-Q(0))\\
&=&Q(0)\left(\bar \pi_J-\frac{p\bar \pi_J}{\bar p}+\frac{ p\bar \pi_J}{\bar p}-\bar \pi_J\right)
+\mathcal{G}_Q(\bar p)\left(\pi_J+\frac{p \pi_J}{\bar p} -\bar p\pi_J-\frac{1+\bar p}{\bar p}p \pi_J  \right)=0.
\end{eqnarray*}
\end{proof}
\begin{lemma}
\label{lem:app-F3}
Assume that the support of $Q_B$ is finite. Then,
\[
\int_0^\infty \mu_1 e^{-\mu_1 x} \left(\eta_4 \Phi_1(x)+\Phi_2(x)\right)^2 dx<\infty
\] 
if $\mu_1<2\mu_2$.
\end{lemma}
\begin{proof}
Without any loss of generaly, assume that the support of $Q_B$ is in $\{0,1,\ldots,S\}$ with $S\leq 1$.
From (\ref{def:Phi1})-(\ref{def:Phi2}) we obtain
\begin{eqnarray}
\lefteqn{(\eta_1 \Phi_1(x)+\Phi_2(x))^2\leq \max(\eta_4^2, 1) (\Phi_1(x)+\Phi_2(x))^2}\nonumber \\
&\leq & \max(\eta_4^2, 1)\Biggl(\sum_{s=1}^S\frac{g_1*h_s (x)}{g_1(x)}\sum_{l=s}^S Q(l)\bar p^{l-s}
 +  \bar \pi_J \sum_{s=1}^S\frac{ g_1*h_s (x)}{g_1(x)} Q(s) + \pi_J\frac{g_1*h_1(x)}{g_1(x)} \Biggr)^2\nonumber\\
 &\leq &\max(\eta_4^2, 1)
  \left(\sum_{s=1}^S\frac{g_1*h_s(x)}{g_1(x)}\left[Q(s)+ \sum_{l=s}^S Q(l)\bar p^{l-s} \right]\right)^2\nonumber\\
  &\leq & \max(\eta_4^2, 1) (S+1)^2\left( \sum_{s=1}^S\frac{g_1*h_s(x)}{g_1(x)}\right)^2\nonumber\\
  &=&\max(\eta_4^2, 1) (S+1)^2 \sum_{i=1}^S \sum_{j=1}^S \frac{g_1*h_i(x)}{g_1(x)} \frac{g_1*h_j(x)}{g_1(x)},
  \label{ub-phi1phi2}
  \end{eqnarray}
by using the identity $\left(\sum_i a_i\right)^2=\sum_i\sum_j a_i a_j$. 
With   $g_i(x)=\mu_i e^{-\mu_i x}$ for $i=1,2$ it is easily seen that
\begin{eqnarray}
\frac{g_1*h_s(x)}{g_1(x)} &=&\frac{\mu_2^s}{(s-1)!} \int_0^x  t^{s-1} e^{-(\mu_ 2-\mu_1)t}dt\nonumber\\
&=& \left\{ \begin{array}{ll}
\frac{(\mu_1 x)^s}{s!}  &\mbox{if $\mu_1=\mu_2$}\\
 \frac{1}{(s-1)!}\left(\frac{\mu_2}{(\mu_2-\mu_1)}\right)^s-\mu_2^se^{-(\mu_2-\mu_1)x}  &\\
\times \sum_{i=0}^{s-1} \frac{1}{(\mu_2-\mu_1)^{i+1}} \frac{x^{s-1-i}}{(s-1-i)!} & \mbox{if $\mu_1\not=\mu_2$.}
                \end{array}
         \right.
         \label{g1-convol-prop2}
\end{eqnarray}
A glance at (\ref{ub-phi1phi2}) and (\ref{g1-convol-prop2}) shows that the r.h.s. of (\ref{ub-phi1phi2}) is a finite sum of terms of the form $x^k$, $e^{-(\mu_2-\mu_1)x} x^k$, and
   $e^{-2(\mu_2-\mu_1)x} x^k$.  Now, since integrals $$\int_0^\infty \mu_1 e^{-\mu_1x} x^k dx = \frac{k!}{\mu_1^k}$$ and  $$\int_0^\infty \mu_1 e^{-\mu_1x} e^{-(\mu_2 -\mu_1)x} x^k dx = 
  \frac{\mu_1 k!}{\mu_2^{k+1}}$$ are finite,  and the integral 
  \[
  \int_0^\infty \mu_1 e^{-\mu_1x} e^{-2(\mu_2 -\mu_1)x} x^k dx=\frac{\mu_1 k!}{(2\mu_2-\mu_1)^{k+1}}
 \]
is finite for $\mu_1<2\mu_2$,  the lemma is proved. 
\end{proof}

\end{document}